\documentclass[reqno,11pt]{amsart}
\usepackage{amssymb}
\usepackage{mathrsfs}
\usepackage{txfonts}
\usepackage{amsfonts}
\usepackage{amstext}
\usepackage{amssymb}
\usepackage{amsmath}
\usepackage{graphicx}
\usepackage{hyperref}
\usepackage{url}
\usepackage{tikz}
\usepackage{xparse}

\makeatletter
\def\sim@x@scale{.15}
\def\sim@y@scale{.05}
\def\sim@y@thick{.02}
\newsavebox\sim@upper
\newsavebox\sim@lower
\NewDocumentCommand{\xSim}{ O{} m }{\TextOrMath{\PackageError{TEST}{`\string\xSim` is valid in math mode only.}{}}{
\sbox\sim@upper{$\scriptsize #2$}
\sbox\sim@lower{$\scriptsize #1$}
\pgfmathparse{min(max(\wd\sim@upper/1em, \wd\sim@lower/1em, 1.0), 1.5)}
\edef\sim@ratio{\pgfmathresult}
\def\sim@x {\sim@x@scale * \sim@ratio}
\def\sim@y {\sim@y@scale * \sim@ratio}
\def\sim@@y{\sim@y@thick * \sim@ratio}
\pgfmathparse{floor(max(\wd\sim@upper/1em, \wd\sim@lower/1em)) + 1}
\edef\sim@wd{\pgfmathresult em}
\mathrel{
\begin{tikzpicture}[baseline=-.7ex]
\filldraw[line width=.2pt] 
(0, 0)
.. controls +(\sim@x, \sim@y+\sim@@y) and +(-\sim@x, -\sim@y) .. 
+(\sim@wd, 0) 
node[midway, above] {\usebox\sim@upper} 
node[midway, below] {\usebox\sim@lower}
.. controls +(-\sim@x, -\sim@y-\sim@@y) and +(\sim@x, \sim@y) .. 
(0, 0);
\end{tikzpicture}
}
}}
\makeatother
\hypersetup{
    colorlinks  = true,
    citecolor  = blue,
    linkcolor  = blue,
    urlcolor   = blue
}
\allowdisplaybreaks
\newtheorem{theorem}{Theorem}[section]
\newtheorem{proposition}[theorem]{Proposition}
\newtheorem{lemma}[theorem]{Lemma}
\newtheorem{corollary}[theorem]{Corollary}
\newtheorem{remark}[theorem]{Remark}

\newtheorem{definition}[theorem]{Definition}

\newtheorem{example}[theorem]{Example}

\renewcommand{\theequation}{\thesection.\arabic{equation}}
\newenvironment{notation}{\smallskip{\sc Notation.}\rm}{\smallskip}
\newenvironment{acknowledgement}{\smallskip{\sc Acknowledgement.}\rm}{\smallskip}
\DeclareMathOperator*{\supp}{supp}
\DeclareMathOperator*{\einf}{einf}
\DeclareMathOperator*{\esup}{esup}

\DeclareMathOperator*{\cutoff}{cutoff}
\DeclareMathOperator*{\diam}{diam}

\numberwithin{equation}{section}
\newcounter{counterConstant}

\makeatletter

\newcommand{\Rmnum}[1]{\expandafter\@slowromancap\romannumeral #1@}
\@namedef{subjclassname@2020}{\textup{2020} Mathematics Subject Classification}
\makeatother
\input{tcilatex}

\makeatletter
\@namedef{subjclassname@2020}{%
  \textup{2020} Mathematics Subject Classification}
\makeatother

\def\qed{\hfill$\square$\par}

\def\dsum{\displaystyle\sum}

\def\tbigcup{\mathop{\textstyle \bigcup }}
\def\tbigcap{\mathop{\textstyle \bigcap }}

\def\Xint#1{\mathchoice
{\XXint\displaystyle\textstyle{#1}}%
{\XXint\textstyle\scriptstyle{#1}}%
{\XXint\scriptstyle\scriptscriptstyle{#1}}%
{\XXint\scriptscriptstyle\scriptscriptstyle{#1}}%
\!\int}
\def\XXint#1#2#3{{\setbox0=\hbox{$#1{#2#3}{\int}$ }
\vcenter{\hbox{$#2#3$ }}\kern-.6\wd0}}

\def\oint{\Xint-}

\def\fint{\Xint-}

\def\FRAME#1#2#3#4#5#6#7#8
{
 \begin{figure}[H]
 \begin{center}
 \includegraphics[height=#3]{#7}
 \caption{#5}
 \label{#6}
 \end{center}
 \end{figure}
}

\def\enddoc{\end{document}}

\begin{document}
\title[Heat Kernels]{Inhomogeneous Scaling Function and Heat Kernel Estimates on Fractals Satisfying Some Resistance Conditions}
\author[Chang]{Diwen Chang}
\address{School of Mathematics and Statistics,
	Beijing Technology and Business University,
	Beijing, 102488, China.}
\email{20250705@btbu.edu.cn}
\author[Liu]{Guanhua Liu}
\address{Center for Applied Mathematics and KL-AAGDM, Tianjin University, Tianjin, 300072, China.}
\email{liu\_gh@tju.edu.cn}
\date{\today }

\begin{abstract}
In this paper, we focus on strongly local regular Dirichlet forms, especially those satisfying Morrey-type inequalities. We prove the equivalence between resistance estimates and heat kernel estimates in this case. Self-similar forms on fractals serve as a major application, where we construct a spatially inhomogeneous scaling function and characterize all the doubling self-similar measures. Further, on some special examples, the resistance conditions are reduced to some geometric conditions, on which a complete theory on self-similar Dirichlet spaces is established therein. In particular, we construct a concrete example on rotated triangle fractals, where the optimal heat kernel estimate is not related at all to the lower scaling exponent.
\end{abstract}

\subjclass[2020]{Primary: 35K08; Secondary: 28A80, 31E05.}

\keywords{Heat kernel, scaling function, Dirichlet form, self-similar fractal}

\maketitle
\tableofcontents

\section{Background and Motivation}
The heat kernel, defined as the fundamental solution to the heat equation, serves as a fundamental tool for analyzing geometric and analytic properties of an underlying space. On fractals, the study of heat kernel estimates originated with the groundbreaking work of \cite{BarlowBass.1989.AIHPPS225,BarlowBass.1992.PTRF307,BarlowPerkins.1988.PTRF543}, where constructed are Brownian motions on the Sierpi\'nski gasket and the Sierpi\'nski carpet, and derived are their corresponding heat kernel estimates. Over the past four decades, a rich and mature theory has been developed, for example, in \cite{GrigoryanHu.2008.IM81,GrigoryanHuLau.2014.TAMS6397,GrigoryanTelcs.2012.AP1212}, largely centered on the \emph{homogeneous} space-time scaling relation
\[t = r^{\beta},\]
where $\beta > 1$ is a constant known as the \emph{walk dimension}. In Euclidean spaces, $\beta = 2$ corresponds to the classical Brownian motion, while on fractals $\beta \ge 2$ typically, corresponds to some diffusion process. This framework leads to a comprehensive characterization of two-sided sub-Gaussian heat kernel bounds in terms of geometric conditions such as volume doubling, the Faber-Krahn inequality, and resistance estimates. Importantly, the theory remains valid for more general scaling functions $\Phi(r)$ satisfying appropriate regularity and comparability conditions, with $t = \Phi(r)$ replacing $t = r^{\beta}$.

However, many natural inhomogeneous media -- such as materials with varying porosity, fractured rocks, or hierarchical structures with distinct local geometries -- exhibit diffusion whose propagation rate depends significantly on location. In such settings, the scaling
\[t = W(x, r)\]
takes a more general form, where the function $W$ depends \emph{essentially} on the spatial variable $x$. Telcs \cite{Telcs2006} provided an illustrative example on graphs. Meanwhile, there have been scarce rigorously investigated examples on continuous fractals, in particular, on fractals equipped with strongly local Dirichlet forms, which correspond to diffusion processes in continuous settings. Known results under $W(x,r)$-scaling focus primarily on non-local Dirichlet forms \cite{GrigoryanHuHu.2022.TPGcap, GrigoryanHuHu.2022.TPLE,GrigoryanHuHu.2023.TPDUE,GrigoryanHuHu.2024.TPUEq}, leaving a significant gap in the strongly local case. The core difficulty lies in the fact that established heat kernel estimates rely heavily on the homogeneity of $r^{\beta}$ or $\Phi(r)$. Spatial inhomogeneity disrupts the standard covering, chaining and iteration arguments, for which reason a thorough reworking of the underlying analytic machinery is required now.

This paper aims to bridge this gap. We achieve this by developing a general theoretical framework for the strongly local setting and, concurrently, providing a concrete, verifiable fractal model. Together, these contributions demonstrate both the feasibility and the analytic richness of inhomogeneous scalings on continuous fractal spaces. 

Firstly, this work provides, for the first time, a complete theory of sub-Gaussian heat-kernel estimates, along with their equivalent characterizations, for strongly local Dirichlet forms in the setting of $t = W(x, r)$ scaling. Concretely, we extend the classical equivalence theorems for sub-Gaussian heat kernel estimates to the case where a general scaling function $W(x,r)$ depends on the spatial variable. To accommodate the spatial inhomogeneity, we reformulate the key conditions, that is, the Faber-Krahn inequality $(\mathrm{FK})$, the survival estimate $(\mathrm{S})$, the resistance estimate $(\mathrm{R})$, and the corresponding Morrey-type inequality $(\mathrm{MI})$. Within this setting we prove the equivalence
\[
(\mathrm{FK}) + (\mathrm{S}) + (\mathrm{sloc})\; \Leftrightarrow \; (\mathrm{UE}_{\mathrm{exp}}) + (\mathrm{C}),
\]
and, under an extra assumption of the chain condition $(\mathrm{CH})$, a better equivalence
\[
(\mathrm{MI}) + (\mathrm{S}) + (\mathrm{sloc}) \; \Leftrightarrow \; (\mathrm{UE}_{\mathrm{exp}}) + (\mathrm{LE}_{\mathrm{exp}}) + (\mathrm{C}).
\]
The proofs, presented in a self-contained manner in the appendices, systematically overcome the technical challenges introduced by the $x$-dependence of $W$, providing a ready-to-use toolkit to establish heat kernel bounds in inhomogeneous settings.

Secondly, as a concrete realization of the general theory, we show for the first time that the self-similar Dirichlet forms on the rotated triangle family $K_{\lambda}$ (introduced by Barlow \cite{Barlow.1998.1}) admit a scaling function $W(x,r)$ that \emph{essentially} depends on the spatial variable $x$. We then derive precise heat kernel estimates for these forms. To achieve this, we present a general method to construct scaling functions from the self-similar coefficients. This construction decides space-time scalings on fractals with suitable geometric conditions. Through a detailed geometry analysis of $K_{\lambda}$, we confirm all the required assumptions, including the connectivity property $(\mathrm{CP}_{\rho})$ and the average condition $(\mathrm{AV}_{\rho})$. For the established self-similar resistance forms on $K_{\lambda}$ by Cao \cite{Cao.2023.AG}, we then derive \emph{optimal two-sided heat kernel estimate} of the form
\[
p_t(x,y) \asymp \frac{C}{V\bigl(x,W^{-1}(x,t)\bigr)}\,
\exp\biggl\{-c\Bigl(\frac{W(x,d(x,y))}{t}\Bigr)^{\frac{1}{\beta_1-1}}\biggr\},
\]
which is valid whenever the associated exponents satisfy $\beta_1 \ge \beta_4$. Here $\beta_1,\beta_4$ are the scaling exponents arising from the self-similar data; the function $W$ essentially varies with $x$ when $\beta_1 \neq \beta_4$. This provides a concrete example of spatially inhomogeneous diffusions on fractals showing how the abstract $W(x,r)$-framework can be applied to an explicit model.

The main innovations of this work lie in the following aspects:
\begin{itemize}
	\item \textbf{A complete theory for $W(x,r)$-scaling.} We establish the first systematic framework of sub-Gaussian heat kernel estimates and their equivalent characterizations for strongly local Dirichlet forms under spatially dependent scaling.
	
	\item \textbf{A general construction of scaling functions.} We provide a method to explicitly build the scaling function $W(x,r)$ on fractals using self-similar data, enabling the realization of spatially dependent scaling on fractals with specific geometric conditions.
	
	\item \textbf{A concrete fractal exemplar.} We provide the first verifiable example -- the rotated triangle family $K_{\lambda}$ -- where the scaling function $W(x,r)$ is explicitly constructed and shown to depend essentially on $x$, with all supporting geometric conditions rigorously checked.
	
	\item \textbf{Sharp bounds in highly inhomogeneous regimes.} We derive optimal two-sided estimates even when the lower scaling exponent $\beta_*$ is less than 1, revealing how the paths that diffusion takes place along are decided by local inhomogeneity.
\end{itemize}

The discussions here guide to some similar phenomenon on other geometrically complicated spaces, for which details are left to future researches.

\emph{Structure of the paper.} The rest of this paper is organized as follows. Section \ref{mainresults} states the main equivalent characterizations of standard heat kernel estimates under the scaling relation $t=W(x,r)$. Section \ref{SSSWay} shows the general way to construct scaling functions on self-similar sets based on self-similar coefficients. In Section \ref{Kl} we 
combine our general theories in Sections \ref{mainresults}-\ref{SSSWay} to a particular fractal $K_{\lambda}$, explicitly show an inhomogeneous scaling function $W(x,r)$ and corresponding sharp two-sided heat kernel bounds.

\begin{notation}
The letters $c,c_{1},C,C_{0},C_{1},C_{0,1}$, etc., denote positive constants whose values are unimportant and may differ at different occurrences. The relation $f\asymp g$ between two non-negative functions $f,g$ means that there is a constant $C\geq 1$ such that $C^{-1}f\leq g\leq Cf$ for a specified range of the variables.
\end{notation}

\section{Terminologies and Heat Kernel Estimates}\label{mainresults}

In this section, we list basic assumptions on the metric measure structure, and then investigate necessary and sufficient conditions for the sub-Gaussian upper and lower bounds of heat kernels on general metric measure spaces. Based on these results, we particularly focus on the case where the related effective resistance exists, which is commonly observed on fractal spaces. The relation between resistance estimates and corresponding two-sided heat kernel estimates will be systematically examined.

\subsection{Conditions on metric measure spaces}

Let $(M,d)$ be a locally compact separable metric space and let $\mu $ be a
Radon measure on $M$ with full support. Denote metric balls in $(M,d)$ by 
\begin{equation*}
B(x,r):=\{y\in M:d(x,y)<r\}\quad\text{for\quad any }x\in M\quad \text{and}\quad
r>0.
\end{equation*}%
For every ball $B:=B(x,r)$ and every real number $\lambda >0$, set 
\begin{equation*}
V(x,r):=\mu (B(x,r))\quad \text{and}\quad \lambda B:=B(x,\lambda r).
\end{equation*}

Assume that all metric balls in $M$ are precompact. Denote by $\bar{R}$ the
diameter of $(M,d)$, that is, 
\begin{equation*}
\bar{R}:=\sup_{x,y\in M}d(x,y).
\end{equation*}

We always assume that $M$ contains at least two points, or equivalently, $\bar{R}>0$.

Recall that a measure $\mu $ satisfies the \emph{volume doubling} property $(\mathrm{VD})$ (for short, $\mu$ is doubling), if there exists a
constant $C\geq 1$ such that, for all $x\in M$ and all $r>0$, 
\begin{equation*}
V(x,2r)\leq CV(x,r).
\end{equation*}%
Condition $(\mathrm{VD})$ implies that $0<V(x,r)<\infty $ for all $x\in M$
and all $r>0$. Note that $V(x,0)=0$ for all $x\in M$.

It is known that $(\mathrm{VD})$ implies (and actually, is equivalent to) the following condition: there exist two positive constants $C,\alpha ^{\ast }$ such
that, for all $x,y\in M$ and all $0<r\leq R<\infty $, 
\begin{equation*}
\frac{V(x,R)}{V(y,r)}\leq C\left( \frac{d(x,y)+R}{r}\right) ^{\alpha ^{\ast
}}.
\end{equation*}

Recall that a measure $\mu $ satisfies the \emph{reverse volume doubling} property $(\mathrm{RVD})$, if there exist two positive constants $C,\alpha_{\ast }$ such that, for all $x\in M$ and all $0<r\leq R<\bar{R}$, 
\begin{equation*}
\frac{V(x,R)}{V(x,r)}\geq C^{-1}\left( \frac{R}{r}\right) ^{\alpha _{\ast }}.
\end{equation*}
Here are some observations:

\begin{enumerate}
\item[i)] If $\mu $ satisfies both $(\mathrm{VD})$ and $(\mathrm{RVD})$,
then $\alpha ^{\ast }\geq \alpha _{\ast }$, and there exists a positive
constant $C$ such that, for all $0<r\leq R<\bar{R}$ and for all $x,y\in M$
with $d(x,y)\leq R$, 
\begin{equation}
C^{-1}\left( \frac{R}{r}\right) ^{\alpha _{\ast }}\leq \frac{V(x,R)}{V(y,r)}%
\leq C\left( \frac{R}{r}\right) ^{\alpha ^{\ast }}.  \label{muscal0}
\end{equation}

\item[ii)] If $(\mathrm{VD})$ holds, then $\bar{R}<\infty $ implies $\mu
(M)<\infty $. On the other hand, if $(\mathrm{RVD})$ holds, then $\mu
(M)<\infty $ implies $\bar{R}<\infty .$
\end{enumerate}

Recall that a function $W: M\times[0,\infty]\to[0,\infty]$ is called a \emph{%
scaling function} (see \cite{GrigoryanHuHu.2022.TPLE,GrigoryanHuHu.2022.TPGcap}), if

\begin{enumerate}
\item[i)] for any $x\in M$, the function $W(x,\cdot)$ is strictly
increasing, $W(x,0)=0$ and $W(x,\infty)=\infty $;

\item[ii)] there exist three positive constants $C_{W},\beta _{\ast }<\beta^{\ast }$ such that for all $0<r\leq R<\infty $ and all $x,y\in M$ with $d(x,y)\leq R$, 
\begin{equation}
C_{W}^{-1}\left( \frac{R}{r}\right) ^{\beta _{\ast }}\leq \frac{W(x,R)}{%
W(y,r)}\leq C_{W}\left( \frac{R}{r}\right) ^{\beta ^{\ast }}.  \label{scpr}
\end{equation}
\end{enumerate}
Replacing $W(x,r)$, if necessary, by $\displaystyle \int_{0}^{r}\dfrac{W(x,s)}{s}ds$ (which, according to (\ref{scpr}), is always comparable with $W(x,r)$)
we assume that $W$ is continuous. It follows that, for
all $x\in M$ and all $0<r\leq R<\infty $, 
\begin{equation}
C^{-1}\left( \frac{R}{r}\right) ^{1/\beta ^{\ast }}\leq \frac{W^{-1}(x,R)}{%
W^{-1}(x,r)}\leq C\left( \frac{R}{r}\right) ^{1/\beta _{\ast }},
\label{WinvScal}
\end{equation}%
where $W^{-1}(x,\cdot)$ is the inverse function of $W(x,\cdot)$ for any $%
x\in M$.

We say that the number $\beta_{\ast}$ (resp.\ $\beta^\ast$) is the \emph{%
sub-scaling} (resp.\ \emph{super-scaling}) \emph{exponent} of $W$. Note that
if (\ref{scpr}) holds for some $0<\beta_\ast\leq\beta^\ast$, then it also
holds for any exponents $\hat{\beta}_\ast$, $\hat{\beta}^\ast$ whenever $0<%
\hat{\beta}_\ast\leq\beta_\ast$ and $\beta ^\ast\leq\hat{\beta}^\ast<\infty$%
. Sometimes one hopes to decide the sub-scaling and super-scaling exponents
appearing in (\ref{scpr}) as good as possible, for example, when yielding
off-diagonal estimates (see Definitions \ref{dUE}, \ref{dLE}, Proposition \ref{thm:main1} and Theorem \ref{thm:main2}).

For convenience, for any metric ball $B=B(x,r)$, we write 
\begin{equation*}
W(B):=W(x,r).
\end{equation*}

Note that in some metric spaces a ball as a subset of $M$ may have different
centers and radii, that is, it may be possible that $B(x_{1},r_{1})=B(x_{2},r_{2})$ whereas $x_{1}\neq x_{2}$ or $r_{1}\neq r_{2}$%
. To avoid ambiguities in the notation $\lambda B,W(B)$ and other similar
notations, we always identify a ball as a pair of center and radius rather
than as a subset of $M$.

Sometimes we may require $(M,d)$ to satisfy the \emph{chain condition},
denoted by $(\mathrm{CH})$. This means there exists a positive constant $C_{%
\mathrm{CH}}$ such that, for any two distinct points $x,y\in K$ and every
integer $n\geq 1$, there exists a chain of points $\{x_{i}\}_{i=0}^{n}$ in $K $ such that $x_{0}=x$, $x_{n}=y$, and 
\begin{equation}
d(x_{i},x_{i+1})\leq \frac{C_{\mathrm{CH}}}{n}d(x,y)\quad \text{for all}\quad 0\leq i<n\text{.}  \label{CHC}
\end{equation}

\subsection{Conditions on Dirichlet forms}
Let $(\mathcal{E},\mathcal{F})$ be a regular Dirichlet form on $L^{2}(M,\mu) $. Then, it admits a unique \emph{Beurling-Deny decomposition} (cf.\ \cite[Theorems 3.2.1 and 4.5.2]{FukushimaOshimaTakeda.2011.489}): 
\begin{equation*}
\mathcal{E}(u,v)=\mathcal{E}^{(L)}(u,v)+\mathcal{E}^{(J)}(u,v)+\mathcal{E}^{(K)}(u,v),
\end{equation*}
where $\mathcal{E}^{(L)}$ is the \emph{strongly local part}, so that for all $u,v\in \mathcal{F}\cap C_{0}(M)$ such that $u$ equals to some constant $c$ on a neighborhood of $\supp(v)$, 
\begin{equation*}
\mathcal{E}^{(L)}(u,v)=0;
\end{equation*}%
$\mathcal{E}^{(J)}$ is the \emph{jump part} associated with a unique
symmetric Radon measure $j$ on $M\times M\setminus \mathrm{diag}$: 
\begin{equation}
\mathcal{E}^{(J)}(u,v)=\int_{M\times M\setminus \mathrm{diag}%
}(u(x)-u(y))(v(x)-v(y))dj(x,y);  \label{EJ}
\end{equation}%
and finally, $\mathcal{E}^{(K)}$ is the \emph{killing part} associated with
a unique Radon measure $k$ on $M$: 
\begin{equation*}
\mathcal{E}^{(K)}(u,v)=\int_{M}u(x)v(x)dk(x).
\end{equation*}

A regular Dirichlet form $(\mathcal{E},\mathcal{F})$ is said to satisfy condition $(\mathrm{sloc})$ if it is strongly local, i.e., $\mathcal{E}^{(J)}\equiv\mathcal{E}^{(K)}\equiv 0$; it is said to be local, if $\mathcal{E}^{(J)}\equiv 0$.

Recall that $(\mathcal{E},\mathcal{F})$ is \emph{conservative}, denoted by $(\mathrm{C})$, if $P_t1=1$ for all $t>0$, where $P_t$ is the heat semigroup of the form $(\mathcal{E},\mathcal{F})$; it is \emph{parabolic}, denoted by $(\mathrm{Para})$, if for any compact subset $K\subset M$,
\[\mathrm{Cap}(K):=\inf\{\mathcal{E}(u):u\in\mathcal{F}\cap C_0(M), u|_K\geq 1\}=0.\]
Recall by \cite[Lemma 6.4]{GrigoryanHuLau.2014.TAMS6397} that $(\mathrm{Para})\Rightarrow(\mathrm{C})$.

For any open subset $\Omega \subset M$, let $\mathcal{F}(\Omega)$ be the
closure of $\mathcal{F}\cap C_{0}(\Omega)$ under $\mathcal{E}_{1}=\mathcal{E}+\|\cdot\|_{L^2}$. By \cite%
[Theorem 4.4.3]{FukushimaOshimaTakeda.2011.489}, if $(\mathcal{E},\mathcal{F}%
)$ is regular, then $(\mathcal{E},\mathcal{F}(\Omega))$ is also a regular
Dirichlet form on $L^{2}(\Omega ,\mu)$.

Now we fix a scaling function $W$. Recall that the Faber-Krahn inequality $(\mathrm{FK}^{W})$ holds, if there
exist constants $\varepsilon \in (0,1)$ and $C,\kappa >0$ such that, for
every ball $B:=B(x_{0},r)$ with $x_{0}\in M,0<r<\varepsilon \bar{R}$ and for
every non-empty open subset $U\subset B$, 
\begin{equation*}
\lambda _{1}(U):=\inf_{u\in \mathcal{F}(U)\setminus \{0\}}\frac{\mathcal{E}
(u)}{\Vert u\Vert _{2}^{2}}\geq \frac{C^{-1}}{W(B)}\left( \frac{\mu (B)}{\mu
(U)}\right) ^{\kappa }.
\end{equation*}

Let $\{P_{t}^{\Omega }\}_{t>0}$ be the heat semigroup of $(\mathcal{E},\mathcal{F}(\Omega))$.

Recall that condition $(\mathrm{S}^{W})$ (the \emph{survival estimate})
holds, if there exist three constants $\varepsilon ,\delta ,\delta _{S}\in
(0,1)$ such that for every ball $B$ of radius $r\in (0,\bar{R})$ and all $%
0<t\leq \delta W(B)$, 
\begin{equation*}
\einf_{\delta _{S}B}P_{t}^{B}1_{B}\geq \varepsilon .
\end{equation*}

Recall that condition $(\mathrm{S}_{+}^{W})$ (the \emph{strong survival
estimate}) holds, if there exist three constants $c,C>0,\delta _{S}\in (0,1)$
such that for every ball $B$ of radius $r\in (0,\bar{R})$ and all $t>0$, 
\begin{equation*}
\einf_{\delta _{S}B}P_{t}^{B}1_{B}\geq c-\frac{Ct}{W(B)}.
\end{equation*}

Clearly, $(\mathrm{S}_{+}^{W})\Rightarrow (\mathrm{S}^{W})$. Further, $(\mathrm{S}^{W})$ implies (by \cite[Theorem 14.1]{GrigoryanHuHu.2022.TPGcap}) the following \emph{capacity condition} (denoted by $(\mathrm{cap}^W)$): for
any $0<\lambda<1$, there exists a constant $C=C(\lambda)>0$ such that for
every ball $B=B(x_0,r)$ with $x_0\in M$ and $0<r<\bar{R}$, there exists $%
\phi\in\cutoff(\lambda B,B)$ such that 
\begin{equation*}
	\mathcal{E}(\phi)\le\frac{C\mu(B)}{W(B)}.
\end{equation*}

\begin{definition}[Tail estimate of heat semigroup]
We say that condition $(\mathrm{T}^W_{\exp})$ holds, if there exist two constants $C,c>0$ such that, for every ball $B:=B(x_{0},r)\subset M$ with $x_{0}\in M$, $r\in (0,\bar{R})$, and for all $t>0$,
\begin{equation}
	\esup_{\frac{1}{2}B}P_{t}1_{B^{c}}\leq C\exp \left( -c\left( \frac{W(x_{0},r)}{t}\right) ^{\frac{1}{\beta ^{\ast }-1}}\right),\label{Tail}
\end{equation}
where $\beta^\ast$ is given in (\ref{scpr}).
\end{definition}

Finally, we introduce the definition of a Morrey-type inequality with respect to an arbitrary function $F: M \times (0, \infty) \to (0, \infty)$ (not necessarily a scaling function):

\begin{definition}[Morrey-type inequality]
We say that condition $(\mathrm{MI}^{F})$ holds, if there exist two positive constants $C,\iota $ such that 
\begin{equation}\label{MIe}
|u(x)-u(y)|^{2}\leq CF(x,d(x,y))\mathcal{E}(u)
\end{equation}
for all $u\in \mathcal{F}\cap C_{0}(M)$ and all $x,y\in M$ with $d(x,y)<\iota \bar{R}$.
\end{definition}

Here we list an important property of $(\mathrm{MI}^{F})$, which is a special version of \cite[Theorem 1.6]{Murugan2020} (but it works only for $\beta_\ast$, instead of the original for $\beta^\ast$, due to the inhomogenuity of $W$). To begin with,
we recall that given $x,y\in M$ and $\varepsilon>0$, an $\varepsilon$-chain
from $x$ to $y$ is a sequence of points $z_0=x,\dotsc,z_N=y$ with $d(z_i,z_{i+1})<\varepsilon$ for all $0\le i<N$. Let $N_\varepsilon(x,y)$ be
the minimal integer $N\ge 1$ such that there exists an $\varepsilon$-chain
from $x$ to $y$ with length $N$, and define 
\begin{equation*}
	d_\varepsilon(x,y)=\inf\left\{\sum\limits_{i=0}^{N-1}d(z_i,z_{i+1}):\
	\{z_i\}_{i=0}^N\ \text{is an }\varepsilon\text{-chain from }x\text{ to }%
	y\right\}.
\end{equation*}
Clearly $d_\varepsilon(x,y)\ge d(x,y)$. According to \cite[Lemma 6.3]%
{GrigoryanTelcs.2012.AP1212}, there is always 
\begin{equation}  \label{rgNe}
	\left\lceil\frac{d_\varepsilon(x,y)}{\varepsilon}\right\rceil\le
	N_\varepsilon(x,y)\le 9\left\lceil\frac{d_\varepsilon(x,y)}{\varepsilon}%
	\right\rceil,
\end{equation}
where $\lceil x\rceil$ is the smallest integer $n$ such that $n\ge x$.

\begin{proposition}
	\label{bt2+} Assume that $(\mathrm{VD})$, $(\mathrm{sloc})$, $(\mathrm{MI}^F)$
	and $(\mathrm{cap}^W)$ hold for $(\mathcal{E},\mathcal{F})$, where $W$ and $	F $ are two scaling functions satisfying \eqref{WF}. Then, there exists a constant $C>0$
	such that 
	\begin{equation*}
		\left(\frac{d(x,y)}{\varepsilon}\right)^2\le\left(\frac{d_\varepsilon(x,y)}{%
			\varepsilon}\right)^2\le C\sup\limits_{z\in B(x,2d(x,y))}\frac{W(x,d(x,y))}{%
			W(z,\varepsilon)}
	\end{equation*}
	for all $\varepsilon>0$ and $x,y\in M$ with $\varepsilon<d(x,y)<\iota\bar{R}%
	/4$. In particular, $\beta^\ast\ge 2$ holds essentially.
\end{proposition}

\begin{proof}
	Clearly it suffices to consider $\varepsilon\le\frac{1}{3}d(x,y)$ and $\iota\le 2$. Fix such $x,y,\varepsilon$.
	
	Same as in \cite[Lemmas 2.5]{Murugan2020}, under condition $(\mathrm{cap})$,
	there exists a locally finite $\varepsilon$-net $V\subset M$ containing $x,y$
	and $\phi_z\in\cutoff(B_z,5B_z)\cap C_0(M)$, where $B_z=B(z,\frac{\varepsilon%
	}{4})$, for all $z\in V$ such that 
	\begin{equation*}
		\sum\limits_{z\in V}\psi_z\equiv 1,\quad\phi_{z^{\prime }}1_{B_z}\equiv 0\quad\text{for all}\quad z\ne z^{\prime },\quad\text{and}\quad\mathcal{E}(\phi_z)\le%
		\frac{C\mu(B_z)}{W(B_z)}.
	\end{equation*}
	Set $B=B(x,2d(x,y))$, and take $\Phi\in\cutoff(\frac{3}{4}B,B)$ as in
	condition $(\mathrm{cap})$. For all $x^{\prime }\in M$, define 
	\begin{equation*}
		u(x^{\prime }):=\sum\limits_{z\in V\cap B}N_\varepsilon(x,z)\phi_z(x^{\prime
		})
	\end{equation*}
	so that $u\in\mathcal{F}\cap C_0(M)$. Further, same as \cite[Formula (2.20)]%
	{Murugan2020}, for all $z\in V\cap B$ and $x^{\prime }\in B(z,\varepsilon)$, 
	\begin{equation}  \label{upat}
		u(x^{\prime })=u(z)+\sum\limits_{z^{\prime }\in V\cap B\cap
			9B_z}(N_\varepsilon(x,z^{\prime })-N_\varepsilon(x,z))\phi_{z^{\prime
		}}(x^{\prime }),
	\end{equation}
	while obviously 
	\begin{equation*}
		|N_\varepsilon(x,z^{\prime })-N_\varepsilon(x,z)|\le\#(V\cap 9B_z)\le N_1,
	\end{equation*}
	where $N_1$ comes from covering techniques based on $(\mathrm{VD})$.
	
	By $(\mathrm{MI}^F)$, (\ref{upat}), the Leibniz rule (due to strong locality, see \cite[Theorem 3.2.2]{FukushimaOshimaTakeda.2011.489})
	and $(\mathrm{VD})$, we obtain 
	\begin{eqnarray*}
		N_{\varepsilon }(x,y)^{2} &=&|u(x)-u(y)|^{2}=\big|(u\Phi)(x)-(u\Phi)(y)%
		\big|^{2}\leq CF(x,d(x,y))\mathcal{E}(u\Phi) \\
		&\leq &C^{2}\frac{W(B)}{\mu (B)}\int_{M}d\Gamma (u\Phi)\leq C^{2}\frac{W(B)%
		}{\mu (B)}\sum\limits_{z\in V\cap B}\int_{5B_{z}}d\Gamma (u\Phi) \\
		&\leq &C^{2}\frac{W(B)}{\mu (B)}\sum\limits_{z\in V\cap
			B}\sum\limits_{z^{\prime }\in V\cap B\cap 9B_{z}}(N_{\varepsilon
		}(x,z^{\prime })-N_{\varepsilon }(x,z))^{2}\int_{5B_{z}}d\Gamma (\phi
		_{z^{\prime }}\Phi) \\
		&\leq &C^{2}\frac{W(B)}{\mu (B)}\sum\limits_{z\in V\cap
			B}\sum\limits_{z^{\prime }\in V\cap B\cap
			9B_{z}}N_{1}^{2}\int_{5B_{z}}2\left( \Phi ^{2}d\Gamma (\phi _{z^{\prime
		}})+\phi _{z^{\prime }}^{2}d\Gamma (\Phi)\right) \\
		&\leq &C^{3}\frac{W(B)}{\mu (B)}\left( \sum\limits_{z\in V\cap
			B}\int_{5B_{z}}d\Gamma (\Phi)+\sum\limits_{z^{\prime }\in V\cap B}\mathcal{E%
		}(\phi _{z^{\prime }})\right) \leq C^{4}\frac{W(B)}{\mu (B)}\left( \mathcal{E%
		}(\Phi)+\sum\limits_{z^{\prime }\in V\cap B}\frac{V(z^{\prime },\varepsilon
			)}{W(z^{\prime },\varepsilon)}\right) \\
		&\leq &C^{4}\frac{W(B)}{\mu (B)}\left( C\frac{\mu (B)}{W(B)}+\frac{\mu (2B)}{%
			\inf\limits_{z^{\prime }\in B}W(z^{\prime },\varepsilon)}\right) \leq
		C^{6}\left( 1+\frac{W(B)}{\inf\limits_{z^{\prime }\in B}W(z^{\prime
			},\varepsilon)}\right) \leq C^{7}\sup\limits_{z^{\prime }\in B}\frac{%
			W(x,d(x,y))}{W(z^{\prime },\varepsilon)}.
	\end{eqnarray*}%
	The proof is then completed by combining (\ref{rgNe}).
\end{proof}

\subsection{Heat kernel estimates on general spaces}\label{hk-g}

Recall that, if the heat semigroup $\{P_{t}\}_{t>0}$ have an integral kernel 
$p_{t}(x,y)$, then we call $p_{t}(x,y)$ a heat kernel. Now we introduce
several conditions of heat kernel estimates with a given scaling function $W$
satisfying (\ref{scpr}). Let $\beta_\ast,\beta^\ast$ be the constants in (\ref{scpr}).

\begin{definition}[On-diagonal upper estimate]\label{dDUE}
We say that condition $(\mathrm{DUE}^{W})$ holds, if the heat kernel $%
p_{t}(x,y)$ exists pointwise on $(0,\infty)\times M\times M$, and for any $%
C_{0}\geq 1$, there exists a positive constant $C$ such that for all $x\in M$
and $0<t<C_{0}W(x,\bar{R})$, 
\begin{equation*}
p_{t}(x,x)\leq \frac{C}{V(x,W^{-1}(x,t))}.
\end{equation*}
\end{definition}

\begin{definition}[Upper estimate with exponential tail]\label{dUE}
We say that condition $(\mathrm{UE}_{\exp }^{W})$ holds, if the heat kernel $p_{t}(x,y)$ exists pointwise on $(0,\infty)\times M\times M$, the super-scaling exponent $\beta ^{\ast }>1$, and there exist two positive
constants $C,c$ such that for all $x,y\in M$ and all $0<t<W(x,\bar{R})\wedge
W(y,\bar{R})$, 
\begin{equation*}
p_{t}(x,y)\leq \frac{C}{V(x,W^{-1}(x,t))}\exp \left( -c\left( \frac{%
W(x,d(x,y))}{t}\right) ^{\frac{1}{\beta ^{\ast }-1}}\right) .
\end{equation*}
\end{definition}
Denote by $(\mathrm{UE}_{\mu,\exp})$, if $(\mathrm{UE}_{\exp})$ holds for $\mu$-a.e. (instead of ``all'') $x\in M$.

\begin{definition}[Near-diagonal lower estimate]\label{dNLE}
We say that condition $(\mathrm{NLE}^{W})$ holds, if the heat kernel $%
p_{t}(x,y)$ exists pointwise on $(0,\infty)\times M\times M$, and there
exist two positive constants $\eta ,C$ such that for all $x\in M$, all $%
0<t<W(x,\bar{R})$ and all $y\in B(x,\eta W^{-1}(x,t))$, 
\begin{equation*}
p_{t}(x,y)\geq \frac{C^{-1}}{V(x,W^{-1}(x,t))}.
\end{equation*}
\end{definition}

\begin{definition}[Lower estimate with exponential tail]\label{dLE}
We say that condition $(\mathrm{LE}_{\exp }^{W})$ holds if the heat kernel $%
p_{t}(x,y)$ exists pointwise on $(0,\infty)\times M\times M$, the
sub-scaling exponent $\beta _{\ast }>1$, and there exist two positive
constants $C,c$ such that for all $x,y\in M$ and all $0<t<W(x,\bar{R})\wedge
W(y,\bar{R})$, 
\begin{equation*}
p_{t}(x,y)\geq \frac{C^{-1}}{V(x,W^{-1}(x,t))}\exp \left( -c\left( \frac{%
W(x,d(x,y))}{t}\right) ^{\frac{1}{\beta _{\ast }-1}}\right) .
\end{equation*}
\end{definition}

Here $(\mathrm{UE}_{\exp}^{W})$ and $(\mathrm{LE}_{\exp }^{W})$
are new conditions, while all the other conditions come
from \cite{GrigoryanHuHu.2022.TPLE, GrigoryanHuHu.2022.TPGcap}.

Here $\beta ^{\ast }>1$ is always feasible (by replacing $\beta ^{\ast }$ by $\beta ^{\ast }\vee 2$). Furthermore, it is essential that $\beta ^{\ast}\geq 2$ according to Proposition \ref{bt2+}. However, $\beta _{\ast }>1$ is
not universally true, see the one given by (\ref{bt1-}) as a counterexample.
Occasionally, an alternative off-diagonal lower estimate is still attainable
when $\beta _{\ast }\leq 1$, which may not be related to $\beta _{\ast }$ at
all, as shown in Theorem \ref{le1-}.

Now we state the main results, which provide a natural extension of the established heat kernel theory to our broader context. Lengthy technical adaptations of known proofs to the $W(x,r)$-scaling context are involved. To make the presentation self-contained, they will be covered in Appendix \ref{app1}.

\begin{proposition}\label{thm:main1}
	Let $(\mathcal{E},\mathcal{F})$ be a regular Dirichlet form on $L^2(M,\mu)$, and $W$ is a scaling function with scaling exponents $0<\beta_\ast\leq\beta^\ast$.  
	Suppose that $\mu$ satisfies $(\mathrm{VD})$ and $(\mathrm{RVD})$ and $\beta^\ast>1$. Then 
	\[
	(\mathrm{FK}^W)+(\mathrm{S}^W)+(\mathrm{sloc}) \Leftrightarrow (\mathrm{UE}^W_{\mu,\exp})+(\mathrm{C}).
	\]
\end{proposition}

\begin{proof}[Sketch of the proof]
	Both directions will be completed in the same way as some known results.
	
	\noindent\textbf{LHS $\Rightarrow$ RHS.} Assume $(\mathrm{FK}^{W})$, $(\mathrm{S}^{W})$, and $(\mathrm{sloc})$.
	\begin{itemize}
		\item $(\mathrm{VD})+(\mathrm{FK^{W}})+(\mathrm{Gcap}^{W})+(\mathrm{TJ}^{W}) \Rightarrow (\mathrm{DUE}^{W})+(\mathrm{C})$ by \cite[Corollary 2.14]{GrigoryanHuHu.2023.TPDUE}. Here $(\mathrm{Gcap})$ follows from $(\mathrm{S})$ via \cite[Theorem 14.1]{GrigoryanHuHu.2022.TPGcap}, and $(\mathrm{TJ}^{W})$ follows from $(\mathrm{sloc})$ trivially.
		\item $(\mathrm{S}^W)+(\mathrm{sloc}) \Rightarrow (\mathrm{T}^W_{\exp})$, following the same idea as \cite[Proposition 5.2]{Hu.2008.PEMS2171} or \cite[Theorem 3.4]{GrigoryanHu.2008.IM81}. See details in Lemma \ref{pf1-1}.
		\item $(\mathrm{DUE})+(\mathrm{T}^W_{\exp}) \Rightarrow (\mathrm{UE}^W_{\mu,\exp})$, following the same idea as \cite[Pages 550--551]{GrigoryanHu.2014.MMJ505}. See details in Lemma \ref{pf1-2}.
	\end{itemize}
	
	\noindent\textbf{RHS $\Rightarrow$ LHS.} Assume $(\mathrm{UE}^W_{\mu,\exp})$ and $(\mathrm{C})$.
	\begin{itemize}
		\item $(\mathrm{VD})+(\mathrm{RVD})+(\mathrm{DUE}^{W}) \Rightarrow (\mathrm{FK}^W)$ by \cite[Proposition 10.6]{GrigoryanHuHu.2024.TPUEq}, where $(\mathrm{DUE})$ follows easily from $(\mathrm{UE}^W_{\mu,\exp})$.
		\item $(\mathrm{VD})+(\mathrm{UE}^W_{\mu,\exp})+(\mathrm{C}) \Rightarrow (\mathrm{sloc})$, following the same idea as \cite[Lemma 4.5]{GrigoryanHuLau.2009.3}. See details in Lemma \ref{pf1-3}.
		\item $(\mathrm{VD})+(\mathrm{UE}^W_{\mu,\exp})+(\mathrm{C}) \Rightarrow (\mathrm{S}^W)$, following the same idea as \cite[Proposition 5.8 and Page 549]{GrigoryanHu.2014.MMJ505}. See details in Lemma \ref{pf1-4}.
	\end{itemize}
\end{proof}

Further, let us focus on a special case when there exists another scaling function $F$ (with scaling exponents $0<\gamma_\ast\leq\gamma^\ast$) such that 	\begin{equation}\label{WF}
W(x,r)\asymp F(x,r)V(x,r)\quad\text{for all}\quad x\in M, r\in(0,\iota\bar{R}).
\end{equation}

\begin{theorem}\label{thm:main2}
Let $(\mathcal{E},\mathcal{F})$ be a regular Dirichlet form on $L^2(M,\mu)$. Suppose that $\mu$ satisfies both $(\mathrm{VD})$ and $(\mathrm{RVD})$, and \eqref{WF} holds with a new scaling function $F$. Then
\begin{equation*}
(\mathrm{MI}^{F})+(\mathrm{S}^{W})+(\mathrm{sloc})\Leftrightarrow(\mathrm{UE^W_{\exp}})+(\mathrm{NLE}^{W})+(\mathrm{C}).
\end{equation*}
\end{theorem}

\begin{remark}\label{ctn-resist}
When $(\mathrm{MI}^{F})$ and \eqref{WF} hold, it follows by a routine argument as in \cite{GrigoryanHuLau.2014.TAMS6397} that every $u\in\mathcal{F}$ admits a continuous version (still denoted by $u$). In particular, the heat kernel is continuous with respect to the spatial variants (i.e., for any $t>0$, the heat kernel $p_t(x,y)\in C(M\times M)$.)
\end{remark}

\begin{proof}[Sketch of the proof]
Compared with Proposition \ref{thm:main1}, it is enough to prove the following implications:

\begin{itemize}
\item $(\mathrm{VD})+(\mathrm{MI}^{F})+(\mathrm{RVD})\Rightarrow (\mathrm{FK}^{W})$, following the same idea as \cite[Proof of (6.35)]{GrigoryanHuLau.2014.TAMS6397}. See details in Lemma \ref{pf2-1}.

\item $(\mathrm{S}^{W})+(\mathrm{MI}^{F})+(\mathrm{DUE}^{W})\Rightarrow(\mathrm{NLE}^{W})$, following the same idea as \cite[Pages 6434--6435]{GrigoryanHuLau.2014.TAMS6397}. See details in Lemma \ref{pf2-2}.

\item $(\mathrm{VD})+(\mathrm{NLE}^{W})\Rightarrow(\mathrm{MI}^{F})$, following the same idea as \cite[Proof of Theorem 4.11]{GrigoryanHuLau.2003.TAMS2065}. See details in Lemma \ref{pf2-3}.
\end{itemize}
\end{proof}

With the help of chain condition $(\mathrm{CH})$, a refined characterization of heat kernel estimate is established as below:

\begin{proposition}\label{thm:main3}
Under the same assumptions as in Theorem \ref{thm:main2}, if $(M,d)$ satisfies the chain condition $(\mathrm{CH})$ and $\beta_\ast>1$, where $\beta_\ast$ is the lower scaling exponent of $W$, then 
\begin{equation*}
(\mathrm{MI}^{F})+(\mathrm{S}^{W})+(\mathrm{sloc})\Leftrightarrow(\mathrm{UE}^W_{\exp})+(\mathrm{LE}^W_{\exp})+(\mathrm{C}).
\end{equation*}
\end{proposition}
We will show the proof details in Appendix \ref{pf2}.

\subsection{Heat kernel estimates using resistance}

Direct verification of the Faber-Krahn inequality $(\mathrm{FK})$ and the survival estimate $(\mathrm{S})$ is often difficult. In this subsection, we show that both can be derived from appropriate estimates on the effective resistance, which are typically easier to check on self-similar structure.

We begin by recalling the notion of effective resistance. For disjoint non-empty closed sets $A,B\subset M$, define
\begin{equation}
\mathcal{R}(A,B)^{-1}:=\inf \big\{\mathcal{E}(u):u\in \mathcal{F}\cap
C_{0}(M),u|_{A}\equiv 1,u|_{B}\equiv 0\big\}.  \label{ER2}
\end{equation}%
Clearly $\mathcal{R}(A,B)$ is decreasing in sets $A$, $B$ in the sense that 
\begin{equation*}
\mathcal{R}(A_{1},B_{1})\leq \mathcal{R}(A_{2},B_{2})\quad \text{if}\quad
A_{1}\supseteq A_{2}\quad\text{and}\quad B_{1}\supseteq B_{2}.
\end{equation*}%
For simplicity, denote 
\begin{equation*}
\mathcal{R}(x,y):=\mathcal{R}(\{x\},\{y\})\quad \text{and}\quad \mathcal{R}%
(x,A):=\mathcal{R}(\{x\},A).
\end{equation*}

In general, it may happen that $\mathcal{R}(x,y)=\infty$ for some points $x,y\in M$. In this subsection, we will exclude this case by the following group of conditions with respect to a function $F:M\times[0,\infty)\to[0,\infty)$ (not necessarily a scaling function):

\begin{definition}[Resistance estimates]
We say that

\begin{enumerate}
\item[i)] condition $(\mathrm{R}_{\leq }^{F})$ holds, if there exist two
positive constants $C,\iota $ such that for all $x,y\in M$ with $%
d(x,y)<\iota \bar{R}$, 
\begin{equation*}
\mathcal{R}(x,y)\leq CF(x,d(x,y)).
\end{equation*}

\item[ii)] condition $(\mathrm{R}_{\geq }^{F})$ holds, if there exists a
positive constant $C$ such that for all $x,y\in M$, 
\begin{equation*}
\mathcal{R}(x,y)\geq C^{-1}F(x,d(x,y)).
\end{equation*}

\item[iii)] condition $(\mathrm{R}^{F})$ holds, if both conditions $(\mathrm{R}_{\leq }^{F})$ and $(\mathrm{R}_{\geq }^{F})$ are satisfied with the same $F$.
\end{enumerate}
\end{definition}

We introduce a stronger condition concerning $\mathcal{R}$.

\begin{definition}[Resistance estimate with respect to balls]
We say that condition $(\mathrm{RB}_{\geq }^{F})$ holds, if there exists a positive constant $C$ such that, for every ball $B(x,r)\subsetneqq M$, 
\begin{equation*}
\mathcal{R}(x,B(x,r)^{c})\geq C^{-1}F(x,r).
\end{equation*}
\end{definition}

Note there are trivial implications $(\mathrm{MI}^{F})\Rightarrow(\mathrm{R}_{\leq}^{F})$ 
and $(\mathrm{RB}_{\geq}^{F})\Rightarrow(\mathrm{R}_{\geq}^{F})$. Many other implications concerning these conditions are covered in Appendix \ref{pf4}, in order to show the following:

\begin{theorem}\label{maint}
Let $(\mathcal{E},\mathcal{F})$ be a regular Dirichlet form and \eqref{WF} holds for scaling functions $W,F$. Assume that $(\mathrm{VD})$ and $(\mathrm{RVD})$ hold.

\begin{enumerate}
\item[i)] Assume that $\beta ^{\ast }>1$. Then 
\begin{equation*}
	(\mathrm{Para})+(\mathrm{R}^{F})+(\mathrm{sloc})\Leftrightarrow (\mathrm{UE}%
	_{\exp }^{W})+(\mathrm{NLE}^{W}).
\end{equation*}

\item[ii)] If both $\beta _{\ast }>1$ and $(\mathrm{CH})$ are satisfied,
then 
\begin{equation*}
	(\mathrm{Para})+(\mathrm{R}^{F})+(\mathrm{sloc})\Leftrightarrow (\mathrm{UE}%
	_{\exp }^{W})+(\mathrm{LE}_{\exp }^{W}).
\end{equation*}
\end{enumerate}
\end{theorem}

\begin{proof}[Sketch of the proof]
Compared with results in Subsection \ref{hk-g}, it is enough to prove the following implications:

\begin{itemize}
\item $(\mathrm{R}^{F}_{\leq})+(\mathrm{Para})\Rightarrow(\mathrm{MI}^F)$ and $(\mathrm{sloc})+(\mathrm{MI}^F)+(\mathrm{R}_{\geq }^{F})\Rightarrow(\mathrm{RB}_{\geq }^{F})$, following the same idea as \cite[Propositions 6.6 and 6.9]{GrigoryanHuLau.2014.TAMS6397}. See details in Lemma \ref{RandQR}.

\item $(\mathrm{MI}^{F})+(\mathrm{RB}_{\geq}^{F})\Rightarrow(\mathrm{S}^{W}_+)$ (and hence $(\mathrm{S}^{W})$ trivially), following the same idea as \cite[Proof of Theorem 6.13]{GrigoryanHuLau.2014.TAMS6397}. See details in Lemma \ref{RandS}.

\item $(\mathrm{VD})+(\mathrm{NLE}^{W})+(\mathrm{UE}_{\exp}^W)\Rightarrow(\mathrm{Para})+(\mathrm{R}_{\geq}^{F})$, following the same idea as \cite[Proof of Theorem 6.17]{GrigoryanHuLau.2014.TAMS6397}. See details in Lemma \ref{R-low}.
\end{itemize}
\end{proof}

\section{Constructing Scaling Functions on Self-Similar Sets via Coefficients}
\label{SSSWay}

Now we apply the general theory above to a concrete and important class of metric measure spaces: self-similar fractals. In this section, we develop a systematic scheme for
constructing scaling functions on such sets. The construction intrinsically depends on the iterated function system (IFS) generating the set, utilizing two given tuples of self-similar coefficients. The central idea
is that if one tuple satisfies an \emph{average condition} (see Definition \ref{AV}) relative to the other, and if the set satisfies appropriate geometric
conditions, then one can explicitly construct a scaling function with scaling exponents prescribed by the two tuples of coefficients.

The main result of this section is the following constructive lemma.

\begin{proposition}[Construction of scaling function]
\label{ML4} Let $(K,\{F_{i}\}_{i=1}^{N})$ be a self-similar set and $\mathbf{%
\theta },\mathbf{\zeta }$ be two $N$-tuples such that $\mathbf{\zeta }$ satisfies the average condition $(%
\mathrm{AV}_{\mathbf{\theta }})$. Assume that $K$ satisfies the connectivity
condition $(\mathrm{CP}_{\mathbf{\theta },c})$ for some
constant $c\in (0,1)$. Then there exists a scaling function $H$, such that

\begin{enumerate}
\item The sub- and super-scaling exponents of $H$ are sharply given by
\begin{equation}
a_{\ast }:=\min_{1\leq i\leq N}\frac{\ln \zeta _{i}}{\ln \theta _{i}}\quad 
\text{and}\quad a^{\ast }:=\max_{1\leq i\leq N}\frac{\ln \zeta _{i}}{\ln
\theta _{i}}.  \label{thet'a}
\end{equation}

\item There exists a constant $C$ such that, for all $r\in (0,\bar{R})$, every
$x\in K$ and every $w\in \Lambda _{\mathbf{\theta }}(r/\bar{R})$ with $x\in
K_{w}$, 
\begin{equation*}
C^{-1}\xi _{w}\leq H(x,r)\leq C\xi _{w}.
\end{equation*}
\end{enumerate}
\end{proposition}

Notions involved here and the constructive proof will be detailed in Subsections \ref{subsec:preliminaries} and \ref{subsec:construction} respectively.

The so-constructed scaling function benefits much to the analysis on self-similar fractals. In particular, we could verify sufficient conditions for the doubling property of self-similar measures in terms of their contraction ratios and probability weights (see Corollary \ref{muScal}) by applying Proposition \ref{ML4} to $\mathbf{\theta}=\mathbf{\rho}$ (the contraction ratios) and $\mathbf{\zeta}=\mathbf{p}$ (the probability weights). This would further serve as a fundamental tool for heat kernel bounds in Section \ref{Kl}.

\subsection{Preliminaries: self-similar structure and related conditions}

\label{subsec:preliminaries}

We begin by recalling the standard framework of self-similar sets and notation, mostly from \cite{Kigami.2001.226}. Let $N\geq 2$ and $\{F_{i}\}_{i=1}^{N}$ be a finite
family of distinct contracting maps on a metric space $(M,d)$ such that for
each $1\leq i\leq N$, 
\begin{equation}  \label{ssp1}
d(F_{i}(x),F_{i}(y))\leq \rho _{i}d(x,y)
\end{equation}%
for all $x,y\in M$, with $\rho _{i}\in(0,1)$ a \emph{contraction ratio} of $F_{i}$.

If $(M,d)$ is complete, then there exists a unique non-empty compact set $K$
such that 
\begin{equation}  \label{ssp2}
K=\tbigcup\limits_{i=1}^NF_i(K),
\end{equation}
which is called the \emph{self-similar set} generated by the \emph{iterated
function system} (IFS) $\{F_i\}_{i=1}^N$.

We say an $N$-dimensional vector $\mathbf{\theta}:=(\theta_1,\theta_2,%
\dotsc,\theta_N)$ is an $N$-\emph{tuple}, if $\theta_i\in(0,1)$ for each $i$. For
example, the contraction ratios in (\ref{ssp1}) constitute an $N$-tuple 
\begin{equation}  \label{rho}
\mathbf{\rho}:=(\rho_1,\rho_2,\dotsc,\rho_N).
\end{equation}
For every $N$-tuple $\mathbf{\theta}$, denote 
\begin{equation*}
\theta_{\min}:=\min_{1\leq i\leq N}\theta_i,\quad\text{and}\quad\theta_{\max}:=\max_{1\leq
i\leq N}\theta_i.
\end{equation*}

Let $I:=\{1,2,\dotsc ,N\}$ be the index set. For every $n\geq 1$, set 
\begin{equation*}
I^{n}=\big\{w=w_{1}w_{2}\cdots w_{n}:\quad w_{i}\in I\quad\text{for each}\quad 1\leq
i\leq N\big\}
\end{equation*}%
and $I^{0}=\{\varnothing \}$. We say that $w$ is a \emph{finite word}, if
there exists an integer $n\geq 0$ such that $w\in I^{n}$, and let the \emph{length} of $w$ be $|w|:=n$. We say $\varnothing $ is the \emph{empty word}, and $w\in
I^{n}$ is a non-empty finite word if $n\geq 1$. All finite words belong to these two types. For convenience, we define 
\begin{equation*}
F_{w}:=F_{w_{1}}\circ \cdots \circ F_{w_{n}},\quad K_{w}:=F_{w}(K),\quad
\theta _{w}:=\theta _{w_{1}}\theta _{w_{2}}\cdots \theta _{w_{n}},
\end{equation*}%
for every $w=w_{1}\cdots w_{n}\in I^{n}$ and $N$-tuple $\mathbf{\theta }$. Conventionally we set $F_{\varnothing }=id$, $K_{\varnothing}=K $ and $\theta _{\varnothing }=1$.

Let the set of all infinite words over $I$ be 
\begin{equation*}
	I^{\ast }:=\big\{\omega =\omega _{1}\omega _{2}\cdots :\text{ each }\omega
	_{i}\in I\big\}.
\end{equation*}%
By \cite[Theorem 1.2.3]{Kigami.2001.226}, there exists a surjective
projection $\pi :I^{\ast }\rightarrow K$ such that 
\begin{equation*}
	\{\pi (\omega)\}=\tbigcap\limits_{n\geq 1}F_{\omega _{1}\cdots \omega
		_{n}}(K)
\end{equation*}%
for each infinite word $\omega =\omega _{1}\omega _{2}\cdots \in I^{\ast }$.
We call $\omega \in \pi ^{-1}(\{x\})$ an \emph{address} of 
$x$.

Given two finite words $w=w_1w_2\cdots w_n$ and $\tau=\tau_1\tau_2\cdots%
\tau_m$, we set 
\begin{equation*}
w\tau:=w_1w_2\cdots w_n\tau_1\tau_2\cdots\tau_m,
\end{equation*}
where, if $\tau=\varnothing$, set $w\tau=w$; while if $w=\varnothing$, set $%
w\tau=\tau$. Obviously, for any $N$-tuple $\mathbf{\theta}$, 
\begin{equation}  \label{wtau}
\theta_{w\tau}=\theta_w\theta_\tau.
\end{equation}

We introduce a crucial algebraic condition to relate two general $N$-tuples $\mathbf{\theta }$, $\mathbf{\zeta}$, which will ensure the comparability of the $\mathbf{\zeta }$-weights of two cells that intersect and have comparable $\mathbf{\theta }$-sizes.

\begin{definition}[Average condition]\label{AV}
We say that an $N$-tuple $\mathbf{\zeta }=(\zeta _{1},\zeta _{2},\dotsc
,\zeta _{N})$ satisfies the \emph{average condition} with respect to some $N$-tuple $\mathbf{\theta }$, denoted by $(\mathrm{AV}_{\mathbf{\theta }})$, if
there exist two positive constants $C^{\prime },C^{\prime \prime }$ such
that for all non-empty finite words $v,w$ with 
\begin{equation*}
\theta _{w}\leq C^{\prime }\theta _{v}\quad \text{and}\quad K_{v}\cap
K_{w}\neq \emptyset ,
\end{equation*}
there is 
\begin{equation*}
\zeta _{w}\leq C^{\prime \prime }\zeta _{v}.
\end{equation*}
\end{definition}

Intuitively, $(\mathrm{AV}_{\mathbf{\theta }})$ implies some sort of continuity
of $\mathbf{\zeta }$ with respect to $\mathbf{\theta }$ on intersecting
cells. A more practical equivalent formulation will be given in Proposition %
\ref{av=} below. To state it, we first recall the notion of a partition
induced by $\mathbf{\theta }$.

Given $r\in (0,1)$, the \emph{partition} $\Lambda _{\mathbf{\theta }}(r)$
with respect to $\mathbf{\theta }$ is defined by 
\begin{equation}
\Lambda _{\mathbf{\theta }}(r):=\left\{ w=w_{1}w_{2}\cdots w_{n}\in
\tbigcup\limits_{k\geq 1}I^{k}:\ \ \theta _{w}\leq r<\frac{\theta _{w}}{%
\theta _{w_{n}}}\right\} ,  \label{partfor}
\end{equation}%
see for example \cite{Hambly.1997.AP1059}. It is easy to see that for all $%
w\in \Lambda _{\mathbf{\theta }}(r)$, 
\begin{equation}
\theta _{\min }r<\theta _{w}\leq r.  \label{part1}
\end{equation}%
Further, by the self-similarity (\ref{ssp2}), we know that $K=\cup _{w\in
\Lambda _{\mathbf{\theta }}(r)}K_{w}$.

If $0<r_{1}\leq r_{2}<1$, then the partition $\Lambda _{\mathbf{\theta }%
}(r_{1})$ is a \emph{refinement} of $\Lambda _{\mathbf{\theta }}(r_{2})$ in
the sense that any word $w^{\prime }$ in $\Lambda _{\mathbf{\theta }}(r_{1})$
can be written as $w^{\prime }=w\tau $ for some $w\in \Lambda _{\mathbf{%
\theta }}(r_{2})$ and finite word $\tau $. In this representation, $%
w^{\prime }$ is called an\emph{\ offspring} of $w$, while $w$ is called an 
\emph{ancestor} of $w^{\prime }$.

The following basic observations for the partition will be useful:

\begin{itemize}
\item It follows by (\ref{part1}) that, for any $0<r_{1}\leq r_{2}<1$ and
any $w^{\prime }\in \Lambda _{\mathbf{\theta }}(r_{1})$, $w\in \Lambda _{%
\mathbf{\theta }}(r_{2})$, 
\begin{equation}  \label{part2.1}
\theta _{\min }\frac{r_{1}}{r_{2}}\leq \frac{\theta _{w^{\prime }}}{%
\theta_{w}}\leq (\theta _{\min })^{-1}\frac{r_{1}}{r_{2}}
\end{equation}%
Further, if $w^{\prime }=w\tau $ with some finite word $\tau $, it follows
by (\ref{part2.1}) and $\frac{\theta _{w^{\prime }}}{\theta _{w}}=\theta
_{\tau }\leq \theta _{\max }^{|\tau |}$ that 
\begin{equation}  \label{part2.2}
|\tau |=|w^{\prime }|-|w|\leq \frac{\ln (\theta _{\min }r_{1}/r_{2})}{\ln
\theta _{\max }}.
\end{equation}

\item Let $\mathbf{\theta },\mathbf{\zeta }$ be two $N$-tuples. Then for any
finite word $w$, 
\begin{equation}
(\theta _{w})^{a^{\ast }}\leq \zeta _{w}\leq (\theta _{w})^{a_{\ast }},
\label{part2.3}
\end{equation}%
where $a_{\ast },a^{\ast }$ are constants defined in (\ref{thet'a}). Note by
(\ref{part1}), it follows that for any $r\in (0,1)$ and any finite word $%
v\in \Lambda _{\mathbf{\theta }}(r)$, 
\begin{equation}
(\theta _{\min }r)^{a^{\ast }}\leq \zeta _{v}\leq r^{a_{\ast }}.
\label{part2.4}
\end{equation}
\end{itemize}

Here we show an equivalent characterization of condition $(\mathrm{AV}_{%
\mathbf{\theta }})$ based on partitions.

\begin{proposition}
\label{av=} Let $\mathbf{\zeta },\mathbf{\theta }$ be two $N$-tuples. Then
the following properties are equivalent:

\begin{enumerate}
\item[a)] $\mathbf{\zeta}$ satisfies $(\mathrm{AV}_{\mathbf{\theta}})$.

\item[b)] There exists a constant $C>0$ such that for all non-empty finite
words $v,w$ with $\theta_{\min}\theta_v\leq\theta_w\leq\theta_{\min}^{-1}
\theta_v$ and $K_v\cap K_w\neq\emptyset$, 
\begin{equation*}
C^{-1}\zeta_v\leq\zeta_w\leq C\zeta_v.
\end{equation*}

\item[c)] There exists a constant $C>0$ such that for all $r\in (0,1)$ and
all non-empty finite words $v,w\in \Lambda _{\mathbf{\theta }}(r)$ with $%
K_{v}\cap K_{w}\neq \emptyset $, 
\begin{equation*}
C^{-1}\zeta _{v}\leq \zeta _{w}\leq C\zeta _{v}.
\end{equation*}
\end{enumerate}
\end{proposition}

\begin{proof}
Indeed, it is obvious that a)$\Rightarrow $ b)$\Rightarrow $ c). Now we
prove that c)$\Rightarrow $ a).

Fix a positive constant $C^{\prime }$ and non-empty finite words $v,w$ such
that $\theta _{w}\leq C^{\prime }\theta _{v}$ and $K_{v}\cap K_{w}\neq
\emptyset $. Fix $x\in K_{v}\cap K_{w}$. To find a constant $C^{\prime
\prime }>0$ such that $\zeta _{w}\leq C^{\prime \prime }\zeta _{v}$, we
distinguish three cases.

Case $(1)$: $\theta _{v}=\theta _{w}$. Clearly, $v,w\in
\Lambda _{\mathbf{\theta }}(\theta _{v})$, thus it follows directly by c) that $\zeta
_{w}\leq C\zeta _{v}$, where $C$ also comes from c).

Case $(2)$: $\theta _{v}<\theta _{w}$. Clearly $v\in \Lambda _{\mathbf{%
\theta }}(\theta _{v})$, and we look for an offspring of $w$ that is also in 
$\Lambda _{\mathbf{\theta }}(\theta _{v})$.

Indeed, since $x\in K_{w}$, then we can take an address of $x$ of form $wi_1i_2i_3\cdots $. In particular, $x\in K_{w^{(n)}}$ with $w^{(n)}=wi_{1}i_{2}\cdots i_{n}$ for every $n\geq 1$. Set also $w^{(0)}=w$.
Since $\theta _{w^{(k)}}=\theta _{wi_{1}i_{2}\cdots i_{k}}\rightarrow 0$ as $%
k\rightarrow \infty $ whilst $\theta _{v}>0$, there must be $n\geq 1$ such
that 
\begin{equation*}
\theta _{w^{(n)}}\leq \theta _{v}<\theta _{w^{(n-1)}}.
\end{equation*}
It follows directly that $w^{(n)}\in \Lambda _{\mathbf{\theta }}(\theta
_{v}) $. Meanwhile, $K_{w^{(n)}}\cap K_{v}\neq \emptyset $ since $x\in
K_{w^{(n)}}\cap K_{v}$. Hence, it follows by c) that $\zeta _{w^{(n)}}\leq
C\zeta _{v}$.

On the other hand, since $w\in \Lambda _{\mathbf{\theta }}(\theta
_{w}) $ and $w^{(n)}\in \Lambda _{\mathbf{\theta }}(\theta _{v})$, by (\ref{part2.2}) and $\theta _{w}\leq C^{\prime }\theta _{v}$,
\begin{equation*}
n\leq \frac{\ln \left( \theta _{\min }\frac{\theta _{v}}{\theta _{w}}\right) 
}{\ln \theta _{\max }}\leq \frac{\ln \left( \theta _{\min }/C^{\prime
}\right) }{\ln \theta _{\max }}=\frac{-\ln C^{\prime }+\ln \theta _{\min }}{
\ln \theta _{\max }}.
\end{equation*}
Therefore, 
\begin{equation*}
\frac{\zeta _{w}}{\zeta _{v}}=\frac{\zeta _{wi_{1}\cdots i_{n}}}{\zeta _{v}} 
\frac{1}{\zeta _{i_{1}\cdots i_{n}}}\leq \frac{C}{(\zeta _{\min })^{n}}\leq
C(\zeta _{\min })^{\frac{\ln C^{\prime }-\ln \theta _{\min }}{\ln \theta
_{\max }}}=:C_{1}(C,\mathbf{\theta },\mathbf{\zeta }).
\end{equation*}

Case $(3)$: $\theta _{v}>\theta _{w}$. Similar to Case $(2)$, there exist an integer $m\geq 1$ and digits $j_{1},j_{2},\cdots ,j_{m}\in
I $ such that the word $v^{(m)}=vj_{1}\cdots j_{m}\in \Lambda _{\mathbf{\
\theta }}(\theta _{w})$ satisfies $x\in K_{v^{(m)}}\cap K_{w}$. Then we
have by c) and the fact $0<\zeta _{k}<1$ for every $k\in I$ that 
\begin{equation*}
\zeta _{w}\leq C\zeta _{v^{(m)}}=C\zeta _{v}\zeta _{j_{1}\cdots j_{m}}\leq
C\zeta _{v},
\end{equation*}
where $C$ also comes from c).

In summary, a) holds with the constant $C^{\prime \prime }:=C\vee C_{1}(C,%
\mathbf{\theta },\mathbf{\zeta })$, where $C$ comes from c).
\end{proof}

Now we consider geometry of $K$. Here we always assume that the set $K$ is
not a singleton so that 
\begin{equation*}
\bar{R}:=\sup_{x,y\in K}d(x,y)\in (0,\infty).
\end{equation*}

The following set of words will be frequently used in the sequel: for any
Borel subset $A$ of $M$ and any partition $\Lambda _{\mathbf{\theta }}(r)$,
define 
\begin{equation}
\Gamma _{\mathbf{\theta }}(A,r)=\big\{v\in \Lambda _{\mathbf{\theta }}(r):\
K_{v}\cap A\neq \emptyset \big\}.  \label{gam}
\end{equation}

\begin{definition}[Weak overlapping condition]\label{WOC}
Given an $N$-tuple $\mathbf{\theta}$, we say $(K,\{F_i\}_{i=1}^N)$ satisfies
a \emph{weak overlapping condition} with respect to $\mathbf{\theta}$ and a
parameter $C_1>0$, denoted by $(\mathrm{WOC}_{\mathbf{\theta},C_1})$, if
there exists an integer $L_1=L_1(\mathbf{\theta},C_1,\bar{R})\geq 1$ such
that for every point $x\in K$ and every $r\in(0,1)$, 
\begin{equation*}
\#\Gamma_{\mathbf{\theta}}(B(x,C_1\bar{R}r),r)\leq L_1.
\end{equation*}
We say $(K,\{F_i\}_{i=1}^N)$ satisfies $(\mathrm{WOC}_{\mathbf{\theta}})$ if
there exists $C_1=C_1(\mathbf{\theta})>0$ such that $(\mathrm{WOC}_{\mathbf{\theta},C_1})$ holds.
\end{definition}

This condition says that, for any word $v$ in any partition $%
\Lambda_{\mathbf{\theta}}(r)$ with $r\in(0,1)$, the number of cells $K_v$
intersecting an arbitrary ball of radius $C_1\bar{R}r$ is controlled by a
universal integer $L_1$. See Figure \ref{picWOC}. 
\begin{figure}[h]
\centering\includegraphics[width=0.35\textwidth]{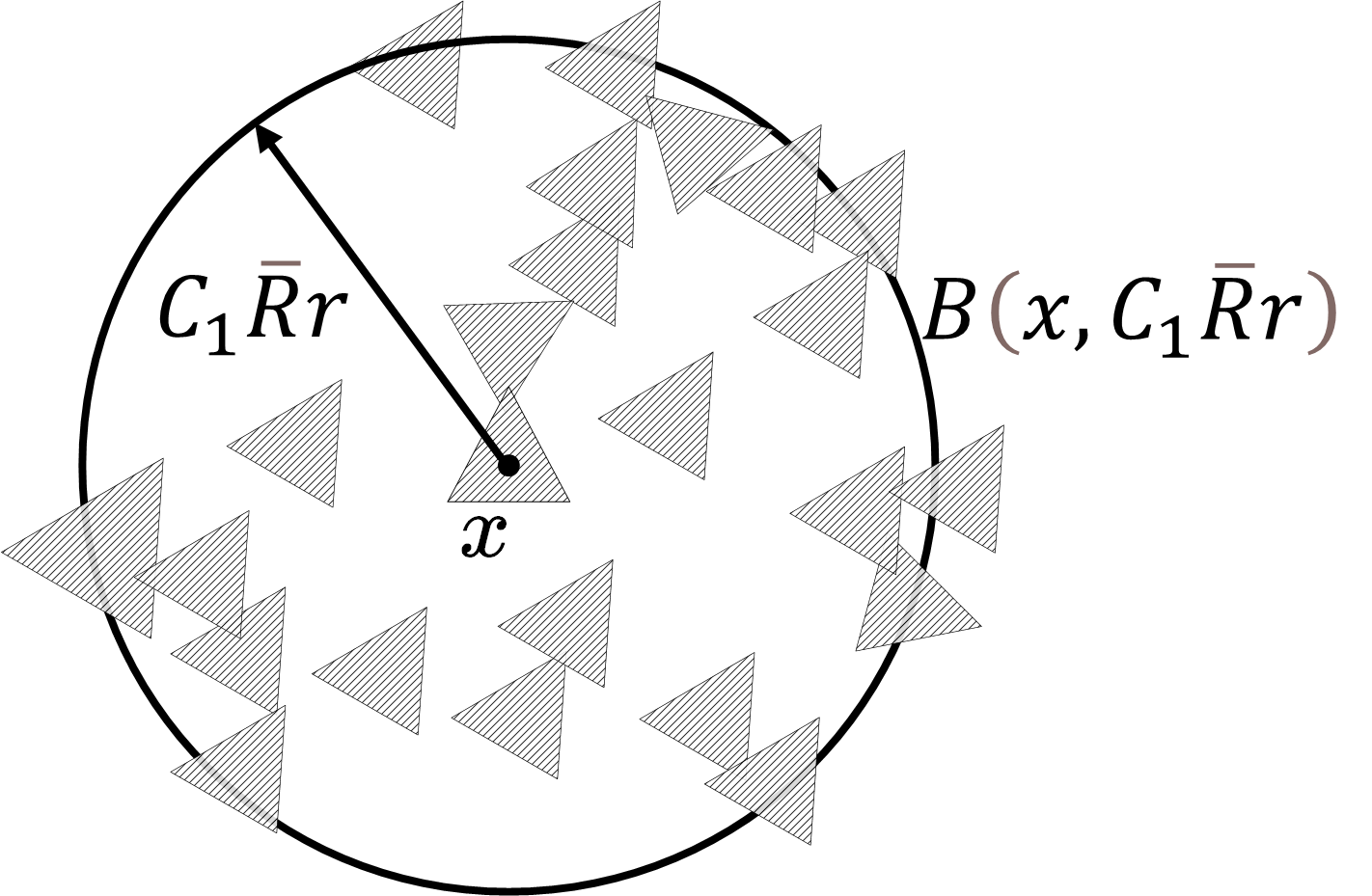}
\caption{The ball $B(x,C_1\bar{R}r)$ intersects at most $L_1$ cells from the
partition $\Lambda_{\mathbf{\protect\theta}}(r)$.}
\label{picWOC}
\end{figure}

Clearly, we know by definition that condition $(\mathrm{WOC}_{\mathbf{\theta}
,C_1})$ is increasingly stronger in $C_1$, that is, 
\begin{equation*}
(\mathrm{WOC}_{\mathbf{\theta},C^{\prime }_1})\Rightarrow(\mathrm{WOC}_{ 
\mathbf{\theta},C_1})\quad\text{for any }\ C^{\prime }_1>C_1.
\end{equation*}
On the converse, by counting words in
$$\Gamma _{\mathbf{\theta }}(B(x,C_{1}^{\prime }\bar{R}r),r)$$
(similar as in the proof of Proposition \ref{av=}), we can easily show $(\mathrm{WOC}_{\mathbf{\theta},C_1})\Rightarrow(\mathrm{WOC}_{\mathbf{\theta},C^{\prime }_1})$. That is,
\begin{equation*}
(\mathrm{WOC}_{\mathbf{\theta},C_1})\Leftrightarrow(\mathrm{WOC}_{\mathbf{\theta},C^{\prime }_1})\quad\text{for all }\ C^{\prime }_1>C_1.
\end{equation*}

The weak overlapping condition with respect to the tuple $\mathbf{\rho}$ defined in (\ref{rho}) is known to hold on a certain class of self-similar sets on $\mathbb{R}^n$, see the following lemma:

\begin{lemma}
\cite[Proposition 1.5.8]{Kigami.2001.226}\label{op'cond} A self-similar set $%
(K,\{F_i\}_{i=1}^N)$ on $\mathbb{R}^n$ satisfies $(\mathrm{WOC}_{\mathbf{%
\rho }})$, if the following two properties hold:

\begin{enumerate}
\item each $F_i$ is a contracting \emph{similitude}, that is, \eqref{ssp1} holds with ``$\le$'' replaced by ``$=$'';

\item the \emph{open set condition} is satisfied, that is, there exists a 
\emph{basic open set} $U\subset\mathbb{R}^n$ such that $\cup_{i=1}^NF_i(U)\subset U$
and $F_i(U)\cap F_j(U)=\emptyset$ for all $1\le i<j\le N$.
\end{enumerate}
\end{lemma}

Recall that a self-similar set $(K,\{F_{i}\}_{i=1}^{N})$ is said to satisfy
the \emph{connectivity property} $(\mathrm{CP}_{\mathbf{\theta },C_{2}})$
for an $N$-tuple $\mathbf{\theta }$ and a constant $C_{2}>0$, if there
exists an integer $L_{2}=L_{2}(\mathbf{\theta },C_{2},\bar{R})\geq 1$ such
that for any number $r\in (0,1)$ and any two points $x,y\in K$ such that $%
y\in B(x,C_{2}\bar{R}r)$, there exists a chain $\{w^{(i)}\}_{i=1}^{L_{2}}\subset \Lambda _{\mathbf{\theta }}(r)$ such that 
\begin{equation*}
x\in K_{w^{(1)}},\quad y\in K_{w^{(L_{2})}},\quad K_{w^{(i-1)}}\cap K_{w^{(i)}}\neq\emptyset\quad\text{for every}\quad 2\leq i\leq L_{2}.
\end{equation*}
Equivalently, there is a chain of points $\{x_{i}\}_{i=0}^{L_{2}}$ in $K$ satisfying $x_{0}=x$, $x_{L_{2}}=y$ and every successive pair $x_i,x_{i+1}$ lie in a same cell $K_w$ with $w\in\Lambda_{\mathbf{\theta}}(r)$. The latter version comes from, for example, \cite{HuWang.2006.SM153}. See Figure \ref{picCP}. 
\begin{figure}[h]
\centering\includegraphics[width=0.35\textwidth]{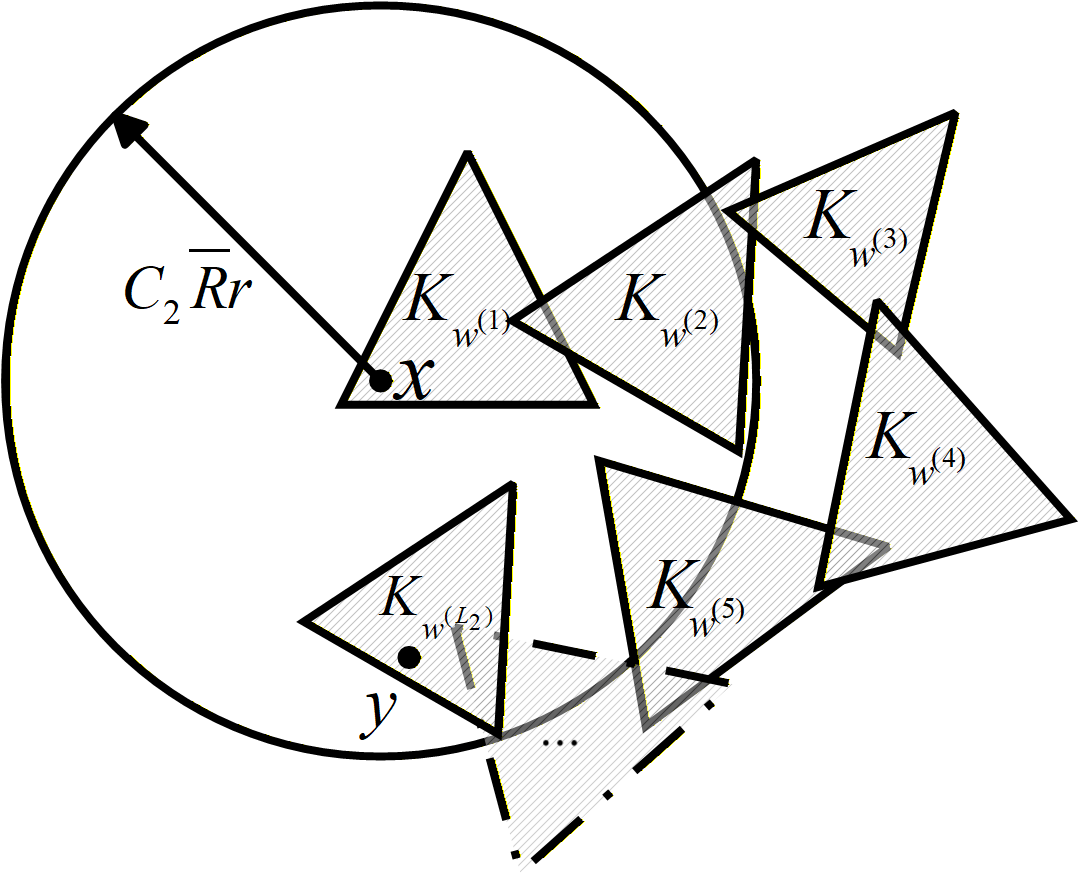}
\caption{Two points $x$ and $y$ are connected by a sequence $\{K_{w^{(i)}}\}_{i=1}^{L_{2}}$ with $x\in K_{w^{(1)}}, y\in K_{w^{(L_2)}}$ and $K_{w^{(i-1)}}\cap K_{w^{(i)}}\neq\emptyset$, where $w^{(i)}\in\Lambda_{\mathbf{\theta}}(r)$.}
\label{picCP}
\end{figure}

We denote by $(\mathrm{CP}_{\mathbf{\theta }})$
if there exists some constant $C_{2}=C_{2}(\mathbf{\theta })>0$ such that $(%
\mathrm{CP}_{\mathbf{\theta },C_{2}})$ holds. Note by definition that
condition $(\mathrm{CP}_{\mathbf{\theta },C_{2}})$ is increasingly stronger
in $C_{2}$, that is, 
\begin{equation*}
(\mathrm{CP}_{\mathbf{\theta },C_{2}^{\prime }})\Rightarrow (\mathrm{CP}_{%
\mathbf{\theta },C_{2}})\quad \text{for any }\ C_{2}^{\prime }>C_{2}.
\end{equation*}%
As for the opposite conclusion, provided that $K$ is connected, $(\mathrm{CP}_{\mathbf{\theta }})$ could be self-improved by a similar argument as in the case of $(\mathrm{WOC}_{\mathbf{\theta}})$. However, if 
$K$ is not connected, then the self-improvement may not be true in general. See the following counterexample:

\begin{example}
Let $K$ be the Cantor set in the interval $[0,1]$ generated by $F_{1}(x)=x/3$, $F_{2}(x)=(x+2)/3$. It is easy to check that $(\mathrm{CP}_{\mathbf{\theta },1/16})$ holds whereas $(\mathrm{CP}_{\mathbf{\theta },5/2})$ fails.
\end{example}

It will be frequently assumed that $C_{2}<1$.

We remark also that condition $(\mathrm{CP}_{\mathbf{\theta }})$ reflects local
connectivity, while the chain condition $(\mathrm{CH})$\ ensures the global
connectivity. An exact relation between them is shown as:

\begin{proposition}
	\label{P-ch2} $(\mathrm{WOC}_{\mathbf{\theta}})+(\mathrm{CH})\Rightarrow(%
	\mathrm{CP}_{\mathbf{\theta}})$.
\end{proposition}

\begin{proof}
	Fix an $N$-tuple $\mathbf{\theta }$, $x_{0}\in K$ and $r\in (0,1)$. Define 
	\begin{equation*}
		\Gamma :=\Gamma _{\mathbf{\theta }}\left( B(x_{0},(C_{\mathrm{CH}}+1)\bar{R}%
		r),r\right) \quad \text{and}\quad K_{\Gamma }:=\tbigcup\limits_{v\in \Gamma
		}K_{v}.
	\end{equation*}%
	For a point $x\in K_{\Gamma }$, define a sequence of sets $\{ \Gamma_{x}^{(i)}\} _{i=0}^{\infty }$ contained in $\Gamma $ inductively by 
	\begin{eqnarray*}
		\Gamma _{x}^{(0)} &:=&\big\{v\in \Gamma :x\in K_{v}\big\}; \\
		\Gamma _{x}^{(i+1)} &:=&\left\{ v\in \Gamma :\text{there exists }\ \tau \in
		\Gamma _{x}^{(i)}\ \text{ such that }\ K_{v}\cap K_{\tau }\neq \emptyset
		\right\} \quad (i\geq 0).
	\end{eqnarray*}%
	In other words, the set $\Gamma _{x}^{(0)}$ consists of all words $v$ in $\Gamma $
	such that the point $x$ locates in $K_{v}$, while $\Gamma _{x}^{(i+1)}$
	consists of all words $v$ in $\Gamma $ such that $K_{v}$ intersects some
	cell $K_{\tau }$ with $\tau \in \Gamma _{x}^{(i)}$. See Figure \ref{Rl1}.
	
	\begin{figure}[h]
		\centering
		\begin{minipage}[t]{0.48\textwidth}
			\centering
			\includegraphics[width=0.8\textwidth]{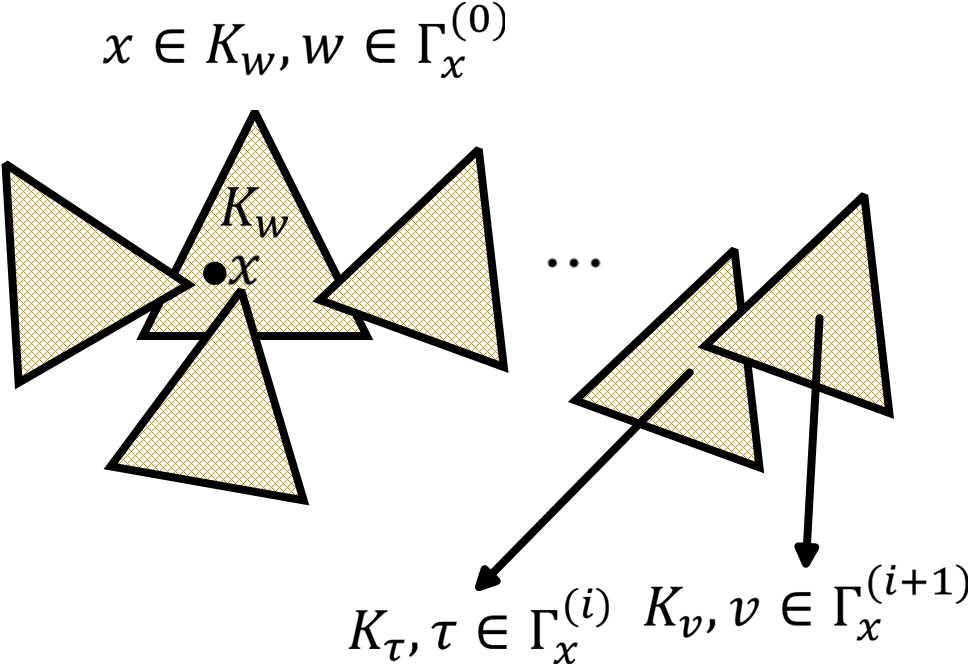}
			\caption{An illustration of $\Gamma_x^{(i)}$.}
			\label{Rl1}
		\end{minipage}
		\begin{minipage}[t]{0.48\textwidth}
			\centering
			\includegraphics[width=0.4\textwidth]{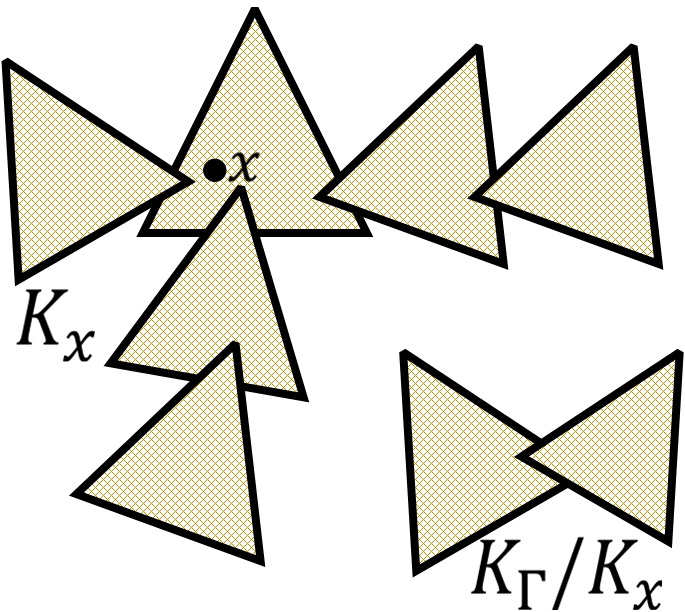}
			\caption{An illustration of $K_x$.}
			\label{Rl2}
		\end{minipage}
	\end{figure}
	
	By definition, it is obvious that 
	\begin{equation*}
		\Gamma_x^{(0)}\subset\Gamma_x^{(1)}\subset\cdots\subset\Gamma_x^{(i)}\subset%
		\Gamma_x^{(i+1)}\subset\cdots\subset\Gamma.
	\end{equation*}
	Therefore, since $\#\Gamma\leq L$ by $(\mathrm{WOC}_{\mathbf{\theta}})$,
	there exists a unique integer $0\leq i_x\leq L$ such that 
	\begin{equation*}
		\Gamma_x^{(i_x-1)}\subsetneqq\Gamma_x^{(i_x)}=\Gamma_x^{(i_x+k)}
	\end{equation*}
	for any $k\geq 0$, where $\Gamma_x^{(-1)}:=\emptyset$. Define the maximal
	connectivity branch containing $x$ by 
	\begin{equation*}
		K_x:=\tbigcup_{v\in\Gamma_x^{(i_x)}}K_v,
	\end{equation*}
	which is then a union of at most $L$ cells from $\Gamma$. See Figure \ref{Rl2}. It is easy to see that for any $x\in K_\Gamma$, $K_x$ is unique and compact.
	
	It suffices to prove $B(x_0,\bar{R}r)\subset K_{x_0}$. On the contrary,
	suppose there exists $y\in B(x_0,\bar{R}r)\setminus K_{x_0}$. Then, it is
	easy to check that $K_{x_0}\cap K_y=\emptyset$. This way, $K_\Gamma\setminus
	K_{x_0}\supset K_y\ni y$ is compact. Therefore, 
	\begin{equation*}
	d_0:=d(K_{x_0},K_\Gamma\setminus K_{x_0})>0.
	\end{equation*}
	
	On the other hand, let $n$ be an arbitrary positive integer. By $(\mathrm{CH}%
	)$, there exists a sequence $\{x_{i}\}_{i=1}^{n}\subset M$ such that $%
	x_{n}=y $ and 
	\begin{equation*}
		d(x_{i-1},x_{i})\leq C_{\mathrm{CH}}\frac{d(x_{0},y)}{n}\quad (1\leq i\leq
		n).
	\end{equation*}%
	Hence for any $1\leq i<n$, 
	\begin{equation*}
		d(x_{0},x_{i})\leq \sum\limits_{i=0}^{i-1}d(x_{j},x_{j+1})\leq \frac{i}{n}C_{%
			\mathrm{CH}}d(x_{0},y)<C_{\mathrm{CH}}\bar{R}r.
	\end{equation*}%
	By the definition of $\Gamma $, it follows that $x_{i}\in K_{\Gamma }$.
	Meanwhile, $x_{0}\in K_{x_{0}}$ but $y=x_{n}\notin K_{x_{0}}$, showing that
	there exists an integer $1\leq j\leq n$ such that $x_{j-1}\in K_{x_{0}}$ but 
	$x_{j}\notin K_{x_{0}}$. Therefore, 
	\begin{equation*}
	d_0\leq d(x_{j-1},x_{j})\leq C_{%
			\mathrm{CH}}\frac{d(x_{0},y)}{n}\leq \frac{C_{\mathrm{CH}}\bar{R}r}{n},
	\end{equation*}%
	which contradicts $d_0>0$ by letting $n\rightarrow \infty $. Thus $%
	K_{x_{0}}=K_{y}\ni y$, completing the proof.
\end{proof}

We give an example where both conditions $(\mathrm{WOC}_{\mathbf{\rho}})$
and $(\mathrm{CP}_{\mathbf{\rho}})$ are satisfied but $(\mathrm{CH})$ fails.

\begin{example}
	\label{ex-1} Consider the Sierpi\'{n}ski Gasket $(K,\{F_i\}_{i=1}^3)$ in $%
	\mathbb{R}^2$. Let $|\cdot|$ be the Euclidean metric, then $%
	d_{1/2}:=|\cdot|^{1/2}$ is another metric. It is easy to see that each $F_i$
	is still a contracting similitude in $\mathbb{R}^2$ under $d_{1/2}$ with contraction ratio $%
	\sqrt{2}/2$. In particular,
	\begin{equation*}
		\rho_1=\left(\frac{1}{2},\frac{1}{2},\frac{1}{2}\right),\quad\rho_{1/2}=%
		\left(\frac{\sqrt{2}}{2},\frac{\sqrt{2}}{2},\frac{\sqrt{2}}{2}\right)
	\end{equation*}
	are respectively contraction tuples of $K$ with respect to metrics $|\cdot|$ and $d_{1/2}$. Using the fact that
	\begin{equation}  \label{square}
	B_{1/2}(x,r)=B(x,r^2)\quad\text{and}\quad\Lambda_{\rho_{1/2}}(r)=\Lambda_{\rho_1}(r^2)
	\end{equation}
	for any $r\in(0,1)$, we see $(\mathrm{WOC}_{\mathbf{\rho}_{1/2}})$ and $(\mathrm{CP}_{\mathbf{\rho}_{1/2}})$ with $C^{\prime }_2=1/2$ and $L^{\prime }_2=2$, since $(\mathrm{WOC}_{\mathbf{\rho}_1})$ holds by Lemma \ref{op'cond} and $K$ satisfies $(\mathrm{CP}_{\mathbf{\rho}_1})$ with $C_2=1/4$ and $L_2=2$ clearly.
	
	However, $(\mathrm{CH})$ fails on $(K,d_{1/2})$: actually, otherwise for any 
	$x,y\in K$ and $n>0$, there exists a chain $\{x_{i}\}_{i=1}^{n}$ in $K$ such
	that $x_{0}=x$, $x_{n}=y$ and 
	\begin{equation*}
		d_{1/2}(x_{i-1},x_{i})\leq \frac{C_{\mathrm{CH}}}{n}d_{1/2}(x,y)
	\end{equation*}%
	for each $1\leq i\leq n$. This would imply
	\begin{equation*}
		|x-y|\leq
		\sum_{i=1}^{n}|x_{i-1}-x_{i}|=\sum_{i=1}^{n}d_{1/2}(x_{i-1},x_{i})^{2}\leq
		\sum \frac{C_{\mathrm{CH}}^{2}d_{1/2}(x,y)^{2}}{n^{2}}=\frac{C^{2}|x-y|}{n}
	\end{equation*}%
	for any $n>0$, which yields a contradiction as $n\rightarrow \infty $.
\end{example}

Now we give another example where $(\mathrm{CH})$ and $(\mathrm{CP}_{\mathbf{\rho}})$ are satisfied but $(\mathrm{WOC}_{\mathbf{\rho}})$ fails.

\begin{example}
	\label{ex3} On $\mathbb{R}$, take $q_{1}=0,q_{2}=1/2,q_{3}=1$ and consider
	the following maps: 
	\begin{equation*}
		F_{i}(x)=(x+q_{i})/2,\ \ i=1,2,3.
	\end{equation*}%
	Then the attractor of the IFS $\{F_{i}\}_{i=1}^{3}$ is the closed interval $%
	[0,1]$. Obviously $(\mathrm{CH})$ and $(\mathrm{CP}_{\mathbf{\rho }})$ hold.
	
	However, by mathematical induction, we have $j_{n,c}=2n+1$ for any $n\geq 1$ and for
	any $c\in(0,1)$, where
	\begin{equation*}
		j_{n,c}:=\#\big\{w\in I^n:d(1/2,K_w)<c2^{-n}\big\}.
	\end{equation*}
Hence $(\mathrm{WOC}_{\mathbf{\rho}})$ fails in this case.
\end{example}

\subsection{Proof of Proposition \ref{ML4}}

\label{subsec:construction}

We now construct the scaling function $H$ as asserted in Proposition \ref{ML4}. For
convenience, throughout this subsection we denote by $C_{\mathrm{%
AV}}$ the constant $C$ appearing in part (c) of Proposition \ref{av=}.

\begin{proof}[Proof of Proposition \ref{ML4}]
The proof is here constructive and divided into several steps.

\textbf{Step 1. Construction of an auxiliary function $H_0$.} Let $x \in K$ and fix an address $\omega = \omega_1\omega_2\cdots \in
I^\ast$ of $x$. For $n \ge 0$, let $\omega(n) =
\omega_1\cdots\omega_n$ (with $\omega(0)=\varnothing$). Define $H_0(x,
\cdot) $ as follows:

\begin{itemize}
\item $H_0(x,0)=0$.

\item For $r \ge 1$, set $H_0(x,r)=r^{a_\ast}$, where $a_\ast$ is given by \eqref{thet'a}.

\item For $r \in [\theta_{\omega(n)}, \theta_{\omega(n-1)})$ with $n \ge 1$,
define 
\begin{equation}  \label{Hxr}
H_0(x,r) = \frac{\zeta_{\omega(n-1)} - \zeta_{\omega(n)}}{%
\theta_{\omega(n-1)} - \theta_{\omega(n)}} (r - \theta_{\omega(n)}) +
\zeta_{\omega(n)}.
\end{equation}
\end{itemize}

Clearly, $H_0(x,\cdot)$ is continuous, strictly increasing, and satisfies 
\begin{equation*}
H_0(x,0)=0\quad\text{and}\quad\lim_{r\to\infty} H_0(x,r)=\infty.
\end{equation*}

\textbf{Step 2. Local control of $H_0$ by $\zeta_w$.} We show that there exists a constant $C_0 > 0$ such that for any $x \in K$, $%
r \in (0,1)$, and $w \in \Lambda_{\mathbf{\theta}}(r)$ with $x \in K_w$, 
\begin{equation}  \label{eq:H0_zeta_control}
C_0^{-1} \zeta_w \le H_0(x,r) \le C_0 \zeta_w.
\end{equation}
Indeed, let $n_0 \ge 0$ be such that $r \in [\theta_{\omega(n_0+1)},
\theta_{\omega(n_0)})$. Then $\omega(n_0+1) \in \Lambda_{\mathbf{\theta}}(r)$
and $x \in K_{\omega(n_0+1)}$. Since also $x \in K_w$, we have $K_w \cap
K_{\omega(n_0+1)} \neq \emptyset$. By (c) of Proposition \ref{av=} (with
constant $C_{\mathrm{AV}}$), we have 
\begin{equation}  \label{won}
C_{\mathrm{AV}}^{-1} \zeta_w \le \zeta_{\omega(n_0+1)} \le C_{\mathrm{AV}}
\zeta_w.
\end{equation}
On the other hand, from the definition of $H_0$, 
\begin{equation}\label{H0def}
\zeta_{\omega(n_0+1)} \le H_0(x,r) \le \zeta_{\omega(n_0)} \le
\zeta_{\min}^{-1} \zeta_{\omega(n_0+1)}.
\end{equation}
Combining \eqref{won}, we obtain \eqref{eq:H0_zeta_control} with $C_0 = C_{\mathrm{AV}} \zeta_{\min}^{-1}$.

\textbf{Step 3. Pointwise scaling property of $H_0$.} We show that, under
the assumptions of Proposition \ref{ML4}, there exists a constant $C_1 \ge 1$ such
that for any $x \in K$ and $0 < r \le R < \infty$, 
\begin{equation}  \label{eq:pointwise_scaling}
C_1^{-1} \left( \frac{R}{r} \right)^{a_\ast} \le \frac{H_0(x,R)}{H_0(x,r)}
\le C_1 \left( \frac{R}{r} \right)^{a^\ast}.
\end{equation}

For this, we distinguish three cases.

\noindent\textbf{Case (1): $r \ge 1$.} Then $H_0(x,R)=R^{a_\ast}$ and $%
H_0(x,r)=r^{a_\ast}$, so \eqref{eq:pointwise_scaling} holds with $C_1=1$.

\noindent\textbf{Case (2): $R < 1$.} By the construction of $H_0$, we can
choose $w \in \Lambda_{\mathbf{\theta}}(R)$ and $v \in \Lambda_{\mathbf{%
\theta}}(r)$ such that $x \in K_w \cap K_v=K_v$ (i.e., $w$ is the ancestor of 
$v$), and 
\begin{equation*}
H(x,\theta_w)=\zeta_w,\quad H(x,\theta_v)=\zeta_v.
\end{equation*}
Since $w$ is the ancestor of $v$, we have $v = w\tau$ for some finite word $\tau$. By \eqref{H0def}, we have 
\begin{equation*}
\zeta_{\min} \frac{1}{\zeta_\tau}= \zeta_{\min} \frac{\zeta_w}{\zeta_v}\le 
\frac{H_0(x,R)}{H_0(x,r)} \le \zeta_{\min}^{-1} \frac{\zeta_w}{\zeta_v}%
=\zeta_{\min}^{-1} \frac{1}{\zeta_\tau}.
\end{equation*}
On the other hand, from \eqref{part2.1}, we have 
\begin{equation*}
\theta_{\min} \frac{r}{R} \le \theta_\tau \le \theta_{\min}^{-1} \frac{r}{R}.
\end{equation*}
Using the power bounds \eqref{part2.3}, we have $\zeta_\tau \le
(\theta_\tau)^{a_\ast}$ and $\zeta_\tau \ge (\theta_\tau)^{a^\ast}$.
Therefore, 
\begin{equation*}
\frac{1}{\zeta_\tau} \le \frac{1}{(\theta_\tau)^{a^\ast}} \le \left( \frac{R%
}{r} \right)^{a^\ast} (\theta_{\min})^{-a^\ast}\quad\text{and}\quad\frac{1}{\zeta_\tau} \ge \frac{1}{(\theta_\tau)^{a_\ast}} \ge \left( \frac{R%
}{r} \right)^{a_\ast} (\theta_{\min})^{a_\ast}.
\end{equation*}
Combining these estimates, we obtain 
\begin{equation}
\zeta_{\min} (\theta_{\min})^{a_\ast} \left( \frac{R}{r} \right)^{a_\ast}
\le \frac{H_0(x,R)}{H_0(x,r)}\le \zeta_{\min}^{-1} (\theta_{\min})^{-a^\ast}
\left( \frac{R}{r} \right)^{a^\ast}.
\end{equation}
Thus, in this case, \eqref{eq:pointwise_scaling} holds with $C_1 = \zeta_{\min}^{-1} (\theta_{\min})^{-a^\ast}$.

\noindent\textbf{Case (3): $r < 1 \le R$.} Write 
\begin{equation*}
\frac{H_0(x,R)}{H_0(x,r)} = \frac{H_0(x,R)}{H_0(x,1)} \cdot \frac{H_0(x,1)}{H_0(x,r)}=R^{a_\ast}\lim\limits_{R\uparrow 1}\frac{H_0(x,R)}{H_0(x,r)}.
\end{equation*}
By Case (2), it follows that
\begin{equation*}
\frac{H_0(x,R)}{H_0(x,r)} \le C_1 R^{a_\ast} \left( \frac{1}{r}
\right)^{a^\ast} = C_1 \left( \frac{R}{r} \right)^{a^\ast} R^{a_\ast -
a^\ast} \le C_1 \left( \frac{R}{r} \right)^{a^\ast},
\end{equation*}
since $R \ge 1$ and $a_\ast \le a^\ast$. On the other hand, it easy to see
that 
\begin{equation*}
\frac{H_0(x,R)}{H_0(x,r)} \ge C_1^{-1} R^{a_\ast} \left( \frac{1}{r}
\right)^{a_\ast} = C_1^{-1} \left( \frac{R}{r} \right)^{a_\ast}.
\end{equation*}
Thus, \eqref{eq:pointwise_scaling} holds with the same $C_1$ as in Case (2).

In conclusion, \eqref{eq:pointwise_scaling} holds with the constant
\begin{equation}
C_1 = \zeta_{\min}^{-1} (\theta_{\min})^{-a^\ast}.  \label{HC1}
\end{equation}

\textbf{Step 4. Comparability of $H_0$ at the same scale.}
We show that, under the assumptions of Proposition \ref{ML4}, there exists a constant $C_2 > 0$ such that for any $R > 0$ and $x, y \in K$ with $d(x,y) \le \bar{R}%
R $, 
\begin{equation}  \label{eq:comparison_same_scale}
C_2^{-1} \le \frac{H_0(x,R)}{H_0(y,R)} \le C_2.
\end{equation}

If $R \ge 1$, then $H_0(x,R)=H_0(y,R)=R^{a_\ast}$, so %
\eqref{eq:comparison_same_scale} trivially holds with $C_2=1$.

Now assume $R < 1$. Let $C_{\mathrm{CP}}$, $L_2$ be the parameters from condition $(\mathrm{CP}_{\mathbf{\theta},\bullet})$. We distinguish two cases.

\noindent\textbf{Case (1): $d(x,y) < C_{\mathrm{CP}} \bar{R} R$.} By $(\mathrm{%
CP}_{\mathbf{\theta},C_{\mathrm{CP}}})$, there exists a chain $%
\{w^{(i)}\}_{i=1}^{L_2} \subset \Lambda_{\mathbf{\theta}}(R)$ such that 
\begin{equation*}
x \in K_{w^{(1)}},\quad y \in K_{w^{(L_2)}},\quad\text{and}\quad K_{w^{(i-1)}} \cap
K_{w^{(i)}} \neq \emptyset\quad\text{for}\quad 2 \le i \le L_2.
\end{equation*}
Applying \eqref{eq:H0_zeta_control} to $w^{(1)}$ and $w^{(L_2)}$, and using
(c) of Proposition \ref{av=} for every adjacent pair, we obtain 
$$H_0(x,R)\le \zeta_{\min}^{-1} \zeta_{w^{(1)}} \le \zeta_{\min}^{-1} C_{\mathrm{AV}} \zeta_{w^{(2)}} \le \cdots \le \zeta_{\min}^{-1} C_{\mathrm{AV}}^{L_2-1} \zeta_{w^{(L_2)}}\le \zeta_{\min}^{-2} C_{\mathrm{AV}}^{L_2-1} H_0(y,R).$$
Similarly, $H_0(y,R) \le \zeta_{\min}^{-2} C_{\mathrm{AV}}^{L_2-1} H_0(x,R)$. Thus, in this case, \eqref{eq:comparison_same_scale} holds with $C_2 =
\zeta_{\min}^{-2} C_{\mathrm{AV}}^{L_2-1}$.

\noindent\textbf{Case (2): $C_{\mathrm{CP}} \bar{R} R \le d(x,y) \le \bar{R} R$%
.} Let $R^{\prime }= R/(C_{\mathrm{CP}}/2)$. Then $R^{\prime }>R$ and $%
d(x,y) < 2\bar{R} R^{\prime }= C_{\mathrm{CP}} \bar{R} R^{\prime }$. If $%
R^{\prime }\ge 1$, then $H_0(x,R^{\prime }) = H_0(y,R^{\prime }) =
(R^{\prime a_\ast}$. If $R^{\prime }< 1$, then by Case 1 (applied at scale $%
R^{\prime }$), we have $H_0(x,R^{\prime }) \asymp H_0(y,R^{\prime })$ with
constant $\zeta_{\min}^{-2} C_{\mathrm{AV}}^{L_2-1}$.

Now write 
\begin{equation*}
\frac{H_0(x,R)}{H_0(y,R)} = \frac{H_0(x,R)}{H_0(x,R^{\prime })} \cdot \frac{%
H_0(y,R^{\prime })}{H_0(y,R)} \cdot \frac{H_0(x,R^{\prime })}{%
H_0(y,R^{\prime })}.
\end{equation*}
By the conclusion in Step 3, the first two factors are bounded by 
\begin{equation*}
\frac{H_0(x,R)}{H_0(x,R^{\prime })} \le C_1 \left( \frac{R}{R^{\prime }}
\right)^{a_\ast} = C_1 (C_{\mathrm{CP}}/2)^{a_\ast},
\end{equation*}
and 
\begin{equation*}
\frac{H_0(y,R^{\prime })}{H_0(y,R)} \le C_1 \left( \frac{R^{\prime }}{R}
\right)^{a^\ast} = C_1 (C_{\mathrm{CP}}/2)^{-a^\ast}.
\end{equation*}
The third factor is bounded by $\zeta_{\min}^{-2} C_{\mathrm{AV}}^{L_2-1}$
(if $R^{\prime }<1$) or $1$ (if $R^{\prime }\ge1$). Therefore, 
\begin{equation*}
\frac{H_0(x,R)}{H_0(y,R)} \le C_1^2 (C_{\mathrm{CP}}/2)^{a_\ast - a^\ast}
\cdot \max(1, \zeta_{\min}^{-2} C_{\mathrm{AV}}^{L_2-1})=\zeta_{\min}^{-4}(%
\theta_{\min})^{-2a^\ast}C_{\mathrm{AV}}^{L_2-1}(C_{\mathrm{CP}}/2)^{a_\ast
- a^\ast},
\end{equation*}
and a similar lower bound holds by exchanging the position of $x$ and $y$.
Thus, in this case, \eqref{eq:comparison_same_scale} holds with constant $%
C_2 $ as in Case (1).

Combining both cases, we prove what we desired with
\begin{equation}  \label{HC2}
	C_2=\zeta_{\min}^{-4}(\theta_{\min})^{-2a^\ast}C_{\mathrm{AV}}^{L_2-1}(C_{\mathrm{CP}}/2)^{a_\ast - a^\ast}.
\end{equation}

\textbf{Step 5. Scaling property of $H_0$ and definition of $H$.}
Following Steps 3 and 4, we see for any $0 < r \le R <
\infty$ and $x,y \in K$ with $d(x,y) \le \bar{R}R$, 
\begin{equation}  \label{H0ScalUE}
\frac{H_0(x,R)}{H_0(y,r)} = \frac{H_0(x,R)}{H_0(y,R)} \cdot \frac{H_0(y,R)}{%
H_0(y,r)} \le C_2 \cdot C_1 \left( \frac{R}{r} \right)^{a^\ast} = C_1 C_2
\left( \frac{R}{r} \right)^{a^\ast}.
\end{equation}
Similarly, 
\begin{equation*}
\frac{H_0(x,R)}{H_0(y,r)} \ge (C_1 C_2)^{-1} \left( \frac{R}{r}
\right)^{a_\ast}.
\end{equation*}
Thus, $H_0$ satisfies the scaling property with constant $C_{H_0} = C_1 C_2$%
, where $C_1, C_2$ are constants in \eqref{HC1}, \eqref{HC2} respectively.

Now define
\begin{equation*}
H(x,r) = H_0(x, r/\bar{R}) \quad\text{for}\quad 0\le r<\infty
\end{equation*}
and $H(x,\infty)=\infty$. We claim that $H$ is the scaling function we
desired.

Indeed, by definition, for any $x\in K$%
, $H(x,\cdot)$ is continuous and strictly increasing with respect to $r$ and
satisfies $H(x_0,0)=0, H(x,\infty)=\infty$. Besides, $H$ also satisfies the
scaling property with sub- and super-scaling exponential $a_\ast, a^\ast$ in %
\eqref{thet'a} and the constant $C_H=C_{H_0}=C_1C_2$.

Finally, for $r \in (0,\bar{R})$ and $w
\in \Lambda_{\mathbf{\theta}}(r/\bar{R})$ with $x \in K_w$, by \eqref{eq:H0_zeta_control}, obviously
\begin{equation*}
C_0^{-1} \zeta_w \le H(x,r) \le C_0 \zeta_w,
\end{equation*}
which proves ii) of Proposition \ref{ML4}.

\textbf{Step 6. Sharpness of the exponents.}
Without loss of the generality, let $\frac{\ln \zeta _{1}}{\ln \theta _{1}}
=a^{\ast }$ and let $x_{1}$ be the fixed point of the mapping $F_{1}$. Then $%
x_{1}\in K_{1^{(m)}}$ for every integer $m\geq 1$. Hence 
\begin{equation}
H(x_{1},\theta _{1^{(m)}})\leq C_{\mathrm{AV}}\zeta _{\min }^{-1}\zeta
_{1^{(m)}}=C_{\mathrm{AV}}\zeta _{\min }^{-1}\zeta _{1}^{m}=C_{\mathrm{AV}
}\zeta _{\min }^{-1}(\theta _{1})^{ma^{\ast }}.  \label{hatctrl}
\end{equation}
For any upper exponent $\hat{a}^{\ast }$, by definition, 
\begin{equation*}
\frac{H(x_{1},1)}{H(x_{1},\theta _{1^{(m)}})}\leq C_{H}\left( \frac{1}{
\theta _{1}^{m}}\right) ^{\hat{a}^{\ast }}=C_{H}\theta _{1}^{-m\hat{a}^{\ast
}}
\end{equation*}
Combining (\ref{hatctrl}), we see 
\begin{equation*}
\theta _{1}^{m(\hat{a}^{\ast }-a^{\ast })}\leq C_{H}C_{\mathrm{AV}}\zeta
_{\min }^{-1}\quad \text{for all}\quad m\geq 1,
\end{equation*}
which is only possible when $\hat{a}^{\ast }\geq a^{\ast }$. That is, $%
a^{\ast }$ is the best super-scaling exponent. Similarly $a_{\ast }$ is the
best sub-scaling exponent.
\end{proof}

\subsection{Doubling property of self-similar measures}

\label{subsec:application}

We now investigate self-similar measures with the help of Proposition \ref{ML4}. Let $\mu$ be the
self-similar measure on $K$ associated with a probability weight $\mathbf{p}
= (p_i)_{i=1}^N$ (where $\sum_{i=1}^Np_i=1$), that is,
$$\mu = \sum_{i=1}^N p_i \mu \circ F_i^{-1}.$$
The existence of such measures is checked in \cite{Hutchinson.1981.IUMJ713}. We aim to find suitable conditions so that $\mu$ is doubling (in other words, its volume
function $V(x,r) = \mu(B(x,r))$ satisfies the scaling property). To avoid ambiguity, in this subsection we consider $K$ as a subset of the whole space $M$ (where $F_i$'s are defined), and $\mu$ as a measure on $M$ supported on $K$.

We first give a lower estimate on the measure of an arbitrary ball, which can
be deduced without any extra geometric assumption.

\begin{proposition}
\label{muLctrl} There exists a constant $C>0$ such that 
\begin{equation}  \label{mulctrl}
V(x,r)\geq C^{-1}p_w
\end{equation}
for every $x\in K, r\in(0,\bar{R})$ and every word $w\in\Lambda_{\mathbf{\rho}
}(r/\bar{R})$ with $x\in K_w$, where $\bar{R}>0$ is the diameter of $K$.
\end{proposition}

\begin{proof}
Fix $x\in K$, $r\in(0,\bar{R})$ and a word $w\in\Lambda_{\mathbf{\rho}}(r/ 
\bar{R})$ with $x\in K_w$. Then, there exists a word $w^{\prime
}\in\Lambda_{ \mathbf{\rho}}(r/2\bar{R})$ such that $x\in K_{w^{\prime }}$
and $w^{\prime }=w\tau$ for some finite word $\tau$.

Note that $K_{w^{\prime }}\subset B(x,r)$, since for any $y\in K_{w^{\prime
}}$, we directly have 
\begin{equation*}
d(x,y)\leq \rho _{w^{\prime }}d(F_{w^{\prime }}^{-1}(x),F_{w^{\prime
}}^{-1}(y))\leq \rho _{w^{\prime }}\bar{R}\leq \frac{r}{2}.
\end{equation*}

Applying (\ref{part2.2}) to $\tau $ with respect to $\mathbf{\rho }$, we see 
$|\tau |\leq \frac{\ln (\rho _{\min }/2)}{\ln \rho _{\max }}$. Consequently, 
\begin{equation}
V(x,r)\geq \mu (K_{w^{\prime }})\geq p_{w^{\prime }}=p_{w}p_{\tau }\geq
p_{w}p_{\min }^{|\tau |}\geq p_{w}(p_{\min })^{\frac{\ln (\rho _{\min }/2)}{
\ln \rho _{\max }}},  \label{mulct2}
\end{equation}
that is, (\ref{mulctrl}) holds with $C=(p_{\min })^{-\frac{\ln (\rho _{\min
}/2)}{\ln \rho _{\max }}}$.
\end{proof}

Upper estimate on the measure of a ball is more complicated, and we give a useful lemma before that.

\begin{lemma}
\label{MCP} Assume that $K$ satisfies $(\mathrm{WOC}_{\mathbf{\rho }})$, and $\mathbf{p}$ satisfies $(\mathrm{AV}_{\mathbf{\rho}})$. Then, there exists a constant $C_\mu>0$ such that for every finite word $\tau$, 
\begin{equation}  \label{mcp}
\mu(K_\tau)\leq C_\mu p_\tau.
\end{equation}
\end{lemma}

\begin{proof}
For each $\tau $ and any $x_{\tau }\in K_{\tau }$, set 
\begin{equation*}
\Gamma _{\tau }=\Gamma _{\mathbf{\rho }}(K_{\tau },\rho _{\tau })\quad\text{and}\quad \Gamma _{x_{\tau }}=\Gamma _{\mathbf{\rho }}\left( B(x_{\tau
},2\rho _{\tau }\bar{R}),\rho _{\tau }\right) .
\end{equation*}
Since $K_{\tau }\subset B(x_{\tau },2\rho _{\tau }\bar{R})$ for any $x_{\tau
}\in K_{\tau }$, it follows that $\Gamma _{\tau }\subset \Gamma _{x_{\tau }}$. Applying $(\mathrm{WOC}_{\mathbf{\rho }})$ at level $\rho _{\tau }$ on $%
B(x_{\tau },2\rho _{\tau }\bar{R})$, we have $\#\Gamma _{x_{\tau }}\leq
L_{1} $. Hence
\begin{equation}
\#\Gamma _{\tau }\leq \#\Gamma _{x_{\tau }}\leq L_{1}.  \label{GamTL}
\end{equation}

On one hand, for any $v\in \Gamma _{\tau }$, since $v\in $ $\Lambda _{ 
\mathbf{\rho }}(\rho _{\tau })$ and $K_{v}\cap K_{\tau }\neq \emptyset $ by
definition of $\Gamma _{\tau }$, it follows by $(\mathrm{AV}_{\mathbf{\rho }
})$ that 
\begin{equation}
p_{v}\leq C_{\ast }p_{\tau },  \label{vinGamT}
\end{equation}
where $C_{\ast }>0$ is a constant independent of words $v,\tau $.

On the other hand, for any $v^{\prime }\in \Lambda _{\mathbf{\rho }}(\rho
_{\tau })\setminus \Gamma _{\tau }$, we show that $F_{v^{\prime
}}^{-1}(K_{\tau })\cap K=\emptyset $. Otherwise, there would exist a point $%
x^{\prime }\in F_{v^{\prime }}^{-1}(K_{\tau })\cap K$, and thus by
definition of both $F_{v^{\prime }}^{-1}(K_{\tau })$ and $K_{v^{\prime }}$, $F_{v^{\prime }}(x^{\prime })\in K_{\tau }\cap K_{v^{\prime }}$,
implying that $v^{\prime }\in \Gamma _{\tau }$, a contradiction. Further,
since $\supp(\mu)=K$, it follows that 
\begin{equation}
\mu (F_{v^{\prime }}^{-1}(K_{\tau }))=\mu \left( F_{v^{\prime }}^{-1}(K_{\tau
})\cap K\right) =0\quad\text{for any}\quad v^{\prime }\in \Lambda _{\mathbf{\rho }
}(\rho _{\tau })\setminus \Gamma _{\tau }.  \label{voutGamT}
\end{equation}

Therefore, we have by self-similarity, (\ref{voutGamT}), (\ref{vinGamT}) and
(\ref{GamTL}) that 
\begin{eqnarray*}
\mu (K_{\tau }) &=&\sum_{v\in \Lambda _{\mathbf{\rho }}(\rho _{\tau
})}p_{v}\mu (F_{v}^{-1}(K_{\tau }))=\sum_{v\in \Gamma _{\tau }}p_{v}\mu
(F_{v}^{-1}(K_{\tau })) \\
&\leq &\sum_{v\in \Gamma _{\tau }}p_{v}\leq \sum_{v\in \Gamma _{\tau
}}C_{\ast }p_{\tau }=\#\Gamma _{\tau }\cdot C_{\ast }p_{\tau }\leq
CL_{1}p_{\tau },
\end{eqnarray*}
which proves (\ref{mcp}) with $C_{\mu }=C_{\ast }L_{1}$.
\end{proof}

\begin{proposition}
\label{muUctrl} Assume that $K$ satisfies conditions $(\mathrm{WOC}_{\mathbf{\rho}})$ and $(\mathrm{CP}_{\mathbf{\rho},C_2})$ for some constant $%
C_2\in(0,1)$. If $\mathbf{p}$ satisfies $(\mathrm{AV}_{\mathbf{\rho}})$,
then there exists a constant $C>0$ such that 
\begin{equation}  \label{muuctrl}
V(x,r)\leq Cp_w
\end{equation}
for any $x\in K$, $r\in(0,\bar{R})$ and any word $w\in\Lambda_{\mathbf{\rho}
}(r/\bar{R})$ with $x\in K_w$.
\end{proposition}

\begin{proof}
Fix $x\in K$, $r\in (0,\bar{R})$ and $w\in \Lambda _{\mathbf{\rho }}(r/\bar{%
R	})$ with $x\in K_{w}$. We distinguish two cases.

Case $(1)$: $C_{2}\bar{R}\leq r<\bar{R}$. In this case,
\begin{equation*}
\rho _{\max }^{|w|}\geq \rho _{w}>\rho _{\min }\frac{r}{\bar{R}}\geq
C_{2}\rho _{\min }
\end{equation*}
by (\ref{part1}), showing that 
\begin{equation*}
|w|\leq \frac{\ln C_{2}\rho _{\min }}{\ln \rho _{\max }}.
\end{equation*}
Since $\mu$ is a probability measure, it follows that 
\begin{equation*}
V(x,r)\leq 1=\frac{p_{w}}{p_{w}}\leq p_{\min }^{-|w|}p_{w}\leq p_{\min }^{- 
\frac{\ln C_{2}\rho _{\min }}{\ln \rho _{\max }}}p_{w},
\end{equation*}
thus proving (\ref{muuctrl}) with $C=p_{\min }^{-\frac{\ln C_{2}\rho _{\min
} }{\ln \rho _{\max }}}$ in this case.

Case $(2)$: $r<C_{2}\bar{R}$. We show that, there exists a constant $C_{\sharp }>0$ such that 
\begin{equation}
p_{v}\leq C_{\sharp }p_{w}\quad \text{for all }\ v\in \Gamma :=\Gamma _{\mathbf{\rho }}\left(B(x,r),r/(C_{2}\bar{R})\right).  \label{pv<w}
\end{equation}

Indeed, since $B(x,r)\subset \cup _{v\in \Gamma
}K_{v}$, then by Lemma \ref{MCP}, 
\begin{equation}
V(x,r)\leq \sum_{v\in \Gamma }\mu (K_{v})\leq C_{\mu }\sum_{v\in \Gamma
}p_{v}.  \label{balcel}
\end{equation}
Fix an arbitrary $v\in \Gamma$ and take $x_{v}\in K_{v}\cap B(x,r)$.
Applying $(\mathrm{CP}_{\mathbf{\rho },C_{2}})$ over the partition $\Lambda
_{\mathbf{\rho }}(r/(C_{2}\bar{R}))$, we find a chain $%
\{v^{(m)}\}_{m=1}^{L_{2}}\subset \Lambda _{\mathbf{\rho }}(r/(C_{2}\bar{R}))$
such that 
\begin{equation*}
x\in K_{v^{(1)}},\quad x_{v}\in K_{v^{(L_{2})}}\quad \text{and}\quad
K_{v^{(i-1)}}\cap K_{v^{(i)}}\neq \emptyset \quad\text{for}\quad 1<i\leq L_{2}.
\end{equation*}
Since $\Lambda _{\mathbf{\rho }}(r/\bar{R})$ is a refinement of $\Lambda _{ 
\mathbf{\rho }}(r/(C_{2}\bar{R}))$, one finds an ancestor $v^{(0)}\in
\Lambda _{\mathbf{\rho }}(r/(C_{2}\bar{R}))$ of $w\in \Lambda _{\mathbf{\rho 
}}(r/\bar{R})$ and a finite word $\tau $ such that $w=v^{(0)}\tau $. Thus $x\in K_{v^{(0)}}\cap K_{v^{(1)}}$. Since $x_{v}\in
K_{v}\cap K_{v^{(L_{2})}}\neq \emptyset $, by $(\mathrm{AV}_{\mathbf{\rho }
}) $ of tuple $\mathbf{p}$ we obtain 
\begin{equation*}
p_{v}\leq C_{\ast }p_{v^{(L_{2})}}\leq C_{\ast }^{2}p_{v^{(L_{2}-1)}}\leq
\cdots \leq C_{\ast }^{L_{2}}p_{v^{(1)}}\leq C_{\ast }^{L_{2}+1}p_{v^{(0)}},
\end{equation*}
where $C_{\ast }>0$ is a constant independent of $v,w$ and $x,r$. Since $|\tau |\leq $ $\frac{\ln C_{2}\rho _{\min }}{\ln \rho
_{\max }}$ by \eqref{part2.3}, it follows that
\begin{equation*}
p_{v}\leq C_{\ast }^{L_{2}+1}p_{v^{(0)}}=C_{\ast }^{L_{2}+1}\frac{p_{w}}{
p_{\tau }}\leq C_{\ast }^{L_{2}+1}p_{\min }^{-\frac{\ln C_{2}\rho _{\min }}{
\ln \rho _{\max }}}p_{w},
\end{equation*}
thus proving (\ref{pv<w}) with $C_{\sharp }=C_{\ast }^{L_{2}+1}p_{\min }^{- 
\frac{\ln C_{2}\rho _{\min }}{\ln \rho _{\max }}}$.

Finally, note that $\#\Gamma \leq L_{1}$ by $(\mathrm{WOC}_{\mathbf{\rho }})$
. Combining this with \eqref{pv<w}, \eqref{balcel}, we obtain 
\begin{equation*}
V(x,r)\leq C_{\mu }\sum_{v\in \Gamma }p_{v}\leq C_{\mu }\sum_{v\in \Gamma
}C_{\sharp }p_{w}\leq C_{\mu }C_{\sharp }L_{1}p_{w},
\end{equation*}
thus proving (\ref{muuctrl}) with $C=C_{\mu }C_{\star }L_{1}$ in this case. The proof is complete.
\end{proof}

\begin{corollary}
\label{muScal} Under the same assumptions as in Proposition \ref{muUctrl}, the volume function $V(x,r)$ satisfies the scaling property,
that is, there exists a constant $C>0$ such that 
\begin{equation}  \label{muscal}
C^{-1}\left(\frac{R}{r}\right)^{\alpha_\ast}\leq\frac{V(x,R)}{V(y,r)}\leq
C\left(\frac{R}{r}\right)^{\alpha^\ast},
\end{equation}
for all $0<r\leq R\leq\bar{R}$ and $x,y\in K$ with $d(x,y)\leq R$, where 
\begin{equation*}
\alpha_\ast:=\min_{1\leq i\leq N}\frac{\ln p_i}{\ln\rho_i}
,\quad\alpha^\ast:=\max_{1\leq i\leq N}\frac{\ln p_i}{\ln\rho_i}.
\end{equation*}
In particular, $\mu$ satisfies $(\mathrm{VD})$ and $(\mathrm{RVD})$.
\end{corollary}

\begin{proof}
Using Propositions \ref{muLctrl} and \ref{muUctrl}, we see $V(x,r)\asymp
p_{w}$ for all $x\in K$, $r\in (0,\bar{R})$ and $w\in \Lambda _{\mathbf{\rho 
}}(r/\bar{R})$ with $x\in K_{w}$. Meanwhile, by Proposition \ref{ML4} with $%
\mathbf{\theta }=\mathbf{\rho}$, $\mathbf{\zeta }=\mathbf{p}$, there exists a
scaling function $H_{\mathbf{p}}$ such that $H_{\mathbf{p}}(x,r)\asymp p_{w}$
. Thus $V\asymp H_{\mathbf{p}}$. Moreover, (\ref{muscal}) obviously holds by
the scaling property of $V\asymp H_{\mathbf{p}}$.
\end{proof}

\section{Spatially Inhomogeneous Scaling on Rotated Triangles}
\label{Kl}

Now we combine ideas in Sections \ref{mainresults} and \ref{SSSWay} on a concrete non-trivial example, so that there is an explicitly constructed inhomogeneous scaling function and (both general and optimal) heat kernel estimates are given on it.

Let us consider the family of rotated triangle fractals $\{K_{\lambda}\}_{\lambda\in(0,1/2)}$, which was first introduced by Barlow in \cite{Barlow.1998.1}. This model is ideal for our purpose due to:
\begin{enumerate}
\item its inherent inhomogeneity from distinct contraction ratios, which is the germ of position-dependent scaling;
\item the analytic tractability provided by Cao's continuous family of Dirichlet forms in \cite{Cao.2023.AG};
\item some balance between explicity and richness of the geometry structure.
\end{enumerate}

By applying our general theory to $K_{\lambda}$, this section achieves two interconnected goals:
\begin{itemize}
\item the first explicit construction of an intrinsic $W(x,r)$ that essentially depends on $x$;
\item the derivation of optimal two-sided heat kernel estimates for a wide range of parameters (see Theorem \ref{le1-}) from detailed geometry analysis.
\end{itemize}

\subsection{Preliminary observations on geometry}
\label{subsec:geo_observations}

With any $0<\lambda <1/2$, we construct the rotated triangle $K_{\lambda }$ as the attractor of the IFS $\{F_{1},F_{2},F_{3},F_{4,\lambda }\}$ on $\mathbb{R}^2$ given by 
\begin{equation*}
F_{i}(x)=\frac{1}{2}(x+q_{i}),\quad i=1,2,3,
\end{equation*}%
where $q_{1}=(\frac{1}{2},\frac{\sqrt{3}}{2})$, $q_{2}=(0,0)$ and $%
q_{3}=(1,0)$ are the vertices of the closed unit triangle region $%
\blacktriangle $, and 
\begin{equation*}
F_{4,\lambda }(x)=\frac{1}{4}(x_{1},x_{2})%
\begin{pmatrix}
1 & \sqrt{3}(1-4\lambda) \\ 
-\sqrt{3}(1-4\lambda) & 1%
\end{pmatrix}%
+\left( \frac{1-\lambda }{2},\frac{\sqrt{3}}{2}\lambda \right) \quad \text{
for}\quad x=(x_{1},x_{2}).
\end{equation*}%
See Figure \ref{lambda4}, \ref{lambda3} and \ref{lambdasq13} for examples
with different $\lambda $'s.

\begin{figure}[h]
\centering
\begin{minipage}[t]{0.32\textwidth}
		\centering\includegraphics[width=1\textwidth]{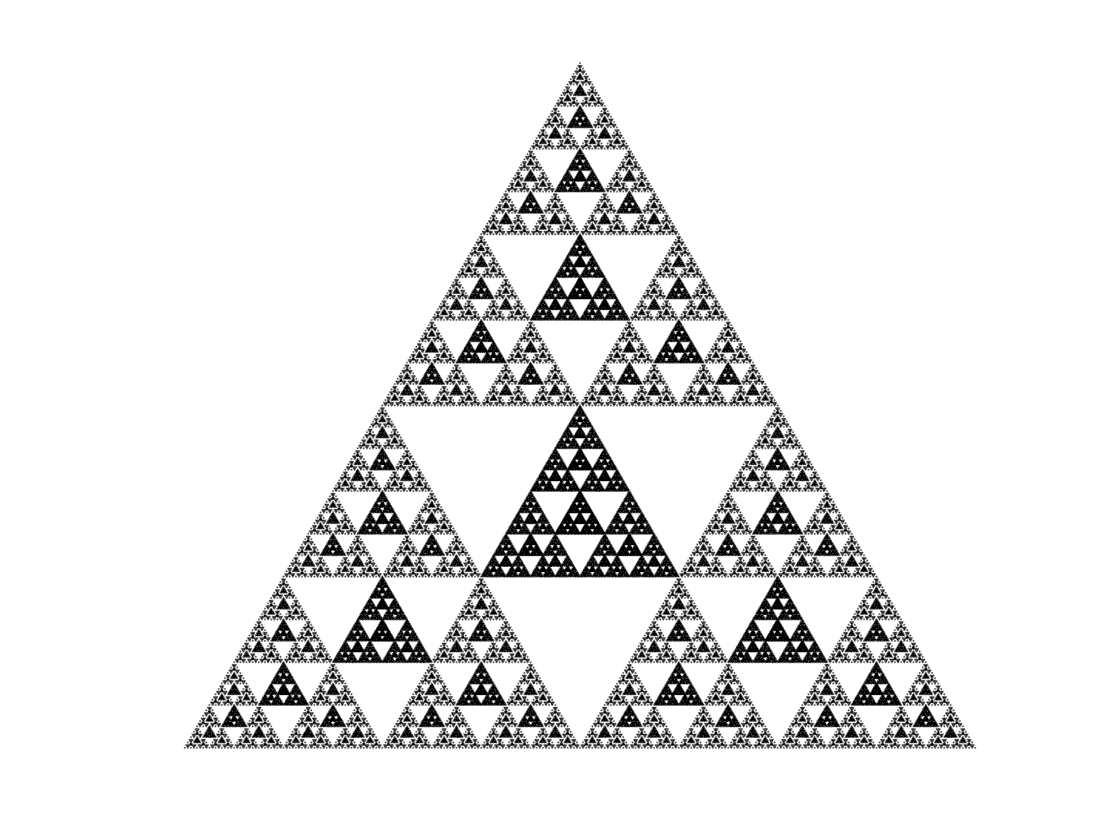}
		\caption{$\lambda=1/4$.}
		\label{lambda4}
	\end{minipage}
\begin{minipage}[t]{0.32\textwidth}
		\centering\includegraphics[width=1\textwidth]{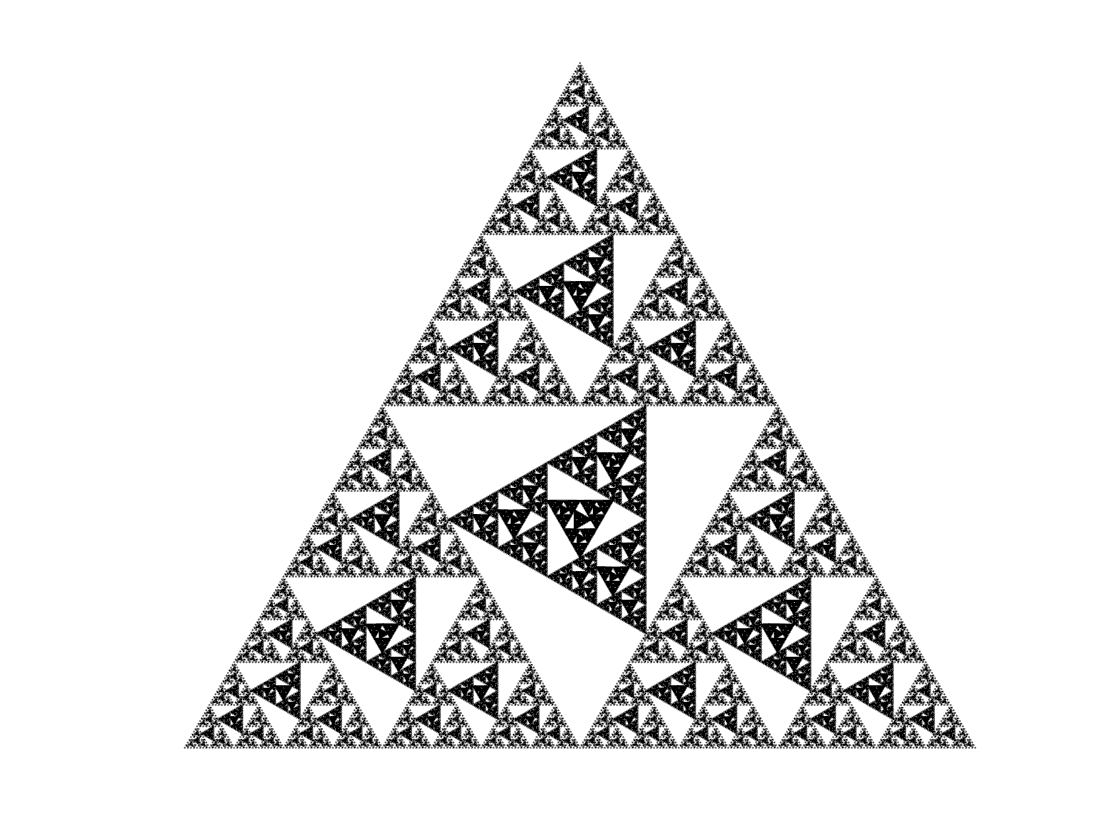}
		\caption{$\lambda=1/3$.}
		\label{lambda3}
	\end{minipage}
\begin{minipage}[t]{0.32\textwidth}
		\centering\includegraphics[width=1\textwidth]{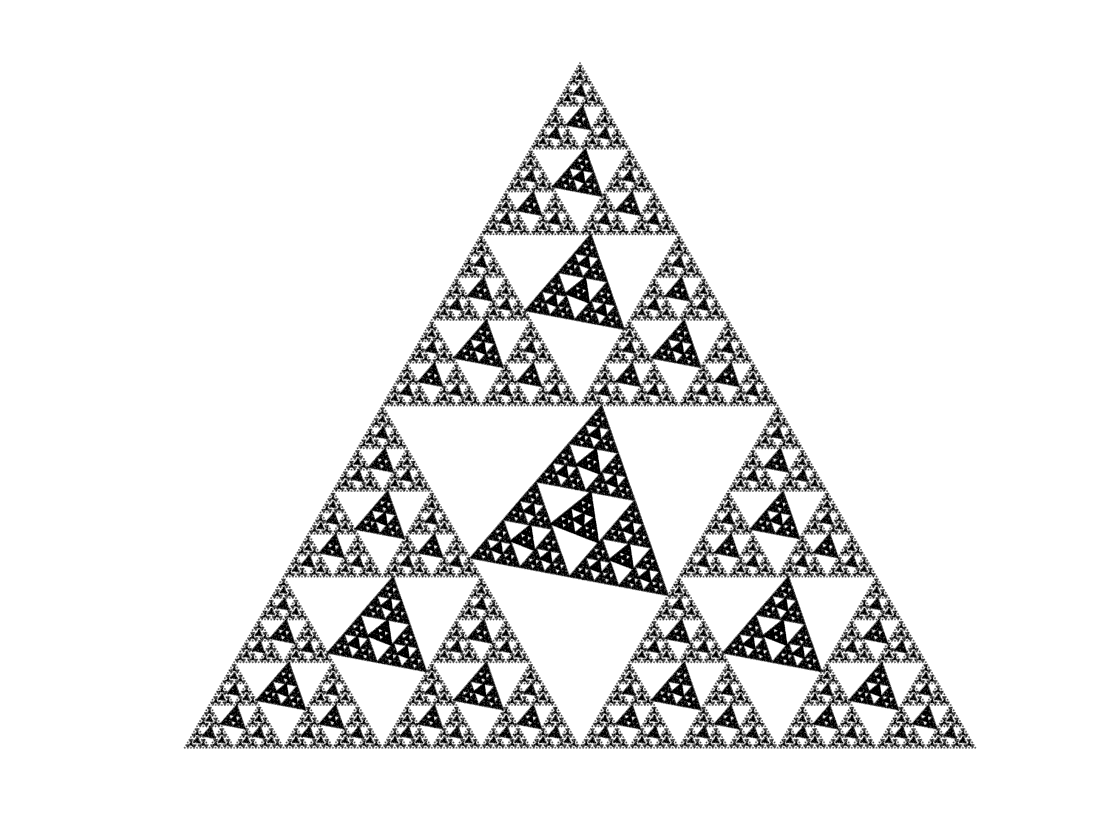}
		\caption{$\lambda=1/\sqrt{13}$.}
		\label{lambdasq13}
	\end{minipage}
\end{figure}

Clearly $K_{\lambda }$ is always connected and the diameter $\bar{R}=1$. Further, $K_\lambda$ always contains the attractor of $\{F_1,F_2,F_3\}$, that is, the well-known Sierpi\'nski gasket (SG, see Example \ref{ex-1}). Meanwhile, the presence of $F_{4,\lambda}$ with its distinct contraction ratio breaks the homogeneity, leading to an enriched cell intersection pattern, which is completely described below and forms the basis of this section.

Henceforth, for any fixed $\lambda\in(0,1/2)$, without giving rise to ambiguity, denote by%
\begin{equation*}
	K=K_{\lambda }\quad\text{ and }\quad F_{4}=F_{4,\lambda}.
\end{equation*}

The following geometric facts about $K_\lambda$ are essential for establishing the analytic framework. Their proofs, involving elementary but tedious plane geometry arguments, are omitted in order to maintain focus on analysis.

\begin{lemma}[Basic geometry of $K$]\label{GeoRT}
	Fix $\lambda \in (0,1/2)$ and set $K = K_\lambda$.
	\begin{enumerate}
		\item $K$ is symmetric with respect to the alternating group $A_3$.
		
		\item The IFS $\{F_1, F_2, F_3, F_4\}$ consists only of similitudes, with contraction ratios
		\begin{equation}\label{KVrho}
		\mathbf{\rho}:=(\rho_1,\rho_2,\rho_3,\rho_4)=\left(\frac{1}{2},\frac{1}{2},\frac{1}{2},\frac{\lambda_*}{2}\right),
		\end{equation}
		where $\lambda_* = \sqrt{12\lambda^2 - 6\lambda + 1} \in [1/2, 1).$
		
		\item The open set condition holds with the basic open set $\blacktriangle \setminus \triangle$, where $\triangle$ is the boundary of $\blacktriangle$.
		
		\item $\triangle\cap K_i\neq\emptyset$ for every $i\in\{1,2,3\}$, whereas $d(\triangle, K_4)=\frac{\sqrt{3}}{2}\left(\lambda\wedge\left(\frac{1}{2}-\lambda\right)\right)>0,$ deriving that $\triangle\cap K_4=\emptyset$.
	\end{enumerate}
\end{lemma}

Lemma \ref{GeoRT} provides a global geometric framework and shows $(\mathrm{WOC}_{\mathbf{\rho}})$ by Lemma \ref{op'cond}. To establish the connectivity property, we must first understand how cells of comparable size intersect.

\begin{lemma}[Cell intersection patterns]\label{GeoRT2} 
Let $v, w$ be distinct finite words with $K_v \cap K_w \neq \emptyset$.
	\begin{enumerate}
		\item The intersection is a single point: $K_v \cap K_w = \{q\}$. The point $q$ is either a common vertex of the triangles $F_v(\triangle)$ and $F_w(\triangle)$, or a vertex of one triangle and a boundary (non-vertex) point of the other.
			
		\item There exist a common prefix $\tau$, distinct digits $i, j \in \{1,2,3,4\}$, and suffixes $\tau^{(1)}, \tau^{(2)}$ containing no digit `4' such that $v = \tau i \tau^{(1)}$, $w = \tau j \tau^{(2)}$ and
		\[
			K_{\tau i} \cap K_{\tau j} =K_v \cap K_w = \{q\}.
		\]
		If $q$ is a vertex of $F_v(\triangle)$ but not of $F_w(\triangle)$, then $i=4$ and $j\neq4$.
	\end{enumerate}
\end{lemma}
For convenience, if $K_v\cap K_w\neq\emptyset$, denote by $q_{vw}$ the unique element of $K_v\cap K_w$ and by $l_v$ (resp., $l_w$) the opposite side of $F_v(\triangle)$ (resp., $F_w(\triangle)$) to $q_{vw}$.

For any finite word $w$, define the \emph{neighbor index set}
\[
\Gamma_{w,c_0} := \left\{ v \in \{1,2,3,4\}^{|w|} : K_v \cap B(K_w, c_0 \rho_w) \neq \emptyset,K_v\nsubseteq K_w \right\},
\]
where $B(K_w, c_0 \rho_w)=\cup_{x\in K_w}B(x, c_0\rho_w)$. Using this notion, we raise a constant that clarifies the local geometry of $K$:

\begin{proposition}[Local uniformity constant]\label{GeoRT3} 
	There exists a positive constant
	$$c_{0}=c_{0}(\lambda):=\frac{3}{32}\left\{ \lambda \wedge \left( \frac{1}{2}-\lambda \right) \right\} ^{2}\in (0,1)$$
	depending continuously on $\lambda$, such that:
	\begin{enumerate}
		\item For any $r\in (0,1)$ and any pair of distinct finite words $v, w \in \Lambda_{\mathbf{\rho}}(r)$, there is no finite word $\tau \in \Lambda_{\mathbf{\rho}}(c_0 r)$ such that $K_\tau$ intersects both $K_v$ and $K_w$.
		\item $\#\Gamma_{w,c_0} \leq 9$,
		\item $|N_4(v) - N_4(w)| \leq 1$ for all $v \in \Gamma_{w,c_0}$, where $N_4(\tau)$ is the number of digit `4` in $\tau$ for any finite word $\tau$. 
		\item For any $v \in \Gamma_{w,c_0}$, there exists $\tau \in \Gamma_{w,c_0}$ (possibly $\tau = v$) such that
		\[K_w \cap K_\tau \neq \emptyset \quad \text{and} \quad K_v \cap K_\tau \neq \emptyset.\]
		If $\tau \neq v$, then the point in $K_\tau \cap K_v$ is a vertex of $F_\tau(\triangle)$.
	\end{enumerate}
\end{proposition}

Based on the geometric structure of $K$ given above, we now begin to verify analytically the doubling property of the self-similar measure and obtain heat kernel estimates. Our first step is to derive an auxiliary estimate following directly from Proposition \ref{GeoRT3}. This estimate will be an essential tool for proof of the connectivity property $(\mathrm{CP}_{\mathbf{\rho}})$ and the averaging property $(\mathrm{AV}_{\mathbf{\rho}})$ later. 

\begin{lemma}\label{rhowc0}
Let $\mathbf{\zeta }:=(\zeta _{1},\zeta _{2},\zeta _{3},\zeta
_{4})$ be a tuple such that $\zeta _{1}=\zeta _{2}=\zeta _{3}$. Then,
for any finite word $w$ and any $v\in \Gamma _{w,c_{0}}$, 
\begin{equation}
\frac{\zeta _{\min }}{\zeta _{\max }}\leq \frac{\zeta _{v}}{\zeta _{w}}\leq 
\frac{\zeta _{\max }}{\zeta _{\min }}.  \label{itasl}
\end{equation}
In particular, for the tuple $\mathbf{\rho }$, it follows that 
\begin{equation}
\lambda _{\ast }\leq \frac{\rho _{v}}{\rho _{w}}\leq \lambda _{\ast }^{-1}
\label{rhosl}
\end{equation}
for any finite word $w$ and for any $v\in \Gamma _{w,c_{0}}$, where $\lambda_{\ast }$ is the constant from (2) of Lemma \ref{GeoRT}.
\end{lemma}

\begin{proof}
Fix a finite word $w$ and fix $v\in \Gamma _{w,c_{0}}$ with $|v|=|w|=n$. It follows by (1) of Lemma \ref{GeoRT2} that $|N_{4}(v)-N_{4}(w)|\leq 1$.
Combining the fact that 
\begin{equation*}
\frac{\zeta _{v}}{\zeta _{w}}=\frac{\zeta _{1}^{n-N_{4}(v)}\zeta
_{4}^{N_{4}(v)}}{\zeta _{1}^{n-N_{4}(w)}\zeta _{4}^{N_{4}(w)}}=\left( \frac{
\zeta _{4}}{\zeta _{1}}\right) ^{|N_{4}(v)-N_{4}(w)|},
\end{equation*}
we obtain (\ref{itasl}). In particular, for $\mathbf{\zeta}=\mathbf{\rho }$, noting that $\frac{\rho
_{\min }}{\rho _{\max }}=\frac{\rho _{4}}{\rho _{1}}=\lambda _{\ast }$, we
obtain (\ref{rhosl}). The proof is complete.
\end{proof}

Now we show that $K$ satisfies $(\mathrm{CP}_{\mathbf{\rho }})$ as follows.

\begin{proposition}
\label{KlCP} $K$ satisfies $(\mathrm{CP}_{\mathbf{\rho },c_{0}})$, where $%
c_{0}$ comes from Proposition \ref{GeoRT3}.
\end{proposition}

\begin{proof}
Fix a ball $B:=B(x,r)$ with $x\in K$, $r\in (0,1)$. Let $y\in c_{0}B\cap K$.
Thus by (2) of Lemma \ref{GeoRT2}, there exist words $w\in \Lambda _{\mathbf{\rho }}(r)$ and $v,\tau \in \Gamma _{w,c_{0}}\subset
\{1,2,3,4\}^{|w|}$ such that 
\begin{equation*}
x\in K_{w},\quad y\in K_{v},\quad K_{w}\cap K_{\tau }\neq \emptyset ,\quad
K_{v}\cap K_{\tau }\neq \emptyset .
\end{equation*}%
Fix $q\in K_{w}\cap K_{\tau }$ and $q^{\prime }\in K_{v}\cap K_{\tau }$.
Obviously, points $x,q$ can be connected by the cell $K_{w}$.

We prove that $q,q^{\prime }$ can be connected by at most two cells coming
from $\Lambda _{\mathbf{\rho }}(r)$. Then the same holds for $q^{\prime },y$, and $(\mathrm{CP}_{\mathbf{\rho },c_{0}})$ follows with $L=5$by coupling the chains.

For convenience, let $|w|=n$ and $\tau =\tau _{1}\tau _{2}\cdots \tau _{n}$.
We distinguish two cases.

Case $(1)$: $\rho _{\tau }\leq r$. Then there exists an integer $1\leq k\leq
n$ such that $\tau ^{(k)}:=\tau _{1}\tau _{2}\cdots \tau _{k}\in \Lambda _{ 
\mathbf{\rho }}(r)$. It is obvious that $q,q^{\prime }$ can be connected by the
cell $K_{\tau }$.

Case $(2)$: $\rho _{\tau }>r$. Since $w\in \Lambda _{\mathbf{\rho }}(r)$ and $\rho _{\tau }\leq \lambda _{\ast }^{-1}\rho _{w}$
by Lemma \ref{rhowc0}, we have by (\ref{part1}) that 
\begin{equation*}
\rho _{\tau }\leq \lambda _{\ast }^{-1}\rho _{w}\leq \lambda _{\ast }^{-1}r.
\end{equation*}
On the other hand, since $\lambda _{\ast }\geq 1/2$, then
$$\rho _{\tau i}=\rho _{\tau }\rho _{i}\leq \frac{1}{2}\lambda _{\ast
}{}^{-1}r\leq r<\rho _{\tau }.$$
Combining the assumption $\rho _{\tau }>r$, we have $\tau i\in \Lambda _{\mathbf{\rho }}(r)$ for every $i\in \{1,2,3,4\}$.

Since $K$ is connected and $K_{4}\cap \triangle =\emptyset $, we can link $%
q,q^{\prime }$ by at most two cells coming from $\{\tau 1,\tau 2,\tau
3\}\subset \Lambda _{\mathbf{\rho }}(r)$, thus proving our claim.
\end{proof}

Now we show the chain condition $(\mathrm{CH})$. Converse to Proposition \ref{P-ch2}, $(\mathrm{CH})$ is obtained from $(\mathrm{CP}_{\mathbf{\rho }})$ here -- the
geometric fact $\triangle \subset K$ plays an important role.

\begin{proposition}
$K$ satisfies $(\mathrm{CH})$.
\end{proposition}

\begin{proof}
We observe by $\triangle \subset K$ that for any integer $k\geq 1$ and any points $z\in K$, there exists a sequence $\{z_{i}\}_{i=0}^{2^{k}}$ such that 
\begin{equation}
z_{0}=q_{1},\quad z_{2^{k}}=z,\quad \{z_{i}\}_{i=0}^{2^{k}-1}\subset
\tbigcup\limits_{v\in I^{k}}F_{v}(\triangle)\quad \text{and}\quad
|z_{i-1}-z_{i}|\leq 2^{2-k}\quad \text{for every }\ 1\leq i\leq 2^{k}.
\label{link}
\end{equation}
Now we show that for any $x,y\in K$, there exists a sequence $%
\{x_{i}\}_{i=0}^{5}\subset K$ such that 
\begin{equation}
x_{0}=x,\quad x_{5}=y\quad \text{and}\quad |x_{i-1}-x_{i}|\leq
2c_{0}^{-1}|x-y|\quad \text{for every}\quad 1\leq i\leq 5.  \label{CH1}
\end{equation}

Indeed, if $c_{0}/2\leq |x-y|\leq 1$, then (\ref{CH1}) trivially holds by
letting $x_{i}=q_{1}$ for all $1\leq i\leq 4$. Thus it suffices to prove for
the case $|x-y|<c_{0}/2$. By $(\mathrm{CP}_{\mathbf{\rho },c_{0}})$ coming
from (i), there exists a sequence $\{x_{i}\}_{i=0}^{5}\subset K$ such that $%
x_{0}=x$, $x_{5}=y$ and for every $1\leq i\leq 5$, $x_{i-1},x_{i}$ lie in a
cell $K_{w^{(i)}}$ with some word $w^{(i)}\in \Lambda _{\mathbf{\rho }
}(2|x-y|/c_{0})$. Hence (\ref{CH1}) is proved with the same $C=10/c_{0}$ by 
\begin{equation*}
|x_{i-1}-x_{i}|\leq \rho _{w^{(i)}}\leq 2c_{0}^{-1}|x-y|\quad \text{for
every}\quad 1\leq i\leq 5.
\end{equation*}

Now we consider chains connecting $x=x_{0}$ and $x_{1}$. Let $%
z_{w^{(1)}}=F_{w^{(1)}}(q_{1})$. By (\ref{link}), for all $n\geq 1$, we can
find chains $\{z_{1,k}\}_{k=0}^{2^{n}}$, $\{z_{2,k}\}_{k=0}^{2^{n}}$ in $%
K_{v}$ connecting $x,z_{w^{(1)}}$ and $z_{w^{(1)}},x_{1}$ respectively such
that $z_{1,0}=x,z_{1,2^{n}}=z_{2,0}=z_{w^{(1)}},z_{2,2^{n}}=x_{1}$ and 
\begin{equation*}
|z_{i,j-1}-z_{i,j}|\leq 2^{2-n}\rho _{w^{(i)}}\leq
2^{3-n}c_{0}^{-1}|x-y|\quad \text{for}\quad i=1,2\quad\text{and}\quad 1\leq j\leq
2^{n}.
\end{equation*}

Coupling these two chains, we obtain a new chain $\{x_{1,k}%
\}_{k=0}^{2^{n+1}}$ such that
\begin{equation*}
x_{1,0}=x,\quad x_{1,2^{n+1}}=x_{1}\quad \text{and}\quad \left\vert
x_{1,j-1}-x_{1,j}\right\vert \leq 2^{3-n}c_{0}^{-1}|x-y|\quad \text{for }\
1\leq j\leq 2^{n+1}.
\end{equation*}%
The same holds for some chain $\{x_{i,k}\}_{k=0}^{2^{n+1}}$ with $i=2,3,4,5$
for all $n\geq 1$. Thus coupling chains $\{x_{i,k}\}_{k=0}^{2^{n+1}}$ in order from $i=1$ to $i=5$, we obtain a chain $\{y_{k}\}_{k=0}^{10\times 2^{n+1}}$such that $y_{0}=x$, $y_{10\times
2^{n+1}}=y$ and 
\begin{equation*}
\left\vert y_{j-1}-y_{j}\right\vert \leq 2^{3-n}c_{0}^{-1}|x-y|=160c_{0}^{-1}%
\frac{|x-y|}{10\times 2^{n+1}}\quad\text{for every}\quad 1\leq j\leq 10\times 2^{n+1},
\end{equation*}%
thus proving $(\mathrm{CH})$ with $C_{\mathrm{CH}}=160/c_{0}$.
\end{proof}

The following gives a sufficient and necessary condition for $(\mathrm{AV}_{ \mathbf{\rho }})$ on $K$.

\begin{proposition}
\label{KlAV} Let $\mathbf{\zeta }:=(\zeta _{1},\zeta _{2},\zeta _{3},\zeta
_{4})$ be a $4$-tuple. Then $\mathbf{\zeta }$ satisfies $(\mathrm{AV}_{%
\mathbf{\rho }})$ if and only if $\zeta _{1}=\zeta _{2}=\zeta _{3}$ where $%
\mathbf{\rho }$ is defined as in (\ref{KVrho}).
\end{proposition}

\begin{proof}
(i) We show the ``if'' part by proving (3) of Proposition \ref{av=}. Let $\zeta _{1}=\zeta _{2}=\zeta _{3}$. Fix $r\in
(0,1)$ and fix finite words $w,v\in \Lambda _{\mathbf{\rho }}(r)$ such that $%
K_{w}\cap K_{v}\neq \emptyset $.

Without loss of the generality, assume that $|w|\geq |v|$. We will show $C_{ 
\mathrm{AV}}^{-1}\zeta _{w}\leq \zeta _{v}\leq C_{\mathrm{AV}}\zeta _{w}$
for some constant $C_{\mathrm{AV}}\geq 1$ independent of $w,v$.

Indeed, it follows by the self-similarity of $K$ that, there exists a finite
word $\tau $ (possibly empty) such that $|v\tau |=|w|$
and $K_{v\tau }\cap K_{w}\neq \emptyset $. Recall by (1) of Lemma \ref{GeoRT2} that 
\begin{equation}
|N_{4}(w)-N_{4}(v\tau)|\leq 1.\label{KAV0}
\end{equation}
On one hand, by Lemma \ref{rhowc0} we have
\begin{equation*}
\frac{\zeta _{v}}{\zeta _{w}}\geq \frac{\zeta _{v\tau }}{\zeta _{w}}\geq 
\frac{\zeta _{\min }}{\zeta _{\max }}\quad\text{and}\quad\lambda _{\ast }\leq \frac{\rho _{v\tau }}{\rho _{w}}\leq \lambda _{\ast
}^{-1}.
\end{equation*}
On the other hand, since $w,v\in \Lambda _{\mathbf{\rho }}(r)$, it follows
by Proposition \ref{part2.1} that 
\begin{equation*}
\frac{\lambda _{\ast }}{2}=\rho _{\min }\leq \frac{\rho _{v}}{\rho _{w}}\leq
\rho _{\min }^{-1}=\frac{2}{\lambda _{\ast }}.
\end{equation*}

Hence, since $\rho _{v\tau }=\rho _{v}\rho _{\tau }$ and $\lambda
_{\ast }\in \lbrack 1/2,1)$, it follows that 
\begin{equation*}
\left( \frac{1}{2}\right) ^{|\tau |}=\rho _{\max }^{|\tau |}\geq \rho _{\tau
}\geq \frac{\lambda _{\ast }^{2}}{2}\geq \frac{1}{8},
\end{equation*}
showing $|\tau |\leq 3$. Combining this with basic facts
$$N_{4}(v\tau)=N_{4}(v)+N_{4}(\tau)\quad\text{and}\quad N_{4}(\tau)\leq|\tau|,$$
we have from (\ref{KAV0}) that 
\begin{equation*}
|N_{4}(w)-N_{4}(v)|\leq |N_{4}(w)-N_{4}(v\tau)|+N_{4}(\tau)\leq 1+|\tau
|\leq 4.
\end{equation*}
Since $\zeta _{1}=\zeta _{2}=\zeta _{3}$, it follows that
\begin{equation}
\left( \frac{\zeta _{\max }}{\zeta _{\min }}\right) ^{-4}\leq \frac{\zeta
_{v}}{\zeta _{w}}=\frac{\zeta _{4}^{N_{4}(v)}\cdot \zeta _{1}^{n-N_{4}(v)}}{
\zeta _{4}^{N_{4}(w)}\cdot \zeta _{1}^{n-N_{4}(w)}}
=\left( \frac{\zeta _{1}}{\zeta _{4}}\right) ^{N_{4}(w)-N_{4}(v)}\leq \left( \frac{\zeta _{\max }}{\zeta _{\min }}\right) ^{4},  \label{20}
\end{equation}
which shows that c) of Proposition \ref{av=} holds with the constant $C_{\mathrm{AV}}=(\frac{\zeta _{\max }}{\zeta _{\min }})^{4}$, thus proving that $\mathbf{\zeta }$ satisfies $(\mathrm{AV}_{\mathbf{\rho }})$ .

(ii) We show the ``only if'' part. Suppose
that $\mathbf{\zeta }$ satisfies $(\mathrm{AV}_{\mathbf{\rho }})$. To show $%
\zeta _{1}=\zeta _{2}=\zeta _{3}$, we consider two words $w:=12^{(m)}$ and $%
v:=21^{(m)}$ for arbitrary $m\geq 1$. By geometry $F_{1}(q_{2})\in
K_{12^{(m)}}\cap K_{21^{(m)}}$ and $\rho _{12^{(m)}}=\rho
_{21^{(m)}}=2^{-(m+1)}$ for all $m\geq 1$. By $(\mathrm{AV}_{\mathbf{\rho }
}) $, 
\begin{equation*}
C^{-1}\leq \frac{\zeta _{w}}{\zeta _{v}}=\frac{\zeta _{1}\zeta _{2}^{m}}{
\zeta _{2}\zeta _{1}^{m}}=\left( \frac{\zeta _{2}}{\zeta _{1}}\right)
^{m-1}\leq C
\end{equation*}
for all $m\geq 1$, showing that $\zeta _{1}=\zeta _{2}$. By the rotational
symmetry of $K$, we similarly have $\zeta _{2}=\zeta _{3}$. Therefore, we
conclude that $\zeta _{1}=\zeta _{2}=\zeta _{3}$.
\end{proof}

Now we obtain a sufficient and necessary condition of $(\mathrm{VD})$ for self-similar measures on $K$.

\begin{corollary}
\label{KlVD} For a fixed number $\lambda\in(0,1/2)$, let $%
\left(\mu,\{p_i\}_{i=1}^{4}\right)$ be a self-similar measure on $%
K:=K_\lambda$. Then $\mu$ satisfies both $(\mathrm{VD})$ and $(\mathrm{RVD})$
if and only if $p_1=p_2=p_3$.
\end{corollary}

\begin{proof}
The ``if'' part follows directly from Proposition \ref{KlAV} (for tuple $\mathbf{p}$) and Corollary \ref{muScal} (using the known properties $(\mathrm{WOC}_{\mathbf{\rho }})$ and $(\mathrm{CP}_{\mathbf{\rho }})$).

We prove the ``only if'' part. Obviously
the IFS $\{ F_{1},F_{2},F_{3},F_{4}\} $ consists of only
similitudes and satisfies the open set condition. Since $\mu $ satisfies $(%
\mathrm{VD})$, it follows by \cite[Theorem 1.1]{Yung2007} that 
\begin{equation*}
p_{v}\leq Cp_{w}
\end{equation*}%
uniformly for any finite words $w,v$ such that $K_{v}\subset \bar{B}(K_{w},\rho _{w})$. In particular, with $w=12^{(m)},v=21^{(m)}$, we obtain 
\begin{equation*}
p_{2}p_{1}^{m}\leq Cp_{1}p_{2}^{m}\quad \text{for all }\ m\geq 1,
\end{equation*}%
which implies $p_{1}\leq p_{2}$. It follows by the rotational symmetry of $K$
that $p_{2}\leq p_{3}$ and $p_{3}\leq p_{1}$, thus showing $%
p_{1}=p_{2}=p_{3} $. 
\end{proof}

\subsection{Refined resistance estimates}

\label{subsec:refined_impedance}

Now we consider the resistance estimates on $K$. Indeed, Cao has
systematically investigated self-similar resistance forms on $K:=K_{\lambda }$
for all $\lambda \in (0,1/2)$ in \cite{Cao.2023.AG}. Here we translate his
main conclusions into the following lemma.

\begin{lemma}
\cite[Theorem 3.6 and Definition 3.7 (b)]{Cao.2023.AG}\label{KlCao} Let $%
\lambda \in (0,1/2)$ and $s\in (0,1)$. Then there exist a unique $%
b:=b(\lambda ,s)\in \lbrack 3/5,1)$ and a unique regular self-similar
resistance form $(\mathcal{E},\mathcal{F}):=(\mathcal{E}_{\lambda ,s}, 
\mathcal{F}_{\lambda ,s})$ on $K:=K_{\lambda }$ satisfying the following
properties:

\begin{enumerate}
\item[i)] $\mathcal{F}\subset C(K)$, and $\mathcal{R}(q_{i},q_{j})=2/3$ for
all $1\leq i<j\leq 3$, where $\mathcal{R}:=\mathcal{R}_{\lambda ,s}$ is the
resistance of $(\mathcal{E},\mathcal{F})$.

\item[ii)] For all $u\in \mathcal{F}$, $u\circ F_{i}\in \mathcal{F}$ for
every $1\leq i\leq 4$ and 
\begin{equation}
\mathcal{E}(u)=b^{-1}\sum_{i=1}^{3}\mathcal{E}(u\circ F_{i})+s^{-1}\mathcal{%
E }(u\circ F_{4}).  \label{Klss}
\end{equation}
\end{enumerate}
\end{lemma}
Definition and properties of resistance forms can be found in Appendix \ref{RFF}.

Note by Proposition \ref{KlAV} that $\mathbf{s}:=(b,b,b,s)$ satisfies $%
(\mathrm{AV}_{\mathbf{\rho }})$. Define 
\begin{equation}\label{Klgam1}
\gamma _{1}:=\gamma _{1}(\lambda ,s)=\frac{\ln b(\lambda ,s)}{\ln \rho _{1}}
,\quad \gamma _{4}:=\gamma _{4}(\lambda ,s)=\frac{\ln s}{\ln \rho _{4}}.
\end{equation}
Since $K$ satisfies $(\mathrm{CP}_{\mathbf{\rho }})$, it follows from
Propositions \ref{KlAV} and \ref{ML4} that, there exists a
scaling function $H(x,r)$ with scaling exponents
\begin{equation}
\gamma ^{\ast }=\gamma _{1}\vee \gamma _{4}\quad\text{and}\quad\gamma _{\ast}=\gamma _{1}\wedge \gamma _{4}.  \label{Klgam}
\end{equation}
such that $H(x,1)=1$ for all $x\in K$, and 
\begin{equation}
C_{\mathbf{s}}^{-1}s_{w}\leq H(x,r)\leq C_{\mathbf{s}}s_{w}  \label{HsimS}
\end{equation}%
for all $x\in K$, $r\in (0,1)$ and any finite word $w\in \Lambda _{\mathbf{\rho }}(r)$ with $x\in K_{w}$, where the constant 
\begin{equation}
C_{\mathbf{s}}:=C_{\mathbf{s}}(\lambda ,s)=\left( \frac{s_{\max }}{s_{\min }}\right) ^{4}s_{\min }^{-1}.  \label{Cs}
\end{equation}

The following proposition is inspired by \cite{Cao.2023.AG}.

\begin{proposition}
\label{KlDF} Let $\lambda \in (0,1/2)$ and $\mu $ be a Radon measure on $K$.
Then for every $s\in (0,1)$, $(\mathcal{E},\mathcal{F})$ is a regular strongly local Dirichlet form on $L^{2}(K,\mu)$.
\end{proposition}

\begin{proof}
Given $\lambda \in (0,1/2)$ and $s\in (0,1)$, we first show there
exists $C=C(\lambda ,s)>0$ such that 
\begin{equation}
C^{-1}d(x,y)^{\gamma ^{\ast }}\leq \mathcal{R}(x,y)\leq Cd(x,y)^{\gamma
_{\ast }},  \label{wKlRE2}
\end{equation}
for all $x,y\in K$, where $\gamma ^{\ast },\gamma _{\ast }$ come from (\ref%
{Klgam}).

Indeed, taking 
\begin{equation*}
\lambda _{n}\in \left[ \frac{\lambda }{2},\frac{1}{4}+\frac{\lambda }{2} %
\right] \cap \mathbb{D},\quad \text{where}\quad \mathbb{D}:=\left\{ \frac{k}{
2^{m}}:k\in \mathbb{Z},m\in \mathbb{N}\right\}
\end{equation*}
such that $\lambda _{n}\rightarrow \lambda $, we see by \cite[Proposition
5.6]{Cao.2023.AG} that (\ref{wKlRE2}) holds for $\lambda _{n},s$ with some $%
C$ depending on $\lambda $ and $s$ but independent of $n$. Further, by \cite[
Subsection 6.1]{Cao.2023.AG}, $\mathcal{R}_{\lambda _{n},s}(x,y)\rightarrow 
\mathcal{R}_{\lambda ,s}(x,y)=\mathcal{R}(x,y)$ for all $x,y\in K$ and $%
b_{n}\rightarrow b$. Therefore, $\gamma ^{\ast }(\lambda _{n},s)\rightarrow
\gamma ^{\ast }(\lambda ,s)=\gamma ^{\ast }$ and $\gamma _{\ast }(\lambda
_{n},s)\rightarrow \gamma _{\ast }(\lambda ,s)=\gamma _{\ast }$. Thus (\ref%
{wKlRE2}) follows for $\lambda ,s$ with the same $C>0$.

As a consequence of (\ref{wKlRE2}), $\mathcal{R}$ induces the same topology as the original one. Thus $(\mathcal{E},\mathcal{F})$ is a regular
Dirichlet form on $L^{2}(K,\mu)$ by Lemma \ref{R2DF}. The strong locality follows from Proposition \ref{ssploc}, using the fact that $b,s<1$.
\end{proof}

Now we show the resistance estimates with respect to $(\mathcal{E},\mathcal{F})$, which are more precise than what is given in \cite[Proposition 5.6]{Cao.2023.AG}.

\begin{proposition}
For any $\lambda\in(0,1/2)$ and $s\in(0,1)$, the resistance $\mathcal{R}:= 
\mathcal{R}_{\lambda,s}$ related to $(\mathcal{E},\mathcal{F}):=(\mathcal{E}
_{\lambda,s},\mathcal{F}_{\lambda,s})$ on $K:=K_\lambda$ satisfies $(\mathrm{%
R}^H)$.
\end{proposition}

\begin{proof}
Let $c_{0}$ be the constant in Proposition \ref{GeoRT3}.

(i) We prove $(\mathrm{R}_{\leq }^{H})$ first. Fix $x,y\in K$. We
distinguish two cases.

Case $(1)$: $|x-y|\geq c_{0}/2$. Then by (\ref{wKlRE2}) and the
definition of $H$, there exists a constant $C$ independent of $\lambda
,s$ such that 
\begin{equation*}
\mathcal{R}(x,y)\leq C=CH(x,1)\leq C(c_{0}/2)^{-\gamma ^{\ast
}}H(x,c_{0}/2)\leq C(c_{0}/2)^{-\gamma ^{\ast }}H(x,|x-y|),
\end{equation*}
thus proving $(\mathrm{R}_{\leq }^{H})$ in this case.

Case $(2)$: $|x-y|<c_{0}/2$. Let $w,v\in \Lambda _{\mathbf{\rho }
}(2|x-y|/c_{0})$ be words satisfying $x\in K_{w},y\in K_{v}$. Since $K$ satisfies $(\mathrm{CP}_{\mathbf{\rho },c_{0}})$ by Proposition \ref{KlCP}, there exist finite words $\left\{ \tau ^{(i)}\right\} _{i=1}^{5}\subset
\Lambda _{\mathbf{\rho }}(2|x-y|/c_{0})$ such that, 
\begin{equation*}
\tau ^{(1)}=w,\quad\tau ^{(5)}=v\quad\text{and}\quad K_{\tau ^{(i)}}\cap K_{\tau
^{(i+1)}}\neq \emptyset\quad\text{for every}\quad 1\leq i\leq 4.
\end{equation*}
Fix $z_{i}\in K_{\tau ^{(i)}}\cap K_{\tau ^{(i+1)}}\ $for every $1\leq i\leq
5$ and let $z_{0}=x,z_{5}=y$.Take $C_{\mathrm{AV}}$, the constant in
Proposition \ref{KlAV}. Noting that $(\mathcal{E},\mathcal{F})$ is
a self-similar resistance form and $\mathbf{\zeta }$ satisfies $(\mathrm{AV}_{\mathbf{\rho }})$, then
\begin{equation*}
\mathcal{R}(x,y)\leq \dsum\limits_{i=1}^{5}\mathcal{R}(z_{i-1},z_{i})\leq
\dsum\limits_{i=1}^{5}Cs_{\tau ^{(i)}}\leq C\dsum\limits_{i=1}^{5}C_{\mathrm{%
\ AV}}^{i-1}s_{w}=5CC_{\mathrm{AV}}^{4}s_{w}=:C^{\prime }s_{w}.
\end{equation*}
Thus by (\ref{HsimS}) and (\ref{scpr}), 
\begin{equation*}
\mathcal{R}(x,y)\leq C^{\prime }s_{w}\leq C^{\prime }C_{\mathbf{s}
}H(x,2|x-y|/c_{0})\leq C^{\prime }C_{\mathbf{s}}(2/c_{0})^{\gamma ^{\ast
}}H(x,|x-y|),
\end{equation*}
which shows $(\mathrm{R}_{\leq }^{H})$ again.

(ii) Now we prove $(\mathrm{R}_{\geq }^{H})$. We distinguish two cases.

Case (1): $\lambda \in (0,1/2)\cap \mathbb{D}$. In this setting, the harmonic extension on cells is well-defined by \cite[Lemma 5.2]{Cao.2023.AG}.

We first show there exists a constant $C:=C(\lambda ,s)>0$ depending continuously on $\lambda,s$ such that, for all $B(x,r)\subsetneqq K$ and $w\in \Lambda _{\mathbf{\rho }}(\frac{r}{4})$ satisfying $x\in K_{w}$, 
\begin{equation}
\mathcal{R}(x,B(x,r)^{c})\geq C^{-1}s_{w}.  \label{rb-w}
\end{equation}

Fix $x\in K$, $r\in (0,1)$ and set
\begin{equation*}
S_{r}(x):=\left\{w\in \Lambda_{\mathbf{\rho }}\left(\frac{r}{4}\right):x\in
K_{w}\right\},\quad K_{r}(x):=\bigcup _{w\in S_{r}(x)}K_{w}.
\end{equation*}
It follows from the geometry of $K$ that $\# S_r(x)\leq 3$ for all $x\in K$ and all $r\in(0,1)$. Figure~\ref{fig:srx123} illustrates all possible configurations of $K_r(x)$ when $\#S_r(x)=1$, $2$, or $3$.
\begin{figure}[htbp]
	\centering
	\includegraphics[width=0.9\textwidth]{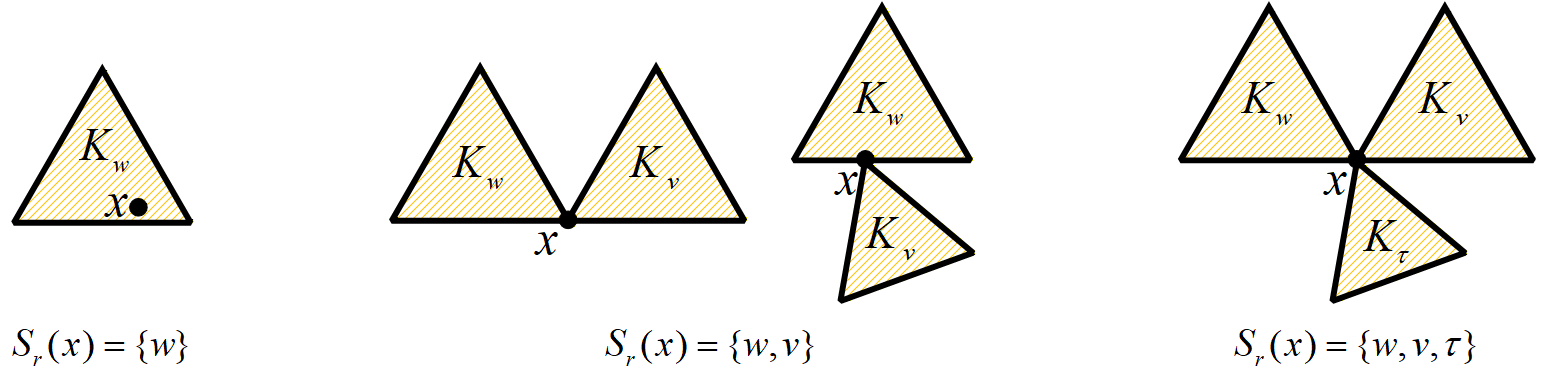}
	\caption{All possible cases for $K_r(x)$.}
	\label{fig:srx123}
\end{figure}

Now we consider sufficiently small cells intersecting $K_r(x)$. Set 
\begin{equation}
	S^{(1)}_{c_0r/4}(x):=\left\{v\in\Lambda_{\mathbf{\rho}}(c_0r/4):K_v\cap K_r(x)\neq\emptyset, K_v\nsubseteq K_r(x)\right\}.
\end{equation}
It follows by (2) of Proposition \ref{GeoRT3} that $\# S^{(1)}_{c_0r/4}(x)\leq 3\times 9=27$ for all $x\in K$ and all $r\in(0,1)$, and by (3) of Lemma \ref{GeoRT2} that, for every $v\in S^{(1)}_{c_0r/4}(x)$, there exists a unique word $w\in S_r(x)$ such that $K_v\cap K_w\neq\emptyset$. Set
\begin{equation}
	S_1[v]:=\{\tau\in\Lambda_{\mathbf{\rho}}(c_0^2 r/4): K_\tau\cap K_v\neq\emptyset, K_\tau\nsubseteq K_v, K_\tau\cap K_w=\emptyset\}.
\end{equation}
It follows by (2) of Proposition \ref{GeoRT3} that $\# S_1[v]\leq 9$.
See Figure \ref{fig:tfr} for the illustration of $S^{(1)}_{c_0r/4}(x)$ and $S_1[v]$.
\begin{figure}[htbp]
	\centering
	\includegraphics[width=0.6\textwidth]{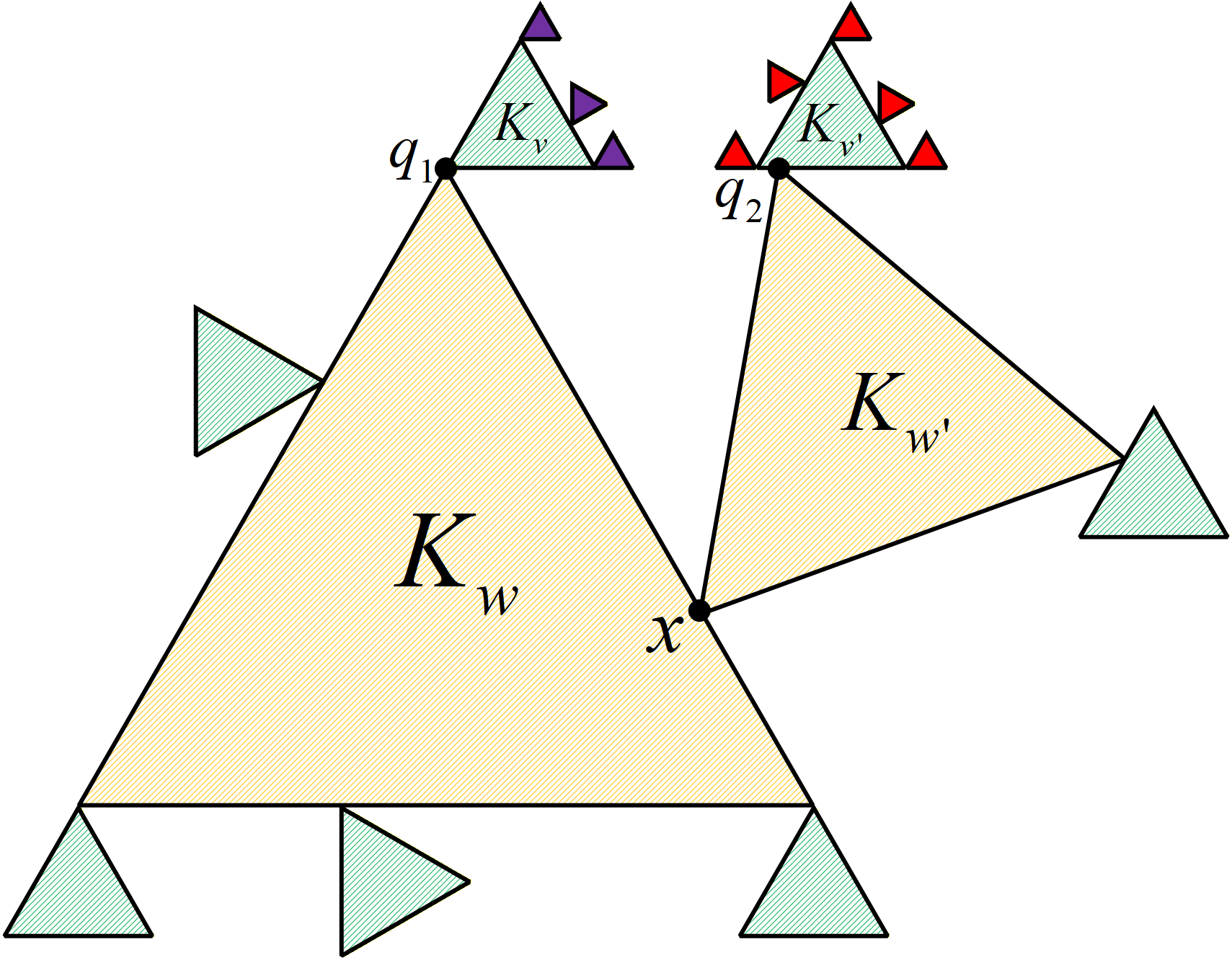}
	\caption{In this figure, the two yellow cells constitute $K_r(x)$, where $w,w'$ are the elements of $S_r(x)$. Green cells correspond to words in $S^{(1)}_{c_0r/4}(x)$, purple cells to $S_1[v]$, and red cells to $S_1[v']$.}
	\label{fig:tfr}
\end{figure}

Now we construct a function $h$. To begin with, for each $v \in S^{(1)}_{c_0r/4}(x)$, let $w$ be the unique word in $S_r(x)$ such that $K_v\cap K_w=\{q_{vw}\}\neq\emptyset$. We define a function $h'_v$ on $K_v$ according to the position of $q_{vw}$ in the triangle $F_v(\triangle)$:
\begin{itemize}
	\item \textbf{Case 1:} $q_{vw}$ is a vertex of $F_v(\triangle)$ (e.g., $q_1$ in Figure~\ref{fig:tfr}). Let $l_v$ be the opposite side. Define $h'_v$ to be the harmonic function on $K_v$ satisfying $h'_v(q_{vw})=1$ and $h'_v|_{l_v}\equiv0$.
	\item \textbf{Case 2:} $q_{vw}$ lies on a side of $F_v(\triangle)$ but is not a vertex (e.g., $q_2$ in Figure~\ref{fig:tfr}). In this case set $h'_v\equiv1$.
\end{itemize}
In both cases, the energy of $h'_v\circ F_v$ is uniformly bounded: when $q_{vw}$ is a vertex, the bound follows from \cite[Lemma 5.2]{Cao.2023.AG}; when $q_{vw}$ is not a vertex, $\mathcal{E}(h'_v\circ F_v) = 0$. Hence, we always have
\begin{equation}\label{Ev}
	\mathcal{E}(h'_v\circ F_v) \leq 2\left( \frac{1}{b} + \frac{1}{s} \right) \leq 2\left( \frac{5}{3} + \frac{1}{s} \right).
\end{equation}

For each $\tau \in S_1[v]$, the intersection $K_\tau \cap K_v$ is a single point $q_{\tau v}$, which is necessarily a vertex of $F_\tau(\triangle)$. Let $l_\tau$ be the opposite side. We define $h_\tau$ on $K_\tau$ as the harmonic function with boundary values $h_\tau(q_{\tau v}) = h'_v(q_{\tau v})$ and $h_\tau|_{l_\tau} \equiv 0$. Again by \cite[Lemma 5.2]{Cao.2023.AG}, we have
\begin{equation}\label{Etau}
	\mathcal{E}(h_\tau \circ F_\tau) \leq 2\left( \frac{5}{3} + \frac{1}{s} \right) \text{ for any } \tau \in S_1[v].
\end{equation}

Recall by Property (1) of Proposition \ref{GeoRT3} that $K_1[v]\cap K_1[v']$ whenever $v\ne v'$, where
$$K_1[v]:=\cup_{\tau\in S_1[v]}K_{\tau}\cup K_v.$$

Now, define function $h$ on $K$ such that
\begin{equation}
	h(y)=\begin{cases}
		1,&y\in K_r(x),\\
		h'_v(y),&y\in K_v\quad\text{for any}\quad v\in S^{(1)}_{c_0r/4}(x),\\
		h_\tau(y),&y\in K_\tau\quad\text{for any}\quad\tau\in S_1[v]\quad\text{with}\quad v\in S^{(1)}_{c_0r/4}(x),\\
		0,&y\in K\setminus\left(K_r(x)\cup\left(\cup_{v\in S^{(1)}_{c_0r/4}(x)}K_1[v]\right)\right),
	\end{cases}
\end{equation}
It is easy to verify that $h$ is well-defined and continuous, with $h\in\mathcal{F}$ and $h(x)=1$. In addition, for any $z\in K_r(x)\cup(\cup_{v\in S^{(1)}_{c_0r/4}(x)}K_1[v])$, there exist $w\in\Lambda_{\mathbf{\rho}}(r/4)$,$v\in S_{ar}^{(1)}(x), \tau\in S_1[v]$ such that
$$x\in K_w,\quad K_w\cap K_v\ne\emptyset,\quad K_v\cap K_\tau\ne\emptyset\quad\text{and}\quad z\in K_\tau.$$
Therefore,
\begin{equation*}
	d(x,z)\leq d(x, q_{vw})+d(q_{vw},q_{\tau v})+d(q_{\tau v},z)\leq \rho_w+\rho_v+\rho_\tau\leq \dfrac{3r}{4}<r,
\end{equation*}
thus 
\begin{equation*}
K_r(x)\cup\left(\cup_{v\in S^{(1)}_{c_0r/4}(x)}K_1[v]\right)\subset B(x,r)\quad\text{and}\quad h|_{B(x,r)^c}\equiv0,
\end{equation*}
showing that $h$ is a test function with respect to $\mathcal{R}(x,B(x,r)^c)$.

Now we estimate $\mathcal{E}(h)$. Indeed, noting by the self-similarity of $(\mathcal{E},\mathcal{F})$, the fact that $h$ is a constant on $K_r(x)$ or $K\setminus(K_r(x)\cup(\cup_{v\in S^{(1)}_{c_0r/4}(x)}K_1[v]))$, \eqref{Ev} and \eqref{Etau}, we have
\begin{eqnarray}
	\mathcal{E}(h)&=&\sum_{v\in\Lambda_{\mathbf{\rho}}(c_0r/4)}\frac{1}{s_v}\mathcal{E}(h\circ F_v)=\sum_{v\in S_{c_0r/4}^{(1)}(x)}\frac{1}{s_v}\mathcal{E}(h\circ F_v)+\sum_{v\in S_{c_0r/4}^{(1)}(x)}\sum_{\tau\in S_{1}(v)}\frac{1}{s_\tau}\mathcal{E}(h\circ F_\tau)\notag\\
	&=& \sum_{v\in S_{c_0r/4}^{(1)}(x)}\frac{1}{s_v}\mathcal{E}(h'_v\circ F_v)+\sum_{v\in S_{c_0r/4}^{(1)}(x)}\sum_{\tau\in S_{1}(v)}\frac{1}{s_\tau}\mathcal{E}(h_{\tau}\circ F_\tau)\notag\\
	&\leq& 2\left(\frac{5}{3}+\frac{1}{s}\right)\left(\sum_{v\in S_{c_0r/4}^{(1)}(x)}\frac{1}{s_v}+\sum_{v\in S_{c_0r/4}^{(1)}(x)}\sum_{\tau\in S_{1}(v)}\frac{1}{s_\tau}\right).\label{Eh1}
\end{eqnarray}

Now fix $\hat{w}\in\Lambda_{\mathbf{\rho}}(r)$ such that $x\in K_w$. We will show that the values of $s_{v}^{-1}, s_{\tau}^{-1}$ above can be controlled by $s_{\hat{w}}^{-1}$. 

Indeed, fix any $v\in S^{(1)}_{c_0r/4}(x)$, let $w\in K_r(x)$ be the word such that $K_v\cap K_w\neq\emptyset$. Then, there exists a finite word $\hat{v}\in\Lambda_{\mathbf{\rho}}(r/4)$ such that $K_v\subset K_{\hat{v}}$. Then it follows by $(\mathrm{AV})$ that for any $v\in S_{ar}^{(1)}(x)$, there is
\begin{equation}\label{vandw}
	\frac{1}{s_v}=\frac{1}{s_w} \cdot\frac{s_w}{s_{\hat{v}}}\cdot \frac{s_{\hat{v}}}{s_v}\leq \frac{1}{s_w}\cdot C_{\mathrm{AV}}\left(\frac{\rho_{\hat{v}}}{\rho_v}\right)^{\gamma^{\ast}}	\leq\left(\frac{3}{5}\wedge s\right)^{-4}\left(\frac{c_0\lambda_{\ast }}{8}\right)^{\frac{\ln ((3/5)\wedge s)}{\ln 2}}\frac{1}{s_w}:=\frac{C_{0,1}(\lambda,s)}{s_w}.
\end{equation}
Clearly $C_{0,1}(\lambda,s)$ depends on $\lambda,s$ continuously.

Similarly, noticing by definition of $S_1(v)$, we obtain that for any $\tau\in S_1(v)$,
\begin{equation}\label{tauandw}
	\frac{1}{s_\tau}\leq \frac{C_{0,1}(\lambda,s)}{s_v}\leq \frac{C_{0,1}^2(\lambda,s)}{s_w}.
\end{equation} 

Applying \eqref{vandw} and \eqref{tauandw} to \eqref{Eh1}, also noticing that $\# S_{ar}^{(1)}(x)\leq 27,\# S_1(v)\leq 9$ for all $v\in S_{ar}^{(1)}(x)$ and for all $x\in K$, we finally obtain
\begin{eqnarray*}
	\mathcal{E}(h)&\leq &2\left(\frac{5}{3}+\frac{1}{s}\right)\left(\sum_{v\in S_{ar}^{(1)}(x)}\frac{1}{s_v}+\sum_{v\in S_{ar}^{(1)}(x)}\sum_{\tau\in S_{1}(v)}\frac{1}{s_\tau}\right)\\
	&\leq &2\left(\frac{5}{3}+\frac{1}{s}\right)\left(27C_{0,1}(\lambda,s)+27\times9C_{0,1}^2(\lambda,s)\right)\frac{1}{s_{\hat{w}}}=:C_{0,2}(\lambda,s)\frac{1}{s_{\hat{w}}},
\end{eqnarray*}
thus proving \eqref{rb-w} with constant
\begin{equation*}
	C_{0,2}(\lambda,s):=2\left(\frac{5}{3}+\frac{1}{s}\right)\left(27C_{0,1}(\lambda,s)+27\times9C_{0,1}^2(\lambda,s)\right)
\end{equation*}
continuously depending on $\lambda,s$. Combing this consequence with \eqref{Cs} and the scaling property of $H$, we obtain that
\begin{eqnarray}
	\mathcal{R}(x,B(x,r)^c)&\geq& \frac{1}{C(\lambda,s)} s_w \geq \frac{1}{C_{0,2}(\lambda,s)C_{\mathbf{s}}} H\left(x,\frac{r}{4}\right)=\frac{s_{\min}}{C_{0,2}(\lambda,s)}\left(\frac{s_{\min}}{s_{\max}}\right)^4H\left(x,\frac{r}{4}\right)\notag\\
	&\geq&\frac{\left(\frac{3}{5}\wedge s\right)^5}{C_{0,2}(\lambda,s)}H\left(x,\frac{r}{4}\right)=:C_{0,3}(\lambda,s)H\left(x,\frac{r}{4}\right)\label{RGeqKV}
\end{eqnarray}
Noting that $y\in B(x,|x-y|/4)^c$, it follows by the monotonicity of $\mathcal{R}$, \eqref{RGeqKV} and the scaling property of $H$ that
\begin{equation*}
	\mathcal{R}(x,y)\geq \mathcal{R}(x,B(x,|x-y|/4)^c)\geq C_{0,3}(\lambda,s)H\left(x, \frac{|x-y|}{4}\right)\geq C_{0,3}(\lambda,s)4^{-\gamma^\ast}C_H^{-1} H(x,r),
\end{equation*}
thus proving $(\mathrm{R}^{H}_{\geq})$ in this case.

Case $(2)$: $\lambda \in (0,1/2)\setminus \mathbb{D}$. Take a
sequence $\{\lambda _{n}\} _{n=1}^{\infty }$ in $(0,1/2)\cap\mathbb{D}$ such that $\lambda _{n}\rightarrow \lambda $ as $n\rightarrow
\infty $. Note that $x_{n}\rightarrow x$, $y_{n}\rightarrow x$ for all $y\in K$ as $\lambda
_{n}\rightarrow \lambda $, where $x_{n}, y_{n}$ are points on $K_{\lambda _{n}}$ with the same addresses as $x,y$ respectively. Then by \cite[Subsection 6.1]{Cao.2023.AG}, we see $\mathcal{R}_{\lambda _{n},s}(x_{n},y_{n})\rightarrow \mathcal{%
\ R}(x,y)$ for all $x,y\in K$ and $b_{n}:=b(\lambda_n,s)\rightarrow b$. In particular, with $%
H_{n}$ the scaling function for $K_{\lambda _{n}}$ with $\lambda _{n},s$, we
see 
\begin{equation*}
H\left(x,\frac{|x-y|}{4}\right)\leq C_0\varliminf_{n\rightarrow \infty }H_{n}\left(x_{n}, \frac{|x_{n}-y_{n}|}{4}\right)
\end{equation*}
with a positive constant $C_0$ depending only on $\lambda $ and $s$ but independent of $x,y,x_n,y_n$.

Since $\{\lambda _{n}\} _{n=1}^{\infty }\subset (0,1/2)\cap 
\mathbb{D}$, we have by \eqref{RGeqKV} that 
\begin{equation*}
\mathcal{R}_{\lambda _{n},s}(x_{n},y_{n})\geq C_{0,3}(\lambda _{n},s)H_n\left(x_n,\frac{|x_n-y_n|}{4}\right)\quad\text{for any}\quad n\geq 1.
\end{equation*}
Note by \eqref{RGeqKV} that $C_{0,3}(\cdot,\cdot)$ continuously depending on $\lambda,s$. By letting $n\rightarrow \infty $, we have 
\begin{eqnarray*}
\mathcal{R}(x,y) &\geq &C_{0,3}(\lambda ,s)\varliminf_{n\rightarrow \infty }H_n\left(x_n,\frac{|x_n-y_n|}{4}\right) \\
&\geq &\frac{C_{0,3}(\lambda ,s)}{C_{0} }H\left(x,\frac{|x-y|}{4}\right)\geq c^{\prime }H(x,|x-y|)
\end{eqnarray*}
with a constant $c^{\prime }>0$ independent of $x,y$, thus proving $(\mathrm{R}_{\geq }^{H})$ again. The proof is complete.
\end{proof}

\subsection{Spatial inhomogeneity of $W$ and heat kernel estimates}

Now we are ready to derive heat kernel estimates on $K$. Fix $\lambda \in (0,1/2)$ and $s\in (0,1)$. Let $( \mu ,\{p_{i}\}_{i=1}^{4})$ be a doubling self-similar measure on $K:=K_{\lambda}$ (with $p_1=p_2=p_3$, by Corollary \ref{KlVD}) and $(\mathcal{E},\mathcal{F}):=(\mathcal{E}_{\lambda ,s},\mathcal{F}_{\lambda ,s})$ be the Dirichlet form on $L^{2}(K,\mu)$ from Proposition \ref{KlDF}. 

The combined analytic-geometric data is encoded in the $4$-tuple 
\begin{equation}
	\mathbf{\psi}:=(\psi _{1},\psi _{2},\psi _{3},\psi _{4})=(p_{1}b,p_{2}b,p_{3}b,p_{4}s). \label{psi}
\end{equation}
Since $p_1=p_2=p_3$, this tuple satisfies $(\mathrm{AV}_{\mathbf{\rho}})$ by Proposition \ref{KlAV}. Parallel to the definition of $\gamma_1,\gamma_4$ in \eqref{Klgam1}, define
\begin{equation}
	\beta _{1}:=\beta _{1}(\lambda ,s)=\frac{\ln \psi _{1}}{\ln \rho _{1}},\quad
	\beta _{4}:=\beta _{4}(\lambda ,s)=\frac{\ln \psi _{4}}{\ln \rho _{4}},\quad
	\beta _{\ast }=\beta _{1}\wedge \beta _{4},\quad \beta ^{\ast }=\beta _{1}\vee \beta _{4}.  \label{bt1-}
\end{equation}
Proposition \ref{ML4} then guarantees the existence of a scaling function $W(x,r)$ with exponents $\beta_\ast, \beta^\ast$ such that
\begin{equation*}
	W(x,1)=1,\quad W(x,r)\asymp V(x,r)H(x,r) \quad\text{for all}\quad x\in K,\ r\in(0,1).
\end{equation*}

The function $W$ is the cornerstone of our heat kernel estimates. Its most striking feature is that, unlike in the homogeneous setting, $W$ \emph{essentially} depends on the spatial variable $x$ whenever the exponents $\beta_1$ and $\beta_4$ differ. This is the analytic presentation of the geometric inhomogeneity raised by the rotated mapping $F_4$.

\begin{remark}
If $\beta_1 \neq \beta_4$, then the scaling function $W(x,r)$ depends essentially on both $x$ and $r$. Specifically, the asymptotic power-law behavior of $W(\cdot, r)$ as $r\to 0$ varies from point to point.
\end{remark}

To be exact, Let $q_1$ and $q_4$ be the fixed points of the maps $F_1$ and $F_4$ respectively. By the construction of $W$ in Proposition \ref{ML4}, at these points
$$W(q_1, r) \asymp r^{\beta_1}, \quad W(q_4, r) \asymp r^{\beta_4} \quad \text{for all}\quad r\in(0,1).$$
Since $\beta_1 \neq \beta_4$, the ratio 
$$\lim_{r\to 0^+} \frac{W(q_1, r)}{W(q_4, r)} = 0\quad\text{or}\quad \infty$$
diverges. Hence $W$ cannot be controlled by a function of $r$ alone; it exhibits essential joint dependence.

We emphasize that, under whatever small scale, the inhomogenity always exists. For example, with $x_m = F_{1^{(m)}}(q_4)$, which constitutes a convergent sequence to $q_1$, we still have
\[W(x_m, r) = o(r^{\beta_4}) \quad \text{as}\quad r\to 0^+.\]
That is, for any $\varepsilon > 0$, one can find two points $x_0, y_0 \in K$ with $d(x_0, y_0) < \varepsilon$ such that $W(x_0, r)$ and $W(y_0, r)$ tend to $0$ with different rates as $r\to 0^+$. This highlights the local heterogeneity of the effective metric induced by $W$.

Since $b\in \lbrack 3/5,1)$ by Lemma \ref{KlCao} and $p_{1}\in (0,1/3)$ by
the fact $p_1=p_2=p_3$, we always have 
\begin{equation*}
\beta ^{\ast }\geq \beta _{1}=\frac{\ln p_{1}b}{\ln \rho _{1}}\geq \frac{\ln
3}{\ln 2}>1,
\end{equation*}%
whereas $\beta _{\ast }$ may be smaller than $1$ provided that $p_{4},s$ are
both sufficiently close to $1$.

Based on all the conclusions through Sections \ref{mainresults} to \ref{Kl}, we have the following estimates:

\begin{corollary}
\label{KlHK} Let $\lambda \in (0,1/2)$, $s\in (0,1)$, and $(\mu
,\{p_{i}\}_{i=1}^{4})$ be a doubling self-similar measure on $K:=K_{\lambda
} $. Let $(\mathcal{E},\mathcal{F}):=(\mathcal{E}_{\lambda ,s},\mathcal{F}
_{\lambda ,s})$ be the Dirichlet form on $L^{2}(K,\mu)$ in Proposition \ref%
{KlDF}. Then $(\mathcal{E},\mathcal{F})$ admits a heat kernel $%
p_{t}(x,y):=p_{t}^{(\lambda ,s)}(x,y)$ such that $p_{t}(\cdot ,\cdot)\in
C(M\times M)$ for any $t>0$ and satisfies $(\mathrm{UE}_{\exp }^{W})$ and $( 
\mathrm{NLE}^{W})$. Moreover, if $\beta _{\ast }>1$, then $(\mathrm{LE}
_{\exp }^{W})$ holds.\qed
\end{corollary}

Indeed, with the help of more detailed geometry of $K$, we can
obtain a better lower estimate. Before that, we give an
observation on chains in $K$.

Denote by $\triangle _{0}:=\triangle $ and $\triangle _{k}:=\cup _{w\in
I^{k}}F_{w}(\triangle)$ for any $k\geq 1$. It is easy to see that%
\begin{equation*}
\triangle =\triangle _{0}\subset \triangle _{1}\subset \triangle _{2}\subset
\cdots \subset \triangle _{k}\subset \triangle _{k+1}\subset \cdots ,
\end{equation*}

and the set $\triangle ^{\ast }:=\cup _{k\geq 1}\triangle _{k}$ is a dense
subset of $K$.

\begin{proposition}
\label{Chain}There exist an integer $A\geq 1$ and a constant $C>0$ such
that, for any $n\geq 1$, any $l\in\{1,2,3\}$ and any $x\in
\triangle_{n}$, there exists a chain $\{x_{i}\}_{i=0}^{L}\subset\triangle _{n}$ satisfying the following:

\begin{enumerate}
\item $L\leq A2^{n},x_{0}=q_{l},x_{L}=x$, 

\item $W(x_{i-1},|x_{i-1}-x_{i}|)\leq C\psi _{1}^{n}$ for any $1\leq
i\leq L$, where $\psi_1$ comes from \eqref{psi}.
\end{enumerate}
\end{proposition}

\begin{proof}
Since $x\in \triangle _{n}$, there exists a finite word $w=w_{1}w_{2}\cdots
w_{n}$ such that $x\in F_{w}(\triangle)$. Let $w^{(0)}$ be the empty word
and denote by 
\begin{equation*}
w^{(k)}=w_{1}w_{2}\cdots w_{k},\quad v^{(k)}=w^{(k-1)}l\quad\text{for any}\quad 1\leq k\leq n.
\end{equation*}
It follows that $\{F_{w^{(k-1)}}(q_{l}),F_{w^{(k)}}(q_{l})\}
\subset F_{v^{(k)}}(\triangle)$ for any $1\leq k\leq n$. We distinghuish
two cases.

Case $(1)$: $\psi _{1}\geq \psi _{4}$. Indeed, it is easy to see by the
geometry of $\triangle $ that, for any $1\leq k\leq n$, there exists a chain $\{x_{k,i}\}_{i=0}^{2^{n-k}}\subset F_{v^{(k)}}(\triangle)$
such that $x_{k,0}=F_{w^{(k-1)}}(q_{l}),x_{k,2^{n-k}}=F_{w^{(k)}}(q_{l})$
and 
\begin{equation*}
|x_{k,i-1}-x_{k,i}|\leq 2\rho _{v^{(k)}}\times 2^{k-n}\quad\text{for any}\quad 1\leq i\leq 2^{n-k}.
\end{equation*}

Then, by using the definition of $W$ and the assertion $\psi _{1}\geq \psi
_{4}$, we see there exists a constant $C>0$ such that, for any $1\leq k\leq
n $ and any $1\leq i\leq 2^{n-k},$ 
\begin{equation}
W(x_{k,i-1},|x_{k,i-1}-x_{k,i}|)\leq C\psi _{v^{(k)}}\times \psi
_{1}^{n-k}\leq C\psi _{1}^{n}.\text{ }  \label{ChBa1}
\end{equation}

Now consider the chain $\{x_{i}\} _{i=0}^{2^{n}}$ defined by 
\begin{equation*}
x_{2^{n}}=x,\quad x_{2^{n}-2^{n-k+1}+i}=x_{k,i}\quad\text{for any}\quad 1\leq k\leq n,\quad 0\leq
i\leq 2^{n-k},
\end{equation*}
then $\{ x_{i}\} _{i=0}^{2^{n}}\subset \triangle _{n}$ is
well-defined and satisfies the property (1) above.

We show that $\{ x_{i}\} _{i=0}^{2^{n}}$ also satisfies the property (2). Indeed, it follows by (\ref{ChBa1}) that
$$W(x_{i-1},|x_{i-1}-x_{i}|)\leq C\psi _{1}^{n}$$
for every $1\leq i\leq 2^{n}-1$. Noting that $x_{2^{n}-1}=F_{w}(q_{l})$ and $x_{2^{n}}=x\in F_{w}(\triangle)$, we have by definition of $W$ and the assertion $\psi_{1}\geq \psi _{4}$ that 
\begin{equation*}
W(x_{2^{n}-1},|x_{2^{n}-1}-x_{2^{n}}|)\leq W(F_{w}(q_{l}),\rho _{w})=\psi
_{w}\leq \psi _{1}^{n},
\end{equation*}
thus proving ii). Hence the chain $\{x_{i}\}_{i=0}^{2^{n}}$ is
what we desired in this case.

Case $(2)$: $\psi _{1}<\psi _{4}$. Let $\eta :=\log _{\psi _{1}}\psi _{4}$
and denote by $s_{k}:=N_{4}(v^{(k)})$ for any $1\leq k\leq n$. It is not
hard to see $0<\eta <1$ and $s_{k}\leq k-1$ for any $1\leq k\leq n.$

It is easy to see that, for any $1\leq k\leq n$, there exists a chain $\{x_{k,i}\} _{i=0}^{2^{n-\lceil\eta (k-1)\rceil }}$
such that $x_{k,0}=F_{w^{(k-1)}}(q_{l}),x_{k,2^{n-\lceil\eta(k-1)\rceil +1}}=F_{w^{(k)}}(q_{l})$ and 
\begin{equation*}
|x_{k,i-1}-x_{k,i}|\leq 2\rho _{v^{(k)}}\times 2^{^{\left\lceil \eta
(k-1)\right\rceil -n}}=2\rho _{1}^{n-\left\lceil \eta (k-1)\right\rceil
+k-s_{k}}\rho _{4}^{s_{k}}\quad\text{for any}\quad 1\leq i\leq 2^{n-\left\lceil \eta
k\right\rceil +1}.
\end{equation*}
Then it follows by using the definition of $W$ and the assertion $\psi
_{1}<\psi _{4}$ that, there exists a constant $C>0$ such that, for any $1\leq k\leq n$ and any $1\leq i\leq 2^{n-\lceil \eta (k-1)\rceil+1},$ 
\begin{equation}
W(x_{k,i-1},|x_{k,i-1}-x_{k,i}|)\leq C\psi _{1}^{n-\left\lceil \eta
(k-1)\right\rceil +k-s_{k}}\psi _{4}^{s_{k}}<C\psi _{1}^{n-\left\lceil \eta (k-1)\right\rceil +1}\left(\psi _{1}^\eta\right)^{k-1}\leq C\psi _{1}^{n}.\label{ChBa2}
\end{equation}

Now we consider the chain linking $F_{w}(q_{l})$ and $x$. Similarly, we can
find a chain $\{x_{n,i}\} _{i=0}^{2^{n-\lceil \eta
n\rceil }}$ such that $x_{n,0}=F_{w}(q_{l}),x_{k,2^{n-\lceil \eta
n\rceil }}=x$ and 
\begin{equation*}
|x_{n,i-1}-x_{n,i}|\leq 2\rho _{w}\times 2^{^{\left\lceil \eta
(k-1)\right\rceil -n}}=2\rho _{1}^{n-\left\lceil \eta n\right\rceil
+n-N_{4}(w)}\rho _{4}^{N_{4}(w)}\quad\text{for any}\quad 1\leq i\leq 2^{n-\left\lceil
\eta k\right\rceil +1}.
\end{equation*}
It follows that, for any $1\leq i\leq n-\lceil \eta n\rceil $, 
\begin{equation}
W(x_{n,i-1},|x_{n,i-1}-x_{n,i}|)\leq C\psi _{1}^{n-\left\lceil \eta
n\right\rceil +n-N_{4}(w)}\psi _{4}^{N_{4}(w)}< C\psi _{1}^{n-\left\lceil \eta n\right\rceil }\left(\psi _{1}^{\eta}\right)^n\leq
C\psi _{1}^{n-1}=\frac{C}{\psi _{1}}\psi _{1}^{n}, \label{ChBa3}
\end{equation}
where $C>0$ is the same constant in (\ref{ChBa2}). Linking $%
F_{w^{(k-1)}}(q_{l})$ and $F_{w^{(k)}}(q_{l})$ by chain $\{
x_{k,i}\} _{i=0}^{2^{n-\lceil \eta (k-1)\rceil }}$ in
sequence with respect to $1\leq k$ $\leq n$, and finally linking $%
F_{w^{(k-1)}}(q_{l})$ and $F_{w^{(k)}}(q_{l})$ by chain $\{
x_{n,i}\} _{i=0}^{2^{n-\lceil \eta n\rceil }}$, we obtain a
new chain $\{x_{i}\} _{i=0}^{L}$ linking $q_{l}$ and $x$, where 
\begin{equation*}
L=\dsum\limits_{k=0}^{n}2^{n-\left\lceil \eta k\right\rceil }\leq
2^{n+1}\dsum\limits_{k=0}^{n}2^{-\eta k}\leq \frac{2^{n+1}}{1-2^{-\eta }}
\leq \left\lceil \frac{2}{1-2^{-\eta }}\right\rceil 2^{n},
\end{equation*}
thus proving i). On the other hand, it follows by (\ref{ChBa2}) and (\ref{ChBa3}) that $\{x_{i}\}_{i=0}^{L}$ satisfies condition ii).
The proof is complete.
\end{proof}

\begin{proposition}
\label{LEGEN}Fix $\lambda \in (0,1/2)$ and $s\in (0,1)$, let $%
p_{t}(x,y):=p_{t}^{(\lambda ,s)}(x,y)$ be the heat kernel admitted in
Corollary \ref{KlHK}. Then there exists two constants $c,C>0$ such that, for
all $0<t<1$ and every $x,y\in K$, 
\begin{equation}
p_{t}(x,y)\geq \frac{C^{-1}}{V\left( x,W^{-1}(x,t)\right) }\exp \left\{
-c\left( \frac{W(x,d(x,y))}{t}\right) ^{\frac{1}{\beta _{1}-1}}\right\} ,
\label{legen}
\end{equation}
where $\beta _{1}$ is the constant defined in (\ref{bt1-}).
\end{proposition}

\begin{proof}
Fix two points $x,y\in K$ and assume that $x\neq y$, (otherwise \eqref{legen} holds
trivially from $(\mathrm{NLE})$). We distinguish two cases.

Case $(1)$: There exists some integer $n\geq 1$ such that $x,y\in \triangle
_{n}$. Let $w$ be the longest finite word (possibly the empty word) such that $x, y \in F_w(\triangle)$. Then there exist two distinct digits $i, j \in \{1,2,3,4\}$ such that $x \in F_{wi}(\triangle)$ and $y \in F_{wj}(\triangle)$. Let $z$ be the unique intersecting point of $F_{wi}(\triangle)$ and $F_{wj}(\triangle)$ (whose existence and uniqueness are guaranteed by Lemma \ref{GeoRT2}). 

By geometry on $F_i(\triangle)$'s and self-similarity, the ratios $d(x,z)/d(x,y)$ and $d(y,z)/d(x,y)$ are uniformly bounded. Consequently, there exists a constant $C_0 > 0$, depending only on $\lambda$, such that
\begin{equation*}
\max \{d(x,z),d(y,z)\}\leq C_{0}d(x,y).
\end{equation*}

Hence, it follows by the scaling property of $W$ that, there exists a
constant $C_{1}>0$ such that 
\begin{equation*}
C_{1}^{-1}W(x,d(x,y))\leq \max \{W(x,d(x,z)),W(y,d(y,z))\}\leq
C_{1}W(x,d(x,y)).
\end{equation*}

Without loss of the generality, assume that $W(x,d(x,z))\geq W(y,d(y,z))$.
Similarly as above, let $v$ be the longest finite word such that $x,z\in
K_{v}$. It is also easy to see that $\frac{\sqrt{3}}{2}\rho _{v}\leq
d(x,z)\leq \rho _{v}$, then it follows by definition of $W$ that there
exists a constant $C_{2}>0$, such that 
\begin{equation*}
C_{2}^{-1}\psi _{v}\leq W(x,d(x,z))\leq \psi _{v}.
\end{equation*}

Then it follows from the proof of Proposition \ref{Chain} that, there exist
an integer $A\geq 1$ and a constant $C>0$ such that, for any $n\geq 1$,
there exists two chains $\{x_{i}\}_{i=0}^{A2^{n}},\{z_{x,i}\} _{i=0}^{A2^{n}}$ satisfying $x_{0}=x,x_{A2^{n}}=$ $%
z_{0}=F_{v}(q_{1}),z_{A2^{n}}=z$ and 
\begin{eqnarray*}
W(x_{i-1},|x_{i-1}-x_{i}|) &\leq &C\psi _{v}\psi _{1}^{n}\leq
CC_{1}C_{2}W(x,d(x,y))\psi _{1}^{n}, \\
W(z_{x,i-1},|z_{x,i-1}-z_{x,i}|) &\leq &C\psi _{v}\psi _{1}^{n}\leq
CC_{1}C_{2}W(x,d(x,y))\psi _{1}^{n}
\end{eqnarray*}
for any $1\leq i\leq A2^{n}$. Similarly there is a chain connecting $z,y$. Coupling the two chains, we can finally obtain a chain $\{z_{i}\}_{i=0}^{4A2^{n}}$ such that $z_{0}=x,z_{4A2^{n}}=y$ and 
\begin{equation*}
W(z_{i-1},|z_{i-1}-z_{i}|)\leq C\psi _{v}\psi _{1}^{n}\leq
CC_{1}C_{2}W(x,d(x,y))\psi _{1}^{n}\quad\text{for any}\quad 1\leq i\leq 4A2^{n}.
\end{equation*}
Taking 
$$n_{0}:=\left\lceil C^{\prime}\left(\frac{W(x,d(x,y))}{t}\right) ^{\frac{1}{\beta _{1}-1}}\right\rceil$$
with some constant $C^{\prime }>0$ independent of $x,y$ and $t$, then under assumptions $\beta _{1}>1$ and $\psi _{1}^{n_{1}}\lesssim t$, we have \eqref{n0}. Hence, by a similar argument as in the proof of Proposition \ref{thm:main3}, we finally obtain (\ref {legen}).

Case $(2)$: The remaining case. Since $\triangle ^{\ast }=\cup _{n\geq
1}\triangle _{n}$ is a dense subset of $K$, it follows that, for any $n\geq
1 $, there exists $x_{n},y_{n}\in \triangle _{n}$ such that $|x_{n}-x|\leq
2^{-n},|y_{n}-y|\leq 2^{-n}$. On one hand, it follows from the proof of
above Case $(1)$ in Proposition \ref{LEGEN} that 
\begin{equation}
p_{t}(x_{n},y_{n})\geq \frac{C^{-1}}{V\left( x_{n},W^{-1}(x_{n},t)\right) }
\exp \left\{ -c\left( \frac{W(x_{n},d(x_{n},y_{n}))}{t}\right) ^{\frac{1}{
\beta _{1}-1}}\right\} .  \label{ptnet}
\end{equation}
for all $0<t<1$ and any $n\geq 1$, where $c,C>0$ are two constants
independent of $x,y,t$ and $n$.

On one hand, since $p_{t}(\cdot ,\cdot)\in C(M\times M)$ for any $t>0$, it
follows that 
\begin{equation}
p_{t}(x_{n},y_{n})\rightarrow p_{t}(x,y)\quad\text{as}\quad n\rightarrow \infty .
\label{ptnet0}
\end{equation}

On the other hand, by the scaling property of $W$ and $V$, there exists a
constant $C_{1}>0$ such that, for all $0<t<1$,
$$C_{1}^{-1}\leq \liminf_{n\rightarrow \infty }\frac{W(x_{n},d(x_{n},y_{n}))}{
W(x,d(x,y))}\leq \limsup_{n\rightarrow \infty }\frac{W(x_{n},d(x_{n},y_{n})) 
}{W(x,d(x,y))}\leq C_{1}$$
and 
$$C_{1}^{-1}\leq \liminf_{n\rightarrow \infty }\frac{V\left(
x_{n},W^{-1}(x_{n},t)\right) }{V\left( x,W^{-1}(x,t)\right) }\leq
\limsup_{n\rightarrow \infty }\frac{V\left( x_{n},W^{-1}(x_{n},t)\right) }{
V\left(x,W^{-1}(x,t)\right) }\leq C_{1}.$$
Combining (\ref{ptnet0}) and letting $n\rightarrow\infty$, we finally obtain (\ref{legen}) again in this case. The proof is complete.
\end{proof}

Now we show the optimal estimates of $p_{t}(x,y)$ provided that $\beta_{1}\geq\beta_{4}$.

\begin{theorem}
\label{le1-} Fix $\lambda \in (0,1/2)$ and $s\in (0,1)$. Take a doubling
self-similar measure $(\mu ,\{p_{i}\}_{i=1}^{4})$ such that $\beta _{1}\geq
\beta _{4}$. Then $\mathcal{E}$ admits a heat kernel $p_{t}(x,y)$ with
respect to $\mu $, and there exists three constants $C,c_{-},c_{+}>0$ such
that for all $0<t<1$ and every $x,y\in K$, 
\begin{equation}
p_{t}(x,y)\asymp \frac{C}{V\left( x,W^{-1}(x,t)\right) }\exp \left\{
-c_{\bullet }\left( \frac{W(x,d(x,y))}{t}\right) ^{\frac{1}{\beta _{1}-1}
}\right\} ,  \label{ue2+}
\end{equation}

where $c_{\bullet }$ has subscript ``$+$'' for the upper
bound and ``$-$'' for the lower bound.
\end{theorem}

\begin{proof}
Since $\beta _{1}\geq \beta _{4}$, we have by Corollary \ref{KlHK} that the
upper bound of (\ref{ue2+}) holds trivially from $(\mathrm{UE}_{\exp }^{W})$
by noting that $\beta ^{\ast }=\beta _{1}$. On the other hand, the lower
bound of (\ref{ue2+}) follows directly by Proposition \ref{LEGEN}. The proof is
complete.
\end{proof}

Clearly (\ref{ue2+}) gives an optimal two-sided estimate. Note that here we give an off-diagonal lower estimate even when $\beta _{\ast }=\beta _{4}<1$. Indeed $\beta _{\ast }$ disappears completely.

The case $\beta _{1}\leq \beta _{4}$ is extremely complicated, where the
estimate on the length of a nearly shortest chain depends irregularly on
(the dyadic representation of) $\lambda $.

\appendix
\section{Proof Details in Section \ref{mainresults}}\label{app1}
For simplicity, since the scaling function $W$ (or $F$) is already clear from the context, we omit the superscript from every involved condition.

\subsection{Proof of Proposition \ref{thm:main1}}\label{pf1}
Recalling from the sketch, it suffices to prove the following four lemmas:
\begin{lemma}\label{pf1-1}
$(\mathrm{S})+(\mathrm{sloc})\Rightarrow(\mathrm{T}_{\exp})$.
\end{lemma}

\begin{proof}
The argument is lengthy, so we divide it into the following six steps.
	
	\textbf{Step 1.} Let $B:=B(x_{0},r)$ be a ball, and let $\{\lambda R_{\lambda}^{B}\}_{\lambda >0}$ be the resolvent associated with the Dirichlet form $(\mathcal{E},\mathcal{F}(B))$. We claim that if $(\mathrm{S})$ holds, then there exist two constants $\varepsilon _{1}\in (0,1)$ and $L>0$ such that for all $\lambda \geq L/W(x_{0},r)$,
	\begin{equation}
		\einf_{\delta _{S}B}\lambda R_{\lambda }^{B}1_{B}\geq \varepsilon _{1}. \label{s1res}
	\end{equation}
	
	Indeed, we have by $(\mathrm{S})$ that
	\begin{equation*}
		\lambda R_{\lambda }^{B}1_{B}(y)=\lambda \int_{0}^{\infty }e^{-\lambda t}P_{t}^{B}1_{B}(y)dt\geq \lambda \int_{0}^{\delta W(x_0,r)}e^{-\lambda t}\varepsilon dt=\varepsilon (1-\exp (-\lambda\delta W(x_0,r)))
	\end{equation*}
	for $\mu$-a.e.\ $y\in \delta _{S}B$ and for any $\lambda >0$, where $\varepsilon,\delta,\delta _{S}$ are the constants from condition $(\mathrm{S})$. Consequently, when $\lambda W(x_{0},r)>L:=(\ln 2)/\delta$, we obtain
	\begin{equation*}
		\einf_{\delta _{S}B}\lambda R_{\lambda }^{B}1_{B}\geq \varepsilon /2=:\varepsilon_{1},
	\end{equation*}
	which proves the claim.
	
	\textbf{Step 2.} Consider the constants $\varepsilon _{1}\in (0,1),L>0$ from Step 1. We claim that under $(\mathrm{sloc})$, if there exists a function $w\in \mathcal{F}\cap L^{\infty }(M,\mu)$ such that $0\leq w\leq 1$ in a ball $B:=B(x_{0},r)$, and $w$ is a weak solution in $B$ of the equation
	\begin{equation*}
		-\mathcal{L}w+\lambda w=0,
	\end{equation*}
	where $\mathcal{L}$ is the infinitesimal generator corresponding to $(\mathcal{E},\mathcal{F})$, and $\lambda $ satisfies
	\begin{equation*}
		\lambda W(B)\geq L,
	\end{equation*}
	then
	\begin{equation*}
		\esup_{\delta_{S}B}w\leq 1-\varepsilon _{1}.
	\end{equation*}
	Indeed, using precompactness of open balls and $(\mathrm{sloc})$, we obtain by \cite[Corollary 4.15]{GrigoryanHu.2008.IM81} that
	\begin{equation*}
		w\leq 1-\lambda R_{\lambda }^{B}1_{B}
	\end{equation*}
	$\mu$-a.e. in $\delta_{S}B$. Combining \eqref{s1res}, the claim is proved.
	
	\textbf{Step 3.} We show that there exist constants $C^{\prime},c^{\prime }>0$ such that for any ball $B:=B(x_{0},r)$ and any $\lambda >0$,
	\begin{equation}
		\esup_{\frac{1}{2n}B}\lambda R_{\lambda }1_{B^{c}}\leq C^{\prime }\exp\left(-c'(\lambda W(B))^{1/\beta ^{\ast }}\right),  \label{s3conc}
	\end{equation}
	where $n\ge 2$ is an integer depending only on $\lambda W(B)$.
	
	In fact, choose a real number $R$ such that $R>\delta_S^{-1}r$, and introduce functions
	\begin{equation*}
		\phi=\mathbf{1}_{B(x_0,R)\setminus B}\quad\text{and}\quad
		u=\lambda R_{\lambda}\phi.
	\end{equation*}
	It suffices to prove
	\begin{equation}  \label{s3ues}
		\esup_{\frac{1}{2n}B}u\leq C^{\prime }\exp(-c^{\prime }(\lambda W(B))^{1/\beta^{\ast}}),
	\end{equation}
	and then let $R\rightarrow\infty$ to complete this step. Since $0\leq\phi\leq 1$ and $\phi\in L^2(M,\mu)$, we have $0\leq u\leq 1$ on $M$, $u\in\mathrm{dom}(\mathcal{L})\subset\mathcal{F}$, and $u$ satisfies on $M$ the equation
	\begin{equation}  \label{s3eq}
		-\mathcal{L}u+\lambda u=\lambda\phi.
	\end{equation}
	We may assume
	\begin{equation}  \label{s3nwelldef}
		\lambda W(B)\geq L_0
	\end{equation}
	with some $L_0>0$ to be determined later, since otherwise \eqref{s3ues} holds trivially due to $u\leq 1$.
	
	Let $n\ge 2$ be an arbitrary integer. For each $1\leq i\leq n $, set
	\begin{equation*}
		r_{i}:=\frac{ir}{n},\quad b_{i}:=\esup_{B(x_{0},r_{i})}u,\quad w_{i}(x):=\frac{u(x)}{b_{i+1}}.
	\end{equation*}
	Clearly, for each $1\leq i\leq n $, $w_{i}\in \mathcal{F}\cap L^{\infty }(M,\mu)$. Since $\phi \equiv 0 $ in $B$, equation \eqref{s3eq} implies that for each $1\leq i<n$, $w_{i}$ satisfies
	\begin{equation*}
		-\mathcal{L}w_{i}+\lambda w_{i}=0
	\end{equation*}
	in $B(x_{0},r)$. By the definition of $b_{i}$, we have $0\leq w_{i}\leq 1$ in the ball $B(x_{0},r_{i+1})$. In particular, the same inequality holds in any ball $B(x,r_{1})$ for $x\in B(x_{0},r_{i})$.
	
	Now we select proper $n$ such that
	\begin{equation}\label{key-i}
	\lambda W(x,r_{1})\geq L \quad \text{for all}\quad x\in B.
	\end{equation}
	Indeed, by the scaling property of $W$, for any $x\in B(x_{0},r)$,
	\begin{equation*}
	\frac{W(x_{0},r)}{W(x,r/n)}\leq C_{W}n^{\beta ^{\ast }}.
	\end{equation*}
	Therefore, \eqref{key-i} holds by choosing
	\begin{equation}
		n=\left\lfloor \left( \frac{\lambda W(x_{0},r)}{C_{W}L}\right) ^{1/\beta^{\ast }}\right\rfloor,\label{s3n}
	\end{equation}
which satisfies $n\geq 2$ when $L_{0}=2C_{W}L$.

Applying Step 1 with $\varepsilon _{2}:=1-\varepsilon _{1}$ on every $B(x,r_1)$ with $x\in B(x_0,r_i)\subset B$, and using standard covering techniques, it follows that $w_{i}\leq \varepsilon_{2}$ holds $\mu$-a.e. on $B(x,\delta _{S}r_{1})$. In particular,
$$b_{i}=b_{i+1}\esup_{B(x_{0},r_{i})}w_i\leq \varepsilon _{2}b_{i+1}$$
for every $1\leq i<n$. Iterating this inequality and noting $b_{n}\leq 1$, we finally obtain $b_{1}\leq \varepsilon _{2}^{n-1}b_{n}\leq \varepsilon _{2}^{n-1}$. Hence
$$\esup_{\frac{1}{2n}B}\lambda R_{\lambda }u\leq b_{1}\leq\varepsilon _{2}^{n-1}\leq C\exp(-cn)\leq C^{\prime}\exp (-c^{\prime }(\lambda W(x_{0},r))^{1/\beta^{\ast }}).$$
Clearly, this implies \eqref{s3ues} -- with the integer $n$ given by \eqref{s3n}.
	
	\textbf{Step 4.} We show that there exists a positive constant $L_1\geq 1$ such that for any ball $B:=B(x_0,r)$ with
	\begin{equation}  \label{s4c}
		\lambda W(x_0, r)\geq L_1,
	\end{equation}
	we have
	\begin{equation}  \label{s4conc}
		\einf_{x\in (2B)^{c}}\lambda R_{\lambda }1_{B^{c}}(x)\geq \varepsilon_{1},
	\end{equation}
	where $\varepsilon_1$ is the constant from \eqref{s1res}.
	
	Indeed, let $L$ be the other constant from \eqref{s1res}, and fix an arbitrary point $x\in (2B)^{c}$. Since $d(x,x_{0})>2r$, it is easy to verify
	\begin{equation*}
		r\leq \frac{d(x,x_{0})}{2}\leq d(x,x_{0})-r.
	\end{equation*}
	Therefore, by the scaling property \eqref{scpr} of $W$,
	\begin{eqnarray*}
		W(x_{0},r) &\leq &W(x_{0},d(x_{0},x)/2)\leq C_{W}2^{-\beta _{\ast }}W(x_{0},d(x_{0},x))\leq C_{W}^{2}2^{-\beta _{\ast }}W(x,d(x_{0},x))\\
		&\leq &C_{W}^{3}2^{\beta ^{\ast }-\beta _{\ast }}W(x,d(x_{0},x)/2)\leq C_{W}^{3}2^{\beta ^{\ast }-\beta _{\ast }}W(x,d(x_{0},x)-r).
	\end{eqnarray*}
	Combining this with \eqref{s4c} and setting $L_{1}:=C_{W}^{3}2^{\beta ^{\ast }-\beta_{\ast }}L$, we see
	$$\lambda W(x,d(x_{0},x)-r)\geq \lambda W(x_{0},r)\geq C_{W}^{-3}2^{\beta _{\ast }-\beta ^{\ast }}L_{1}=L.$$
	
	This implies, together with \eqref{s1res} and the fact $B(x,d(x_{0},x)-r)\subset B^{c}$ for any $x\in (2B)^{c}$, that
	\begin{equation*}
		\lambda R_{\lambda }1_{B^{c}}\geq \lambda R_{\lambda}^{B(x,d(x_{0},x)-r)}1_{B(x,d(x_{0},x)-r)}\geq \varepsilon _{1}
	\end{equation*}
	$\mu$-a.e. in the ball $B(x,\delta _{S}(d(x_{0},x)-r))$ (and hence also $\mu$-a.e. in $B(x,\delta _{S}r)$). This proves \eqref{s4conc}.
	
	\textbf{Step 5.} Same as \cite[Step 4 of Theorem 3.4]{GrigoryanHu.2008.IM81}, we see for any non-negative function $f\in L^\infty(M,\mu)$ and any $t,\lambda>0$, the function $u=\lambda R_\lambda f$ satisfies the inequality
	\begin{equation}  \label{s5conc}
		P_t u\leq\exp(\lambda t)u
	\end{equation}
	$\mu$-a.e. in $M$.
	
	\textbf{Step 6.} We complete the proof of Lemma \ref{pf1-1}. By \eqref{s3conc} in Step 3, for every $\lambda >0$ and $u=\lambda R^{\lambda }1_{B^{c}}$, we have
	\begin{equation}
		\esup_{\frac{1}{2n}B}u\leq C^{\prime }\exp (-c^{\prime}(\lambda W(x_{0},r))^{1/\beta ^{\ast }}).  \label{s6stt}
	\end{equation}
	Now let $L_{1}$ be the constant from \eqref{s4c}. We distinguish two cases.
	
	Case (1): $\lambda W(x_{0},r)\geq L_{1}$. Note that $u\geq \varepsilon _{1}\mathbf{1}_{(2B)^{c}}$ by \eqref{s4conc}. Applying the operator $P_{t}$ to both sides and using \eqref{s5conc}, we see
	\begin{equation*}
		\varepsilon _{1}P_{t}\mathbf{1}_{(2B)^{c}}\leq P_{t}u\leq \exp (\lambda t)u.
	\end{equation*}
	Combining \eqref{s6stt}, we obtain
	\begin{equation}
		\esup_{\frac{1}{2n}B}P_{t}\mathbf{1}_{(2B)^{c}}\leq C_{0}\exp\left( \lambda t-c^{\prime }\left( \lambda W(B)\right) ^{1/\beta ^{\ast }}\right) ,  \label{harmo}
	\end{equation}
	where $C_{0}:=\varepsilon _{1}^{-1}C^{\prime}$.
	
	Case (2): $\lambda W(x_{0},r)<L_{1}$. Noting that $0\leq u\leq 1$ on $M$, inequality \eqref{harmo} holds again trivially with the constant $C_{0}:=\exp(c^{\prime }L_{1}^{1/\beta ^{\ast }})$.
	
	Thus, \eqref{harmo} holds for all $\lambda >0$. Consequently, by \eqref{scpr} and \eqref{s3n}, one can derive
	\begin{equation}
		\esup_{\frac{1}{2n}B}P_{t}\mathbf{1}_{B^{c}}\leq C_{0}^{\prime}\exp\left( \lambda t-c^{\prime \prime}\left( \lambda W(x_{0},r)\right) ^{1/\beta ^{\ast }}\right) ,  \label{harmo1}
	\end{equation}
	for some constants $C_{0}^{\prime}, c^{\prime \prime}>0$ similar to Step 3. The same holds on $\frac{1}{2}B$ by standard covering techniques. Finally, minimizing the right-hand side of \eqref{harmo1} over $\lambda >0$ yields \eqref{Tail}, which completes the proof.
\end{proof}

\begin{lemma}\label{pf1-2}
$(\mathrm{DUE})+(\mathrm{T}_{\exp})\Rightarrow(\mathrm{UE}_{\mu,\exp})$.
\end{lemma}

\begin{proof}
Let $x,y$ be two arbitrary distinct points in $M$, and fix any $0<t<W(x,\bar{R})\wedge W(y,\bar{R})$. To begin with, using the semigroup property, the Cauchy-Schwarz inequality and $(\mathrm{DUE})$ successively, we obtain 
\begin{eqnarray}
	p_t(x,y)&=&\int_K
	p_{t/2}(x,z)p_{t/2}(z,y)d\mu(z)\leq\left(\int_K(p_{t/2}(x,z))^2d\mu(z)%
	\right)^{1/2}\left(\int_K(p_{t/2}(y,z))^2d\mu(z)\right)^{1/2}  \notag \\
	&=&\sqrt{p_t(x,x)p_t(y,y)}\leq\frac{C}{\sqrt{V(x,W^{-1}(x,t))V(y,W^{-1}(y,t))%
	}}.  \label{eDUE}
\end{eqnarray}

Now let $D\geq 2$ be a number to be determined later. We distinguish two
cases.

Case $(1)$: $d(x,y)\leq D( W^{-1}(x,t)\vee W^{-1}(y,t)) $.
Indeed, by a similar argument as in \cite[Proposition 9.1]%
{GrigoryanHuHu.2022.TPLE}, there exists a constant $C>0$ such that 
\begin{equation}
	C^{-1}\leq \frac{W^{-1}(x,t)}{W^{-1}(y,t)},\quad\frac{V(x,W^{-1}(x,t))}{%
		V(y,W^{-1}(y,t))}\leq C.  \label{Wctrl}
\end{equation}%
Consequently, 
\begin{equation*}
	d(x,y)\leq D\left( W^{-1}(x,t)\vee W^{-1}(y,t)\right) \leq DCW^{-1}(x,t),
\end{equation*}%
which implies with the help of (\ref{scpr}) that $W(x,d(x,y))\leq C^{\prime }t$ with a
constant $C^{\prime }$ independent of $x,y,t$. Hence, by (\ref{eDUE}) and (%
\ref{Wctrl}), 
\begin{equation*}
	p_{t}(x,y)\leq \frac{C^{3/2}}{V(x,W^{-1}(x,t))}\leq \frac{C^{3/2}}{%
		V(x,W^{-1}(x,t))}\exp \left( C^{\prime \frac{1}{\beta ^{\ast }-1}}\right)
	\exp \left( -\left( \frac{W(x,d(x,y))}{t}\right) ^{\frac{1}{\beta ^{\ast }-1}%
	}\right) ,
\end{equation*}%
which shows $(\mathrm{UE}_{\exp })$ in this case.

Case $(2)$: $d(x,y)>D( W^{-1}(x,t)\vee W^{-1}(y,t)) $. For
simplicity, set $r_{xy}:=d(x,y)/D$, and denote $V_{z,t}:=V(z,W^{-1}(z,t))$
for any $z\in M$ and $t>0$. Note that 
\begin{equation}
	p_{t}(x,y)\leq \left( \int_{B(x,r_{xy})^{c}}+\int_{B(y,r_{xy})^{c}}\right)
	p_{t/2}(x,z)p_{t/2}(z,y)d\mu (z)=:I_{1}+I_{2}.  \label{DUEtoUE}
\end{equation}
We first estimate $I_{1}$. Indeed, note by $(\mathrm{VD})$ and \eqref{WinvScal} that, there exists a constant $C>0$ such that, for all $x,z\in M$ and all $t>0$ satisfying $d(x,z)>W^{-1}(z,t)$, 
\begin{equation*}
\frac{V(z,d(x,z))}{V(z,W^{-1}(z,t))}\leq C\left( \frac{d(x,z)}{W^{-1}(z,t)}\right) ^{\alpha ^{\ast }}\leq C^{2}\left( \frac{W(z,d(x,z))}{t}\right) ^{\alpha ^{\ast }/\beta_{\ast }}\leq C^{2+\alpha ^{\ast }/\beta _{\ast }}\left( \frac{W(x,d(x,z))}{t}\right) ^{\alpha ^{\ast }/\beta _{\ast }}.
\end{equation*}
Combining $(\mathrm{DUE})$, it follows that
\begin{eqnarray}
	I_{1} &=&\int_{B(x,r_{xy})^{c}}p_{t/2}(x,z)p_{t/2}(z,y)d\mu (z)\leq\frac{C}{\sqrt{V_{x,t}V_{y,t}}}\int_{B(x,r_{xy})^{c}}p_{t/2}(x,z)\sqrt{\frac{V_{x,t}}{V_{z,t}}}d\mu (z)  \notag \\
	&\leq &\frac{C}{\sqrt{V_{x,t}V_{y,t}}}\int_{B(x,r_{xy})^{c}}p_{t/2}(x,z)%
	\left( \frac{W(x,d(x,z))}{t}\right) ^{\theta }d\mu (z)=:\frac{C}{\sqrt{%
			V_{x,t}V_{y,t}}}J_{1},  \label{I1}
\end{eqnarray}%
where $\theta =\alpha ^{\ast }/2\beta _{\ast }$. For each $k\geq 0$, define 
\begin{equation*}
	A_{k}:=B(x,D^{k+1}r_{xy})\setminus B(x,D^{k}r_{xy}),\quad
	s_{k}:=t^{-1}W(x,D^{k-1}r_{xy}).
\end{equation*}%
Note that for all $k\geq 0$, 
\begin{equation}\label{sk}
s_{k}\geq s_{1}\quad\text{and}\quad\frac{s_{k}}{s_{k+1}}=\frac{W(x,D^{k-1}r_{xy})}{W(x,D^{k}r_{xy})}\leq CD^{-\beta _{\ast }}.
\end{equation}
Combining $(\mathrm{T}_{\exp })$, it follows that
\begin{eqnarray}
	J_{1} &=&\int_{B(x,r_{xy})^{c}}p_{t/2}(x,z)\left( \frac{W(x,d(x,z))}{t}\right) ^{\theta }d\mu (z)\leq \sum_{k=0}^{\infty
	}\int_{A_{k}}p_{t/2}(x,z)\left( \frac{W(x,d(x,z))}{t}\right) ^{\theta }d\mu(z)  \notag \\
	&\leq &\sum_{k=0}^{\infty }P_{t/2}1_{B(x,D^{k}r_{xy})^{c}}(x)\left( \frac{%
		W(x,D^{k+1}r_{xy})}{t}\right) ^{\theta }\leq \sum_{k=0}^{\infty }C^{\prime
	}s_{k+1}^{\theta }\exp \left( -cs_{k+1}^{\frac{1}{\beta ^{\ast }-1}}\right) 
	\notag \\
	&\leq &C^{\prime }\sum_{k=0}^{\infty }s_{k+1}^{\theta }\exp \left( -\frac{c}{%
		2}s_{1}^{\frac{1}{\beta ^{\ast }-1}}\right) \exp \left( -\frac{c}{2}s_{k+1}^{%
		\frac{1}{\beta ^{\ast }-1}}\right) =C^{\prime }\Phi
	(s_{1})\sum_{k=0}^{\infty }s_{k+1}^{\theta }\Phi (s_{k+1}),  \label{J1plus}
\end{eqnarray}%
where 
\begin{equation*}
	\Phi (s):=\exp \left( -\frac{c}{2}s^{\frac{1}{\beta ^{\ast }-1}}\right) .
\end{equation*}%
Since $\Phi $ is decreasing on $[0,\infty)$, for any $k\geq 0$, we have by (\ref{sk}) that 
\begin{eqnarray*}
	\int_{s_{k}}^{s_{k+1}}s^{\theta -1}\Phi (s)ds &\geq &\Phi
	(s_{k+1})\int_{s_{k}}^{s_{k+1}}s^{\theta -1}ds=\theta ^{-1}\Phi
	(s_{k+1})\left( s_{k+1}^{\theta }-s_{k}^{\theta }\right) \\
	&\geq &\theta ^{-1}\Phi (s_{k+1})\left( s_{k+1}^{\theta }-(CD^{-\beta _{\ast
	}})^{\theta }s_{k+1}^{\theta }\right) .
\end{eqnarray*}%
By choosing $D\geq 2$ sufficiently large such that $CD^{-\beta _{\ast }}\leq
1/2$, we obtain by summing up both sides over $k\geq 0$ that 
\begin{equation}
	\sum_{k=0}^{\infty }s_{k+1}^{\theta }\Phi (s_{k+1})\leq \frac{\theta }{%
		1-2^{-\theta }}\int_{s_{1}}^{\infty }s^{\theta -1}\Phi (s)ds\leq \frac{%
		\theta }{1-2^{-\theta }}\int_{0}^{\infty }s^{\theta -1}\Phi (s)ds<\infty .
	\label{infsumint}
\end{equation}%
Combining (\ref{I1}), (\ref{J1plus}), (\ref{infsumint}) with the fact that $%
s_{1}=W(x,r_{xy})/t=W(x,D^{-1}d(x,y))/t$, we obtain 
\begin{equation*}
	I_{1}\leq \frac{CC^{\prime }}{\sqrt{V_{x,t}V_{y,t}}}\exp \left( -\frac{c}{2}s_{1}^{\frac{1}{\beta ^{\ast }-1}}\right) \sum_{k=0}^{\infty
	}s_{k+1}^{\theta }\Phi (s_{k+1})\leq \frac{C^{\prime \prime }}{\sqrt{V_{x,t}V_{y,t}}}\exp \left( -c^{\prime }\left( \frac{W(x,d(x,y))}{t}\right)
	^{\frac{1}{\beta _{\ast }-1}}\right) .
\end{equation*}

Symmetrically, since $W(x,d(x,y))\asymp W(y,d(x,y))$, we see 
\begin{equation*}
	I_2\leq\frac{C^{\prime \prime }}{\sqrt{V_{x,t}V_{y,t}}}\exp\left(-c^{\prime
		\prime }\left(\frac{W(x,d(x,y))}{t}\right)^{\frac{1}{\beta_\ast-1}}\right),
\end{equation*}
yielding by (\ref{DUEtoUE}) that (say $c^{\prime \prime }<c^{\prime }$ without loss of generality)
\begin{eqnarray*}
	p_t(x,y)&\leq&\frac{2C^{\prime \prime }}{\sqrt{V_{x,t}V_{y,t}}}%
	\exp\left(-c^{\prime \prime }\left(\frac{W(x,d(x,y))}{t}\right)^{\frac{1}{%
			\beta_\ast-1}}\right) \\
	&\leq&\frac{2C^{\prime \prime }}{V_{x,t}}\left(\frac{W(x,d(x,y))}{t}%
	\right)^\theta\exp\left(-c^{\prime \prime }\left(\frac{W(x,d(x,y))}{t}%
	\right)^{\frac{1}{\beta_\ast-1}}\right) \\
	&\leq&\frac{2C^{\prime \prime }c(\theta)}{V(x,W^{-1}(x,t))}\exp\left(-\frac{%
		c^{\prime \prime }}{2}\left(\frac{W(x,d(x,y))}{t}\right)^{\frac{1}{%
			\beta_\ast-1}}\right),
\end{eqnarray*}
where 
\begin{equation*}
	c(\theta)=\sup\limits_{s>0}s^{\theta}\exp\left(-\frac{c^{\prime \prime }}{2}%
	s^{\frac{1}{\beta^\ast-1}}\right)<\infty.
\end{equation*}
Hence $(\mathrm{UE}_{\mu,\exp})$ is proved again.
\end{proof}
		
\begin{lemma}\label{pf1-3}
$(\mathrm{VD})+(\mathrm{UE}_{\mu,\exp})+(\mathrm{C}) \Rightarrow (\mathrm{sloc})$.
\end{lemma}

\begin{proof}
	Let $u,v\in \mathcal{F}$ be functions with disjoint compact supports $A=\supp%
	(u)$ and $B=\supp(v)$. Noticing that $(u,v)=0$, there is 
	\begin{equation*}
		\mathcal{E}_{t}(u,v):=\frac{1}{t}\left( u,v-P_{t}v\right) =-\frac{1}{t}\left(
		u,P_{t}v\right) =-\frac{1}{t}\int_{A}\int_{B}u(x)v(y)p_{t}(x,y)d\mu (y)d\mu
		(x).
	\end{equation*}%
	Let $R:=d(A,B)$ and $r:=\diam A$. Fix $x_{0}\in A$. By (\ref{scpr}) and $(%
	\mathrm{VD})$, for all $x\in A,y\in B$ and sufficiently small $t>0$, 
	\begin{eqnarray*}
		\frac{W(x,d(x,y))}{W(x_{0},r)} &\geq &\frac{W(x,R)}{W(x,r+R)}\cdot \frac{%
			W(x,r+R)}{W(x_{0},r)}\geq \frac{C_{W}^{-2}R^{\beta ^{\ast }}}{r^{\beta
				_{\ast }}(R+r)^{\beta ^{\ast }-\beta _{\ast }}}=:c_{R,r};\\
		\frac{V(x_{0},r)}{V(x,W^{-1}(x,t))} &\leq &\frac{V(x,2r)}{V(x,W^{-1}(x,t))}%
		\leq C\left( \frac{W(x,r)}{t}\right) ^{\alpha ^{\ast }/\beta _{\ast }}\leq
		C_{0}\left( \frac{W(x_{0},r)}{t}\right) ^{\alpha ^{\ast }/\beta _{\ast }}.
	\end{eqnarray*}%
	Consequently, by $(\mathrm{UE}_{\exp})$ we obtain 
	\begin{eqnarray*}
		\left\vert \mathcal{E}_{t}(u,v)\right\vert &\leq &\frac{1}{t}\int_{A}\int_{B}%
		\frac{C}{V(x,W^{-1}(x,t))}\Phi \left( \frac{W(x,d(x,y))}{t}\right)
		|u(x)||v(y)|d\mu (y)d\mu (x) \\
		&\leq &\frac{CC_{0}\Vert v\Vert _{1}}{t}\left( \frac{W(x_{0},r)}{t}\right)
		^{\alpha ^{\ast }/\beta _{\ast }}\Phi \left( \frac{c_{R,r}W(x_{0},r)}{t}%
		\right) \int_{A}\frac{|u(x)|d\mu (x)}{V(x_{0},r)} \\
		&\leq &\frac{c_{R,r}^{-\alpha ^{\ast }/\beta _{\ast }-1}CC_{0}\Vert u\Vert
			_{1}\Vert v\Vert _{1}}{V(x_{0},r)W(x_{0},r)}\left( \frac{c_{R,r}W(x_{0},r)}{t%
		}\right) ^{\alpha ^{\ast }/\beta _{\ast }+1}\Phi \left( \frac{%
			c_{R,r}W(x_{0},r)}{t}\right),
	\end{eqnarray*}
	where $\Phi(s):=\exp\left(-cs^{\frac{1}{\beta^\ast-1}}\right)$.
	Then, there exists an increasing chain $\{s_k\}_{k=1}^\infty$ such that
	\begin{equation*}
		s_k\to\infty,\quad s_k^{\alpha^\ast/\beta_\ast+1}\Phi(s_k)\to 0\quad\text{as}%
		\quad k\to\infty.
	\end{equation*}
	Now letting $t_k=c_{R,r}W(x_0,r)/s_k$ for each $k\geq 1$, we obtain (cf.\ \cite[Lemma 1.3.4]{FukushimaOshimaTakeda.2011.489})
	\begin{equation*}
		|\mathcal{E}(u,v)|=\lim_{k\to\infty}|\mathcal{E}_{t_k}(u,v)|=0,
	\end{equation*}
	that is, $(\mathcal{E},\mathcal{F})$ is local. This further implies strong locality due to $(\mathrm{C})$, thus obtaining what is desired.
\end{proof}

\begin{lemma}\label{pf1-4}
$(\mathrm{VD})+(\mathrm{UE}_{\mu,\exp})+(\mathrm{C})\Rightarrow (\mathrm{S})$.
\end{lemma}

\begin{proof}	
Under $(\mathrm{T}_{\exp})$, by a routine argument we see there exist constants $\varepsilon,\delta\in(0,1/4)$ such that for any ball $B:=B(x_0,r)$ and any $t\in(0,\delta W(x_0,r))$,
\begin{equation}
	P_t 1_{B^c} (x)<\varepsilon\quad\text{for}\quad\mu\text{-a.e. }\ x\in\frac{1}{4}B.
\end{equation}
Then, it follows by the argument in the proof of \cite[Proposition 5.8]{GrigoryanHu.2014.MMJ505} that $(\mathrm{S})$ holds. The proof is completed.
\end{proof}

\subsection{Proofs of Theorem \ref{thm:main2} and Proposition \ref{thm:main3}}\label{pf2}
Recall that, in the setting here, there exists a scaling function $F$ with scaling exponents $0<\gamma_\ast\leq\gamma^\ast$ such that \eqref{WF} holds, and hence (by Remark \ref{ctn-resist}) all functions $u\in\mathcal{F}$ and the heat kernel $p_t(x,y)$ admit continuous versions.

Recalling from the sketch, to prove Theorem \ref{thm:main2}, it suffices to prove the following three lemmas:

\begin{lemma}\label{pf2-1}
$(\mathrm{VD})+(\mathrm{MI})+(\mathrm{RVD})\Rightarrow (\mathrm{FK})$.
\end{lemma}

\begin{proof}
	Since both $(\mathrm{VD})$, $(\mathrm{RVD})$ hold and $\supp(\mu)=M$, it
	follows that there exists a constant $\varepsilon _{0}\in (0,1)$ such that
	for every open ball $B$ with radius $r\in (0,\bar{R})$, 
	\begin{equation}
		B\setminus \varepsilon _{0}B\neq \emptyset .  \label{UP}
	\end{equation}
	
	Fix a ball $B:=B(x_{0},r)$ with $x_{0}\in M$ and $0<r<(\frac{1}{3}\wedge 
	\frac{\iota }{2})\varepsilon _{0}\bar{R}$.
	
	Since here $\varepsilon _{0}^{-1}r<\bar{R}/3$, it is easy to check that $%
	M\setminus \varepsilon _{0}^{-1}B\neq \emptyset $. Fix a point $\hat{x}\in
	\varepsilon _{0}^{-1}B\setminus B$, which exists by (\ref{UP}) with $%
	\varepsilon _{0}^{-1}B$ replacing $B$.
	
	For any non-empty open set $U\subset B$, take an arbitrary $u\in \mathcal{F}\cap
	C_{0}(U)$ such that $\Vert u\Vert _{\infty }=1$. Then there exists $y_{0}\in
	U$ such that $|u(y_{0})|=1$. Noting that $U\subset B$ and $\hat{x}\in B^{c}$%
	, we have $u(\hat{x})=0$. Since $d(\hat{x},y_{0})\leq (\varepsilon
	_{0}^{-1}+1)r<\iota \bar{R}$, it follows by $(\mathrm{MI})$ that 
	\begin{eqnarray*}
		1&=&|u(\hat{x})-u(y_{0})|^{2}\leq CF(\hat{x},d(\hat{x},y_{0}))\mathcal{E}(u)\\
		&\leq& CF(\hat{x},(\varepsilon _{0}^{-1}+1)r)\mathcal{E}(u)\leq CC_{F}(\varepsilon _{0}^{-1}+1)^{\gamma ^{\ast }}F(x_{0},r)\mathcal{E}(u).
	\end{eqnarray*}%
	On the other hand, using the fact that $\supp(u)\subset U$, we see 
	\begin{equation*}
		\Vert u\Vert _{2}^{2}=\int_{U}u^{2}d\mu \leq \mu (U)=V(x_{0},r)\left( \frac{\mu (U)}{\mu (B)}\right) .
	\end{equation*}%
	Combining the two inequalities above, we obtain 
	\begin{equation*}
		\frac{\mathcal{E}(u)}{\Vert u\Vert _{2}^{2}}\geq \frac{1}{CC_{F}(\varepsilon
			_{0}^{-1}+1)^{\gamma ^{\ast }}F(x_{0},r)V(x_{0},r)}\cdot \left( \frac{\mu (U)%
		}{\mu (B)}\right) ^{-1}\geq \frac{c}{W(x_{0},r)}\left( \frac{\mu (B)}{\mu (U)%
		}\right) ,
	\end{equation*}%
	thus showing $(\mathrm{FK})$ with $\varepsilon =(\frac{1}{3}\wedge \frac{\iota }{2})\varepsilon _{0}$ and $\nu =1$.
\end{proof}

\begin{lemma}\label{pf2-2}
$(\mathrm{S})+(\mathrm{MI})+(\mathrm{DUE})\Rightarrow(\mathrm{NLE})$.
\end{lemma}

\begin{proof}
	Fix $x\in M$ and $0<t<W^{-1}(x,\bar{R})$, and let $\eta \in (0,1)$ be a
	constant to be determined. Noting that $p_{t}(x,\cdot)\in \mathcal{F}$, we
	obtain by the spectral decomposition of $\mathcal{E}$ that 
	\begin{equation}
		\mathcal{E}(p_{t}(x,\cdot))\leq \frac{1}{et}p_{t}(x,x).  \label{pt-sp}
	\end{equation}%
	It follows by $(\mathrm{MI}^F)$, \eqref{WF} and the scaling property of $F$ that, for all $y\in B(x,\eta W^{-1}(x,t))$, 
	\begin{eqnarray*}
		|p_{t}(x,x)-p_{t}(x,y)|^{2} &\leq &CF(x,d(x,y))\mathcal{E}(p_{t}(x,\cdot
		))\leq \frac{CF(x,\eta W^{-1}(x,t))}{et}p_{t}(x,x) \\
		&\leq &CC_F\eta ^{\gamma _{\ast }}\frac{F(x,W^{-1}(x,t))}{W(x,W^{-1}(x,t))}%
		p_{t}(x,x)\leq \frac{C^{2}C_F\eta ^{\gamma _{\ast }}}{V(x,W^{-1}(x,t))}%
		p_{t}(x,x).
	\end{eqnarray*}%
	Combining this inequality with $(\mathrm{DUE})$, we obtain 
	\begin{equation}
		p_{t}(x,y)\geq p_{t}(x,x)-\frac{CC_F^{1/2}\eta ^{\gamma _{\ast }/2}}{V(x,W^{-1}(x,t))}.  \label{nle2}
	\end{equation}
	
	Meanwhile, for every ball $B:=B(x,r)$ with $0<r<\bar{R}$ and every $0<t\leq \delta W(B)$, by the semi-group property of $p_{t}$, Cauchy-Schwarz inequality and $(\mathrm{S})$, we have 
	\begin{eqnarray*}
		p_{t}(x,x) &=&\int_{M}(p_{t/2}(x,z))^{2}d\mu (z)\geq
		\int_{B}(p_{t/2}(x,z))^{2}d\mu (z) \\
		&\geq &\frac{1}{\mu (B)}\left( \int_{B}p_{t/2}(x,z)d\mu (z)\right) ^{2}\geq 
		\frac{1}{V(x,r)}\left( P_{t/2}^{B}\mathbf{1}_{B}(x)\right) ^{2}\geq \frac{%
			\varepsilon ^{2}}{V(x,r)}.
	\end{eqnarray*}%
	In particular, with $t=\delta W(x,r)$, by $(\mathrm{VD})$ and (\ref{scpr})
	we obtain 
	\begin{equation*}
		p_{t}(x,x)\geq \frac{\varepsilon ^{2}}{V(x,W^{-1}(x,t/\delta))}\geq \frac{%
			\varepsilon ^{2}\delta ^{\alpha ^{\ast }/\beta _{\ast }}}{V(x,W^{-1}(x,t))}.
	\end{equation*}%
	Inserting this inequality into \eqref{nle2}, we obtain for all $x\in M$ and $%
	0<t<\delta W(x,\bar{R})$ that 
	\begin{equation*}
		p_{t}(x,\cdot)\geq \frac{\varepsilon ^{2}\delta ^{\alpha ^{\ast }/\beta_{\ast }}-CC_F^{1/2}\eta ^{\gamma _{\ast }/2}}{V(x,W^{-1}(x,t))}\quad \text{on}\quad B(x,\eta W^{-1}(x,t)).
	\end{equation*}%
	By choosing $\eta >0$ sufficiently small such that $\varepsilon
	^{2}-CC_F^{1/2}\eta ^{\gamma _{\ast }/2}\geq \varepsilon ^{2}/2$ and applying
	the semigroup property to cover general $t\in (0,W(x,\bar{R}))$, we finally
	obtain $(\mathrm{NLE})$.
\end{proof}

\begin{lemma}\label{pf2-3}
$(\mathrm{VD})+(\mathrm{NLE})\Rightarrow(\mathrm{MI})$.
\end{lemma}

\begin{proof}
	Indeed, noting by $(\mathrm{NLE})$, there exists a constant $C>0$ such that, for any $u\in\mathcal{F}$, 
	\begin{equation}\label{spem}
		\mathcal{E}(u)\geq \sup_{s\in(0,R_{W})}\int_{M}\int_{B(x,W^{-1}(x,s))}\frac{(u(x)-u(y))^{2}}{sV(x,W^{-1}(x,s))}d\mu (y)d\mu (x),
	\end{equation}
	where $R_{W}:=\inf_{x\in M}W(x,\bar{R})\in \lbrack 0,\infty].$ 
	Note that $R_W=\infty$ if $\bar{R}=\infty$, while if $0<\bar{R}<\infty$, by (\ref{scpr}), with any $x_0\in M$, we have 
	\begin{equation*}
		\infty>W(x_0,\bar{R})\geq R_W\geq C_W^{-1}\sup_{x\in M}W(x,\bar{R})>0.
	\end{equation*}

	Fix $u\in \mathcal{F}\cap C_{0}(M)$. For any $x\in M$ and any $r>0$, let 
	\begin{equation*}
		u_{r}(x):=\fint_{B(x,r)}u(\xi)d\mu (\xi).
	\end{equation*}%
	For any $r\in (0,\bar{R})$ and any $x,y\in M$ with $d(x,y)\leq r$, 
	\begin{equation*}
		|u_{r}(x)-u_{r}(y)|^{2}=\left\vert \fint_{B(x,r)}\fint_{B(y,r)}(u(\xi
		)-u(\zeta))d\mu (\zeta)d\mu (\xi)\right\vert ^{2}\leq \fint_{B(x,r)}\fint%
		_{B(y,r)}|u(\xi)-u(\zeta)|^{2}d\mu (\zeta)d\mu (\xi).
	\end{equation*}%
	For any $\xi \in B(x,r)$, since 
	\begin{equation*}
		d(y,\xi)\leq d(y,x)+d(x,\xi)<2r\quad \text{and}\quad
		3r=C_{W}^{-1}W^{-1}(x,s)\leq W^{-1}(\xi ,s)
	\end{equation*}%
	by \cite[Proposition 9.1]{GrigoryanHuHu.2022.TPLE}, we see 
	\begin{equation*}
		B(y,r)\subset B(\xi ,3r)\subset B\left( \xi ,W^{-1}(\xi ,s)\right) \quad 
		\text{with}\quad s=s(r)=W(x,3C_{W}r).
	\end{equation*}%
	Consequently, by $(\mathrm{VD})$, 
	\begin{equation*}
		\frac{V(\xi ,W^{-1}(\xi ,s))}{V(y,r)}\leq \frac{V(y,2r+C_{W}W^{-1}(x,s))}{%
			V(y,r)}\leq C(2+3C_{W}^{2})^{\alpha ^{\ast }}=:C^{\prime }.
	\end{equation*}%
	Therefore, for any $r\in (0,\bar{R})$ with $s(r)\in (0,R_{W})$, and any $%
	x,y\in M$ with $d(x,y)\leq r$, we obtain 
	\begin{eqnarray}
		|u_{r}(x)-u_{r}(y)|^{2} &\leq &\frac{1}{V(x,r)V(y,r)}\int_{B(x,r)}%
		\int_{B(y,r)}|u(\xi)-u(\zeta)|^{2}d\mu (\zeta)d\mu (\xi)  \notag \\
		&\leq &\frac{s}{V(x,r)}\int_{M}\int_{B(\xi ,W^{-1}(\xi ,s))}\frac{|u(\xi
			)-u(\zeta)|^{2}}{sV\left( \xi ,W^{-1}(\xi ,s)\right) }\frac{V(\xi
			,W^{-1}(\xi ,s))}{V(y,r)}d\mu (\zeta)d\mu (\xi)  \notag \\
		&\leq &\frac{C^{\prime }s}{V(x,r)}\int_{M}\int_{B(\xi ,W^{-1}(\xi ,s))}\frac{%
			|u(\xi)-u(\zeta)|^{2}}{sV\left( \xi ,W^{-1}(\xi ,s)\right) }d\mu (\zeta
		)d\mu (\xi)  \notag \\
		&\leq &\frac{C^{\prime }W(x,3C_{W}r)}{V(x,r)}E^{W,2}(u)\leq \frac{C^{\prime\prime
			}W(x,r)}{V(x,r)}\mathcal{E}(u)\leq C_{0}F(x,r)%
		\mathcal{E}(u),  \label{73tem4}
	\end{eqnarray}%
	where we used \eqref{spem} and \eqref{WF} for the last line.
	
	Similarly, for every $k\geq 0$, 
	\begin{equation}  \label{73tem6}
		\left|u_{2^{-k}r}(x)-u_{2^{-(k+1)}r}(x)\right|^2\leq
		C_0F(x,2^{-k}r)E^{W,2}(u)\leq C_12^{-k\gamma_\ast}F(x,r)\mathcal{E}(u).
	\end{equation}
	
	Since $u$ is continuous, we have by (\ref{73tem4}) and (\ref{73tem6}) (for
	both $x$ and $y$) that 
	\begin{eqnarray*}
		|u(x)-u(y)|&\leq&|u_r(x)-u_r(y)|+|u(x)-u_r(x)|+|u(x)-u_r(x)| \\
		&\leq&|u_r(x)-u_r(y)|+\sum_{k=0}^\infty%
		\left|u_{2^{-k}r}(x)-u_{2^{-(k+1)}r}(x)\right|+\sum_{k=0}^\infty%
		\left|u_{2^{-k}r}(y)-u_{2^{-(k+1)}r}(y)\right| \\
		&\leq&\sqrt{C_0F(x,r)\mathcal{E}(u)}+\sum_{k=0}^\infty\sqrt{%
			C_12^{-\gamma_\ast k}F(x,r)\mathcal{E}(u)}+\sum_{k=0}^\infty\sqrt{%
			C_12^{-\gamma_\ast k}F(y,r)\mathcal{E}(u)} \\
		&\leq&\left(\sqrt{C_0}+\frac{\sqrt{C_1}}{1-2^{-\gamma_\ast/2}}\right)\sqrt{%
			F(x,r)\mathcal{E}(u)}+\frac{\sqrt{C_1F(y,r)\mathcal{E}(u)}}{%
			1-2^{-\gamma_\ast/2}}\leq C^{\prime }_F\sqrt{F(x,r)\mathcal{E}(u)}.
	\end{eqnarray*}
	
	Finally, $(\mathrm{R}_\leq)$ follows directly by taking a sufficiently small 
	$\iota\in(0,1)$ such that 
	\begin{equation*}
		W(x,3C_Wd(x,y))<(2C_W)^{-1}R_W
	\end{equation*}
	for all $x,y$ with $d(x,y)<\iota\bar{R}$, since $s(r)<R_W$ in this case.
\end{proof}

Now we turn to lower bounds of $p_{t}(x,y)$, for which we will consider chains in $M$. For any $n\geq 1$ and for any $x,y\in M$, let 
\begin{equation*}
	L_{x,y}^{n}:=\left\{ \{x_{i}\}_{i=0}^{n}\subset M:x_{0}=x,x_{n}=y\right\}
\end{equation*}%
be the set of all chains from $x$ to $y$ with length $n$. We improve lower bounds from $(\mathrm{NLE})$ under assumption
$$\limsup_{n\rightarrow \infty }\left[ n\inf_{\{x_{i}\}_{i=1}^{n}\in L_{x,y}^{n}}\max_{0\leq i<n}W(x_{i},d(x_{i},x_{i+1}))\right] =0\quad\text{for all}\quad x,y\in M$$
on the metric structure of $(M,d)$. In other words, we assume for any $x,y\in M$ and $t>0$, there exists a positive integer $n_{0}:=n_{0}(t,x,y)$ such that for any $n\geq n_{0}$, there exists a chain $\{x_{i}\}_{i=1}^{n}\in L_{x,y}^{n}$ such that for any $0\leq i<n$, 
\begin{equation}
	W(x_{i},d(x_{i},x_{i+1}))<\left( \frac{\eta }{3C_{W}^{3}}\right) ^{\beta
		^{\ast }}\frac{t}{n}.  \label{n0}
\end{equation}%
Note that (\ref{n0}) implies with the help of (\ref{WinvScal}) that for each 
$0\leq i<n$, 
\begin{equation*}
	d(x_{i},x_{i+1})<\frac{\eta }{3C_{W}^{2}}W^{-1}\left( x_{i},\frac{t}{n}%
	\right) =:r_{i}.
\end{equation*}%
Then, for every $0\leq i<n$, $z_{i}\in B(x_{i},r_{i})$ and $z_{i+1}\in
B(x_{i+1},r_{i+1})$, we have 
\begin{eqnarray*}
	d(z_{i},z_{i+1}) &\leq
	&d(x_{i},x_{i+1})+d(x_{i},z_{i})+d(x_{i+1},z_{i+1})<2r_{i}+r_{i+1} \\
	&=&\frac{2\eta }{3C_{W}^{2}}W^{-1}\left( x_{i},\frac{t}{n}\right) +\frac{%
		\eta }{3C_{W}^{2}}W^{-1}\left( x_{i+1},\frac{t}{n}\right)\\
	&\leq&\left( \frac{%
		2\eta }{3C_{W}}+\frac{\eta }{3}\right) W^{-1}\left( z_{i},\frac{t}{n}\right)
	\leq \eta W^{-1}\left( z_{i},\frac{t}{n}\right) ,
\end{eqnarray*}%
where we use \cite[Proposition 9.1]{GrigoryanHuHu.2022.TPLE} in the last line.

Fix $z_{0}=x_{0}=x$ and $z_{n}=x_{n}=y$. By the semigroup identity and $(\mathrm{NLE})$, we obtain 
\begin{eqnarray}
	p_{t}(x,y) &=&\int_{M^{n-1}}p_{\frac{t}{n}}(x,z_{1})p_{\frac{t}{n}%
	}(z_{1},z_{2})\cdots p_{\frac{t}{n}}(z_{n-1},y)d\mu (z_{1})d\mu
	(z_{2})\cdots d\mu (z_{n-1})  \notag \\
	&\geq &\int_{\times _{i=1}^{n-1}B(x_{i},r_{i})}p_{\frac{t}{n}}(x,z_{1})p_{%
		\frac{t}{n}}(z_{1},z_{2})\cdots p_{\frac{t}{n}}(z_{n-1},y)d\mu (z_{1})d\mu
	(z_{2})\cdots d\mu (z_{n-1})  \notag \\
	&\geq &\int_{\times _{i=1}^{n-1}B(x_{i},r_{i})}\frac{cd\mu (z_{1})d\mu
		(z_{2})\cdots d\mu (z_{n-1})}{V(x,W^{-1}(x,t/n))}\prod_{i=1}^{n-1}\frac{c}{%
		V(x_{i},W^{-1}(x_{i},t/n))}  \notag \\
	&\geq &\frac{c^{n}}{V(x,W^{-1}(x,t))}\prod_{i=1}^{n-1}\frac{V(x_{i},r_{i})}{%
		V(x_{i},W^{-1}(x_{i},t/n))}  \notag \\
	&\geq &\frac{C^{-1}c^{n}}{V(x,W^{-1}(x,t))}\left( \frac{\eta }{3C_{W}^{2}}%
	\right) ^{\alpha ^{\ast }(n-1)}=\frac{c^{\prime }}{V(x,W^{-1}(x,t))}\exp
	(-c^{\prime \prime }n)  \label{glb}
\end{eqnarray}%
for all $x,y\in M$, $t\in (0,W(x,\bar{R}))$ and every integer $n\geq
n_{0}(t,x,y)$.

Apparently, (\ref{glb}) gives a relatively better lower estimate. However,
it is hard to verify (\ref{n0}), especially to obtain a precise
estimate on the $n_0$ above so as to get the lower estimate of $p_t(x,y)$ as
we desired. Meanwhile, with the help of the chain condition, a relatively
weak upper bound of $n_0$ will be obtained and then $(\mathrm{LE}_{\exp})$
is satisfied, which completes the proof of Proposition \ref{thm:main3}:

\begin{proof}[Proof of Proposition \ref{thm:main3}]
Based on Theorem \ref{thm:main2}, it suffices to prove the implication
\begin{equation*}
(\mathrm{NLE})+(\mathrm{CH})\Rightarrow(\mathrm{LE}_{\exp})
\end{equation*}
when $\beta_\ast>1$. For that purpose, it is enough to find $n_0=n_0(t,x,y)\geq 1$ ensuring (\ref{n0}).
	
For that purpose, let $x,y\in M$ be two different points, and fix $0<t<W(x,\bar{R})$. We have by $(\mathrm{CH})$ that, for each positive integer $n$,
there exists a chain $\{x_{i}\}_{i=0}^{n}$ satisfying 
\begin{equation*}
	x_{0}=x,\quad x_{n}=y\quad \text{and}\quad d(x_{i},x_{i+1})\leq \frac{C_{\mathrm{CH}}}{n}d(x,y)\quad\text{for all}\quad 0\leq i<n.
\end{equation*}%
This implies in particular that $d(x,x_{i})\leq C_{\mathrm{CH}}d(x,y)$ for
every $0\leq i<n$. Therefore, 
\begin{equation*}
	\frac{W(x,d(x,y))}{W(x_{i},d(x_{i},x_{i+1}))}\geq C_{W}^{-1}C_{\mathrm{CH}}^{-\beta ^{\ast }}\frac{W(x,C_{\mathrm{CH}}d(x,y))}{W(x_{i},d(x_{i},x_{i+1}))}\geq C_{W}^{-2}C_{\mathrm{CH}}^{-\beta ^{\ast
	}}n^{\beta _{\ast }}.
\end{equation*}
Consequently, under the assumption $\beta _{\ast }>1$, (\ref{n0}) is satisfied when 
\begin{equation*}
n\geq \left\lceil \left( \frac{C_{W}^{2}C_{\mathrm{CH}}^{\beta ^{\ast }}}{\left( \eta /3C_{W}^{3}\right) ^{\beta ^{\ast }}t}W(x,d(x,y))\right) ^{\frac{1}{\beta _{\ast }-1}}\right\rceil =:n_{0}.
\end{equation*}
Inserting this formula of $n_{0}$ into (\ref{glb}), then $(\mathrm{LE}_{\exp}^W)$ follows.
\end{proof}

\subsection{Proof of Theorem \ref{maint}}\label{pf4}
Throughout this subsection, the reference function $F$ is always a scaling function.

\begin{lemma}\label{RandQR}
	Let $(\mathcal{E},\mathcal{F})$ be a regular Dirichlet form on $L^2(M,\mu)$. 
	\begin{enumerate}
		\item Suppose that $(\mathrm{R}_{\leq})$ and $(\mathrm{Para})$ holds. Then $(\mathrm{MI})$ holds.
		\item Suppose that $(\mathrm{sloc})$, $(\mathrm{MI})$ and $(\mathrm{R}_{\geq })$ hold with a scaling function $F$. Then $(\mathrm{RB}_{\geq })$ holds.
	\end{enumerate}
\end{lemma}

\begin{proof}
We first show $(\mathrm{Para})+(\mathrm{R}_{\leq })\Rightarrow (\mathrm{MI})$ with some scaling function $F$.

Since $(\mathrm{Para})$ holds, same as in \cite[Proposition 6.6]{GrigoryanHuLau.2014.TAMS6397} we have 
\begin{equation}
|u(x)-u(y)|^{2}\leq \mathcal{R}(x,y)\mathcal{E}(u),  \label{res+}
\end{equation}
which implies $(\mathrm{MI})$ by estimating $\mathcal{R}(x,y)$ by $(\mathrm{R}_{\leq })$. Here we emphasize that no extra condition is needed.

Now we show that $(\mathrm{MI})+(\mathrm{R}_{\geq })+(\mathrm{sloc})\Rightarrow (\mathrm{RB}_{\geq }).$ Our proof here is motivated by \cite[Proof of Lemma 2.4]{BarlowCoulhonKumagai.2005.CPAM1642}.

Fix a ball $B:=B(x_{0},r)$ with $x_{0}\in M$, $r\in (0,\bar{R})$. For each $z\in B\setminus \frac{1}{2}B$, we can easily see that $\mathcal{R}(x_{0},z)>0 $ and there exists a function $\phi _{z}^{x_{0}}\in \mathcal{F}\cap C_{0}(M)$ such that 
\begin{equation*}
0\leq \phi _{z}^{x_{0}}\leq 1,\quad \phi _{z}^{x}(x)=1,\quad \phi_{z}^{x_{0}}(z)=0\quad \text{and}\quad \mathcal{R}(x_{0},z)^{-1}\leq\mathcal{E}(\phi _{z}^{x_{0}})\leq 2\mathcal{R}(x_{0},z)^{-1}.
\end{equation*}
For any $0<\eta <\iota /2$, it follows from $(\mathrm{MI}^F)$, $(\mathrm{R}_{\geq })$ and \eqref{WF} that, for all $z^{\prime }\in B(z,\eta r)$,
\begin{eqnarray*}
\phi _{z}^{x_{0}}(z^{\prime })^{2}&=&\left\vert \phi _{z}^{x_{0}}(z)-\phi_{z}^{x_{0}}(z^{\prime })\right\vert ^{2}\leq CF(z,d(x_{0},z))\mathcal{E}(\phi _{z}^{x_{0}})\\
&\leq&2C\frac{F(z,\eta r)}{\mathcal{R}(x_{0},z)}\leq \frac{2C^{2}F(z,\eta r)}{F(z,d(z,x_{0}))}\leq 2C^{3}(2\eta)^{\gamma _{\ast }}.
\end{eqnarray*}
Choose $\eta $ small enough such that $2C^{3}(2\eta)^{\gamma _{\ast }}\leq 1/4$. Hence for every $z\in B\setminus \frac{1}{2}B$, 
\begin{equation*}
\phi_{z}^{x_{0}}\leq \frac{1}{2}\quad\text{in}\quad B(z,\eta r).
\end{equation*}

Since $(\mathrm{VD})$ holds, there exist $N\geq 1$ independent of $B$ and a sequence $\{z_{i}\}_{i=1}^{N}\subset B\setminus \frac{1}{2}B$ such that $B\setminus \frac{1}{2}B\subset \tbigcup_{i=1}^{N}B(z_{i},\eta r).$ Fix $\phi\in \cutoff(\frac{1}{2}B,B)\cap C_{0}(M)$, and define 
\begin{equation*}
f:=\begin{cases}\phi _{z_{1}}^{x_{0}}\wedge \phi _{z_{2}}^{x_{0}}\wedge\cdots\wedge\phi _{z_{N}}^{x_{0}},&\text{if}\quad N\geq 1;\\
1_{\frac{1}{2}B},&\text{if}\quad N=0,
\end{cases}
\end{equation*}
so that $f(x_{0})=1$ and $f\leq \frac{1}{2}$ in $B\setminus \frac{1}{2}B$.
Consider the functions 
\begin{equation*}
	g:=(2f-1)_{+},\quad g_{1}:=g\phi \quad \text{and}\quad g_{2}=g\mathbf{1}%
	_{B^{c}}.
\end{equation*}%
Clearly $g(x_{0})=g_{1}(x_{0})=1$ and $0\leq g\leq 1$ in $M$. Note also that 
$g_{1}$ is supported in $B$ and $g_{1}\in \mathcal{F}$. Hence $%
g_{2}=g-g_{1}\in \mathcal{F}$.

Since the supports of $g_{1}$ and $g_{2}$ are disjoint, by $(\mathrm{sloc})$, we directly have $\mathcal{E}(g_{1},g_{2})=0$. Therefore, 
\begin{equation}
	\mathcal{E}(g_{1})=\mathcal{E}(g)-\mathcal{E}(g_{2})-2\mathcal{E}%
	(g_{1},g_{2})\leq \mathcal{E}(g)+2C^{\prime }F(x_{0},r)^{-1}.  \label{ed}
\end{equation}%
On the other hand, by \cite[Formula (6.14)]{GrigoryanHuLau.2014.TAMS6397}
and $(\mathrm{R}_{\geq })$, 
\begin{equation*}
	\mathcal{E}(g)\leq 4\mathcal{E}(f)\leq 4\sum_{i=1}^{N}\mathcal{E}(\phi
	_{z_{i}}^{x_{0}})\leq 8\sum_{i=1}^{N}\mathcal{R}(x_{0},z_{i})^{-1}\leq
	8\sum_{i=1}^{N}F(x_{0},d(x_{0},z_{i}))^{-1}\leq 8N_{1}CF(x_{0},r)^{-1}.
\end{equation*}%
Combining (\ref{ed}), and noting that $g_{1}$ is a test function for the
resistance $\mathcal{R}(x_{0},B^{c})$, we have 
\begin{equation*}
	\mathcal{R}(x_{0},B^{c})^{-1}\leq \mathcal{E}(g_{1})\leq \mathcal{E}%
	(g)+2C^{\prime }F(x_{0},r)^{-1}\leq (8N_{1}C+2C^{\prime })F(x_{0},r)^{-1},
\end{equation*}%
which completes the proof.
\end{proof}

To deduce survival estimates, we introduce Green functions. Let $\Omega \Subset M$ (i.e., the closure of $\Omega$ is a compact subset of $M$) be an open set satisfying
\begin{equation}\label{Ld}
\diam\overline{\Omega }<\iota \bar{R}\quad\text{and}\quad\Omega ^{c}\neq \emptyset.
\end{equation}
For any non-empty closed $A\subset \Omega $, set
\begin{equation*}
\mathcal{F}_{\Omega }^{A}:=\big\{u\in \mathcal{F}\cap C_{0}(M):\ \ u|_{A}\equiv 1\ \text{ and }\ u|_{\Omega ^{c}}\equiv 0\big\}.
\end{equation*}
 Noticing that $(\mathrm{MI})$ holds with a scaling function $F$, it follows by Lemma \ref{preHF} that, there exists a unique function $\phi _{\Omega }^{A}\in 
 \mathcal{F}_{\Omega }^{A}$ such that $0\leq \phi _{\Omega }^{A}\leq 1$ in $M$ and
\begin{equation}
	\mathcal{E}\left( \phi _{\Omega }^{A}\right) =\inf \left\{ \mathcal{E}%
	(u):u\in \mathcal{F}_{\Omega }^{A}\right\} =\mathcal{R}(A,\Omega ^{c})^{-1}.
	\label{vrxB}
\end{equation}

Now we show the definition of Green functions.
\begin{definition}[Green Functions]
	Let $\Omega \Subset M$ be an open set satisfying \eqref{Ld}. Then the Green function with respect to $\Omega $ is defined as 
	\begin{equation*}
		g^{\Omega }(x,y)=%
		\begin{cases}
			\phi _{\Omega }^{x}(y)\mathcal{R}(x,\Omega ^{c}), & x\in \Omega ; \\ 
			0, & x\notin \Omega ,%
		\end{cases}%
	\end{equation*}%
	where $\phi _{\Omega }^{x}(y)$ is the solution of the variational problem (%
	\ref{vrxB}) with $A=\{x\}$.
\end{definition}

\begin{lemma}\label{gBE}
Suppose that $(\mathrm{RB}_{\geq })$ and $(\mathrm{MI})$ hold with that $F$ is a scaling function. Then under condition $(\mathrm{VD}),(\mathrm{RVD})$, there exist two constants $C>0$, $0<\eta <\frac{1}{2}$ such that, for any ball $B:=B(x_{0},r)$ with $0<r<\frac{\iota \wedge 1}{\varepsilon _{0}^{-1}+1}\bar{R}$ (where $\varepsilon _{0}$ comes from (\ref{UP})), 
\begin{eqnarray}
g^{B}(x,y) &\leq &CF(x_{0},r)\quad \quad \text{for all }\ x,y\in B,\label{grup} \\
g^{B}(x,y) &\geq &C^{-1}F(x_{0},r)\quad \text{for all }\ x\in \frac{1}{2}B\ \text{ and }\ y\in B(x,\eta r).  \label{grlow}
\end{eqnarray}
\end{lemma}

\begin{proof}
	Note by (\ref{UP}) that there exists $z\in \varepsilon _{0}^{-1}B\setminus B$%
	, and $d(x,z)<(\varepsilon _{0}^{-1}+1)r<\iota \bar{R}$ for all $x\in B$.
	Using $(\mathrm{MI})$ and the fact that $0\leq \phi _{\Omega }^{x}\leq 1$,
	we prove (\ref{grup}) by 
	\begin{equation*}
		g^{B}(x,y)=\phi _{B}^{x}(y)\mathcal{R}(x,B^{c})\leq \mathcal{R}(x,z)\leq
		CF(x,(\varepsilon _{0}^{-1}+1)r)\leq C^{\prime }F(x_{0},r).
	\end{equation*}
	
	To prove (\ref{grlow}), observe that given any $\eta \in (0,1/2)$, by $(\mathrm{MI})$, we have for all $y\in B(x,\eta r)$, 
	\begin{equation*}
	\left\vert \phi _{B}^{x}(x)-\phi _{B}^{x}(y)\right\vert ^{2}\leq CF(x,d(x,y))\mathcal{E}(\phi _{B}^{x})=C\frac{F(x,d(x,y))}{\mathcal{R}(x,B^{c})}\leq C^{2}\frac{F(x,\eta r)}{F(x,r)}\leq C^{3}\eta ^{\gamma _{\ast }}.
	\end{equation*}
	Choosing $\eta $ with $C^{3}\eta ^{\gamma _{\ast }}\leq \frac{1}{4}$, it follows that 
	\begin{equation*}
		\phi _{B}^{x}(y)\geq \phi _{B}^{x}(x)-\frac{1}{2}=\frac{1}{2}
	\end{equation*}
	for all $y\in B(x,\eta r)$. On the other hand, by $(\mathrm{RB}_{\geq })$
	and the monotonicity of $\mathcal{R}(x,\cdot)$, 
	\begin{equation*}
		\mathcal{R}(x,B^{c})\geq \mathcal{R}(x,B(x,r/2)^{c})\geq C^{-1}F(x,r/2)\geq
		C^{-2}F(x,r).
	\end{equation*}%
	Therefore, 
	\begin{equation*}
		g^{B}(x,y)=\mathcal{R}(x,B^{c})\phi _{B}^{x}(y)\geq \frac{1}{2}%
		C^{-2}F(x,r)\geq cF(x_{0},r)
	\end{equation*}%
	for all $x\in \frac{1}{2}B$ and $y\in B(x,\eta r)$, which proves our
	claim.
\end{proof}

	\begin{corollary}
	\label{hboth} For any ball $B:=B(x_{0},r)$ with $0<r<\frac{\iota \wedge 1}{%
		\varepsilon _{0}^{-1}+1}\bar{R}$, take $f\in \mathcal{F}\cap C_{0}(B)$ such
	that $0\leq f\leq 1$ on $M$ and $f=1$ on $(\frac{1}{2}+\eta)B$. Let 
	\begin{equation}
		h(x)=\int_{B}g^{B}(x,y)f(y)d\mu (y)  \label{htemp}
	\end{equation}%
	for all $x\in M$. Suppose that all conditions in Lemma \ref{gBE} are still
	satisfied. Then there exists a constant $C>0$ such that 
	\begin{eqnarray*}
		h(x) &\leq &CV(x_{0},r)F(x_{0},r)\quad \quad \text{for all }\ x\in B, \\
		h(x) &\geq &C^{-1}V(x_{0},r)F(x_{0},r)\quad \text{for all }\ x\in \frac{1}{2}%
		B.
	\end{eqnarray*}
\end{corollary}

\begin{proof}
	On one hand, it follows by (\ref{grup}) and $0\leq f\leq 1$ on $B$ that, for
	all $x\in B$, 
	\begin{equation*}
		h(x)=\int_{B}g^{B}(x,y)f(y)d\mu (y)\leq CF(x_{0},r)\int_{B}f(y)d\mu (y)\leq
		CV(x_{0},r)F(x_{0},r).
	\end{equation*}%
	On the other hand, for all $x\in \frac{1}{2}B$ and all $y\in B(x,\eta
	r)\subset (\frac{1}{2}+\eta)B$, we have $f(y)\geq 1/2$ and $g^{B}(x,y)\geq
	C^{-1}F(x_{0},r)$. Combining this fact with (\ref{grlow}) and $(\mathrm{VD})$%
	, we have 
	\begin{equation*}
		h(x)\geq \int_{B(x,\eta r)}g^{B}(x,y)f(y)d\mu (y)\geq
		C^{-1}F(x_{0},r)V(x,\eta r)\geq C^{-1}V(x_{0},r)F(x_{0},r).
	\end{equation*}%
	The proof is then completed.
\end{proof}

\begin{lemma} \label{RandS}
Under assumption \eqref{WF}, $(\mathrm{VD})+(\mathrm{MI}^{F})+(\mathrm{RB}_{\geq}^{F})\Rightarrow(\mathrm{S}^{W}_+).$
\end{lemma}

\begin{proof}
The proof here is similar to \cite[Proof of Theorem 6.13]%
{GrigoryanHuLau.2014.TAMS6397}.

Fix $B:=B(x_{0},r)$ of radius $r\in (0,\bar{R})$. We distinguish two cases.

Case $(1)$: $0<r<\iota \bar{R}/2$. Same as in the proof of \cite[Formula (6.34)]{GrigoryanHuLau.2014.TAMS6397}, we see 
\begin{equation*}
	P_{t}^{B}1_{B}(x)\geq \frac{h(x)-t}{\Vert h\Vert _{\infty }}
\end{equation*}%
for all $t>0$ and $\mu $-a.e.\ $x\in B$, where $h$ is the function defined
in (\ref{htemp}). Hence by Corollary \ref{hboth}, 
\begin{equation*}
	\esup_{\frac{1}{2}B}P_{t}^{B}1_{B}\geq c-\frac{Ct}{V(x_{0},r)F(x_{0},r)}\geq
	c-\frac{C^{\prime }t}{W(x_{0},r)}
\end{equation*}%
with constants $c,C^{\prime }>0$ independent of $x$, $t$ and $B$, thus
proving $(\mathrm{S}_{+})$.

Case $(2)$: $\iota \bar{R}/2\leq r<\bar{R}$ (which is only possible when $%
\iota <2$ and $0<\bar{R}<\infty $). Using the conclusion in the former case
and basic properties of the heat semigroup, we obtain 
\begin{equation*}
	\esup_{\frac{\iota }{8}B}P_{t}^{B}1_{B}\geq \esup_{\frac{\iota }{8}B}P_{t}^{%
		\frac{\iota }{4}B}1_{\frac{\iota }{4}B}\geq c-\frac{Ct}{W(x_{0},\iota r/4)}%
	\geq c-\frac{C^{\prime \prime }t}{W(x_{0},r)}
\end{equation*}%
with constants $c,C^{\prime \prime }>0$ independent of $x$, $t$ and $B$,
which proves $(\mathrm{S}_{+})$ again.
\end{proof}

To show relations between resistance and heat kernel estimates, we begin with the following simple lemma:

\begin{lemma}\label{1inF}
Let $(\mathcal{E},\mathcal{F})$ be a regular Dirichlet form on $L^{2}(M,\mu)$. If $\bar{R}<\infty$, then $1\in\mathcal{F}$. In addition, if $(\mathrm{sloc})$ holds, then $\mathcal{E}(1)=0$.
\end{lemma}

\begin{proof}
Same as in the proof of \cite[Proposition A.1]{HuLiu}, we see $1\in\mathcal{F}$. Further, $\mathcal{E}(1)=0$ by strong locality easily.
\end{proof}

\begin{lemma}\label{R-low}
Under assumption \eqref{WF}, $(\mathrm{VD})+(\mathrm{NLE})+(\mathrm{UE}_{\exp})\Rightarrow(\mathrm{Para})+(\mathrm{R}_{\geq}).$
\end{lemma}

\begin{proof}
We show first that \eqref{WF} and $(\mathrm{NLE})$ imply $(\mathrm{Para})$.

Actually, if $\bar{R}<\infty$, then $(\mathrm{Para})$ follows Lemma \ref{1inF} directly. Meanwhile, if $\bar{R}=\infty$, let $x,y\in M$ be two different points in $M$. If $d(x,y)<\eta W^{-1}(x,t)$, then 
\begin{equation*}
	t>W\left( x,\eta ^{-1}d(x,y)\right) =:a_{xy}.
\end{equation*}%
Note by \eqref{WF} that
\begin{equation*}
	\frac{t}{V(x,W^{-1}(x,t))}=\frac{W(x,W^{-1}(x,t))}{V(x,W^{-1}(x,t))}\geq
	C^{-1}F(x,W^{-1}(x,t)).
\end{equation*}
Combining $(\mathrm{NLE}^W)$, it follows that
\begin{eqnarray*}
	\int_{0}^{\infty }p_{t}(x,y)dt &\geq &\int_{a_{xy}}^{\infty }\frac{C^{-1}dt}{V(x,W^{-1}(x,t))}\geq \int_{a_{xy}}^{\infty }\frac{C^{-2}dt}{t}F(x,W^{-1}(x,t)) \\
	&\geq &C^{-2}F\left( x,\eta ^{-1}d(x,y)\right) \int_{a_{xy}}^{\infty }\frac{dt}{t}=\infty .
\end{eqnarray*}
By \cite[Criterion (1.6.2)]{FukushimaOshimaTakeda.2011.489}, $(\mathcal{E},\mathcal{F})$ is recurrent. Then $(\mathrm{Para})$ follows by the 14 lines after \cite[Formula (6.47)]{GrigoryanHuLau.2014.TAMS6397}.

In particular, we
see \eqref{res+} again for all $u\in \mathcal{F}\cap C_{0}(M)$, so that $(\mathrm{MI})$ (and hence $(\mathrm{R}_\le)$) holds naturally. Thus by Remark \ref{ctn-resist}, every $v\in \mathcal{F}$ also admits a continuous version such that (\ref{res+}) holds. In particular, for all $x\in M$ and $t>0$, with the help
of (\ref{pt-sp}) and the semigroup property of $p_{t}$, we obtain 
\begin{equation*}
	|p_{t}(x,x)-p_{t}(x,y)|^{2}\leq C\mathcal{R}(x,y)\mathcal{E}(p_{t}(x,\cdot
	))\leq \frac{C}{et}\mathcal{R}(x,y)p_{t}(x,x).
\end{equation*}%
Then, it follows from (\ref{res+}), $(\mathrm{NLE})$ and $(\mathrm{UE}_{\exp})$
that, there exists a constant $C>0$ such that, for any $x,y\in M$ and for
any $0<t<W(x,\bar{R})\wedge W(y,\bar{R})$, 
\begin{eqnarray}
	\sqrt{\mathcal{R}(x,y)} &\geq &\frac{|p_{t}(x,x)-p_{t}(x,y)|}{\sqrt{\mathcal{%
				E}(p_{t}(x,\cdot))}}\geq \frac{p_{t}(x,x)-p_{t}(x,y)}{\frac{1}{\sqrt{et}}%
		\sqrt{p_{t}(x,x)}}  \notag \\
	&\geq &\left( \sqrt{\frac{C}{tV(x,W^{-1}(x,t))}}\right) ^{-1}\left\{ \frac{%
		C^{-1}}{V(x,W^{-1}(x,t))}-\frac{C}{V(x,W^{-1}(x,t))}\Phi \left( \frac{%
		W(x,d(x,y))}{t}\right) \right\}   \notag \\
	&\geq &\sqrt{\frac{C^{-3}t}{V(x,W^{-1}(x,t))}}\left( 1-C^{2}\Phi \left( 
	\frac{W(x,d(x,y))}{t}\right) \right),\label{Rgeq}
\end{eqnarray}
where $\Phi(s)=\exp(-cs^{\frac{1}{\beta^\ast-1}})$. Noticing that $\Phi (s)\rightarrow 0$ as $s\rightarrow \infty $, we can find $\varepsilon \in (0,1)$ such that $\Phi (\varepsilon ^{-1})<(2C^{2})^{-1}$. Then $(\mathrm{R}_{\geq })$ follows from (\ref{Rgeq}) by letting $t:=\varepsilon W(x,d(x,y))$.
\end{proof}

\section{Resistance Forms}\label{RFF}

In this section, we introduce resistance forms, and then
show the relation between resistance forms and Dirichlet forms.

\begin{definition}[Resistance Form, Definition 2.3.1 in \protect\cite{Kigami.2001.226}]
	\label{RF} Let $X$ be any set, and we say $(\mathcal{E},\mathcal{F})$ is a
	resistance form on $X$, if all the following conditions are satisfied:
	
	i) $\mathcal{F}$ is a linear subspace of $l(X)$ (all real-valued functions
	on $X$) containing all constant functions, and $\mathcal{E}$ is a
	non-negative symmetric quadratic form on $\mathcal{F}$. Further, $\mathcal{E}%
	(u):=\mathcal{E}(u,u)=0$ if and only if $u$ is a constant on $X$.
	
	ii) Let $\sim$ be the equivalent relation on $\mathcal{F}$ such that for any 
	$u,v\in\mathcal{F}$, $u\sim v$ if and only if $u-v$ is a constant. Then $(%
	\mathcal{\mathcal{F}}/\sim,\mathcal{E})$ is a Hilbert space.
	
	iii) For any finite subset $V\subset X$ and $v\in l(V)$, there exists $u\in%
	\mathcal{F}$ such that $u|_V=v$.
	
	iv) For any $p,q\in X$, there is 
	\begin{equation}  \label{res}
		\sup\left\{\frac{|u(p)-u(q)|^2}{\mathcal{E}(u)}:\ \ u\in\mathcal{F},%
		\mathcal{E}(u)>0\right\}<\infty.
	\end{equation}
	
	v) The Markov property holds.
\end{definition}

A resistance form $(\mathcal{E},\mathcal{F})$ on $X$ is said to be regular (cf.\ \cite[Definition 6.2 of Part 1]{Kigami.2012.MAMS132}), if $\mathcal{F}\cap C_0(X)$ is dense in $C_0(X)$ under $\|\cdot\|_\infty$. Note by \cite[Theorem 2.3.4]{Kigami.2001.226} that the left side of (\ref{res}) is equal
to the effective resistance $\mathcal{R}(x,y)$.

Now, given a metric measure space $(M,d,\mu)$, we show the relation between resistance forms on $M$ and Dirichlet forms on $L^2(M,\mu)$. Indeed, Kigami has given a sufficient condition for a resistance form to be also a Dirichlet form:

\begin{lemma}
	\cite[Theorem 9.4 in Part 1]{Kigami.2012.MAMS132}\label{R2DF} Let $(\mathcal{%
		E},\mathcal{F})$ be a regular resistance form on a space $M$, and $\mathcal{R%
	}$ be its resistance metric. Let $\mu$ be a Radon measure on $M$ such that $%
	0<\mu(B)<\infty$ for every metric ball $B$. If $(M,\mathcal{R})$ is
	separable, then $(\mathcal{E},\mathcal{D})$ is a regular Dirichlet form on $%
	L^2(M,\mu)$, where $\mathcal{D}$ is the closure of $\mathcal{F}\cap C_0(X)$
	under $\mathcal{E}_1$.
\end{lemma}

Now we consider conditions for a Dirichlet form to be also a resistance form. The following lemma plays the key role:

\begin{lemma}
	\label{preHF} Let $(\mathcal{E},\mathcal{F})$ be a regular Dirichlet form on 
	$L^2(M,\mu)$. Then 
	\begin{equation*}
		\mathcal{F}_V^g:=\big\{u\in\mathcal{F}\cap C(M), u|_V\equiv g\big\}%
		\ne\emptyset
	\end{equation*}
	for any finite set $V\subset M$ and any $g\in l(V)$. Moreover, if $M$ is
	compact and $(\mathrm{MI})$ holds, then there exists a unique $v\in\mathcal{F%
	}_V^g$ such that 
	\begin{equation}  \label{HarExt}
		\mathcal{E}(v)=\inf\left\{\mathcal{E}(u):u\in\mathcal{F}_V^g\right\}.
	\end{equation}
\end{lemma}

This can be proved by a standard argument in variational calculus, which is omitted here.

\begin{proposition}
	\label{D2RF} Let $(M,d)$ be a compact metric space and $\mu $ be a doubling
	Radon measure on $M$. Let $(\mathcal{E},\mathcal{F})$ be a regular Dirichlet
	form on $L^{2}(M,\mu)$. Suppose that $(\mathcal{E},\mathcal{F})$ satisfies $%
	(\mathrm{MI})$ with some $\iota >1$, then $(\mathcal{E},\mathcal{F}_{0})$ is
	a resistance form, where $\mathcal{F}_{0}=\mathcal{F}\cap C(M)$.
\end{proposition}

\begin{proof}
Based on Lemma \ref{preHF}, properties i)-iii) and v) are clear; while iv) directly inherits $(\mathrm{MI})$. The proof is completed.
\end{proof}

To the end, we introduce a slightly generalized version of \cite[Part (D3) for Theorem 4.11]{Liu.2024.MAMA}, showing strong locality of self-similar forms:

\begin{proposition}\label{ssploc}
Let $(\mathcal{E},\mathcal{F})$ be a regular Dirichlet form 	on $L^2(K,\mu)$ (where $(K,\{F_i\}_{i=1}^N)$ is a self-similar set) satisfying the self-similarity as follows: for every $u\in\mathcal{F}$, $u\circ F_i\in\mathcal{F}$ for any $1\leq i\leq N$, and there exists a positive constant $s_i$ independent of $u$ for every $1\leq i\leq N$ such that 
\begin{equation}\label{Fsimi}
\mathcal{E}(u)=\sum_{i=1}^Ns_i^{-1}\mathcal{E}(u\circ F_i).
\end{equation}
If $s_i\in(0,1)$ for each $1\leq i\leq N$, then $(\mathcal{E},\mathcal{F})$ is strongly local.
\end{proposition}

\begin{proof}
Recall by Lemma \ref{1inF} that $1\in\mathcal{F}$ and $\mathcal{E}(1)=0$.

To show the strong locality of $(\mathcal{E},\mathcal{F})$, take $f,g\in%
\mathcal{F}$ with compact supports such that $g$ is equal to a constant $c$
in an open neighborhood $V_g$ of $\supp(f)$. Since both sets $\supp(f)$ and $%
K\setminus V_g$ are compact, the distance $2\delta$ between them is
positive, and hence 
\begin{equation*}
	\supp(f)\subset B(\supp(f),\delta)\subset V_g.
\end{equation*}

For any $k\geq 1$, define 
\begin{equation*}
	\Gamma_{f,k}:=\left\{w\in I^k:K_w\cap\supp(f)\neq\emptyset\right\}.
\end{equation*}
For any $x\in\cup_{w\in\Gamma_{f,k}}K_w$, there exists a word $%
w\in\Gamma_{f,k}$ such that $x\in K_w$. As $K_w\cap\supp(f)\neq\emptyset$,
there exists a point $z\in K_w\cap\supp(f)$. Writing $x=F_w(x^{\prime })$, $%
z=F_w(z^{\prime })$ for some points $x^{\prime },z^{\prime }\in K$, we know
that 
\begin{equation*}
	d(x,z)=d(F_w(x^{\prime }),F_w(z^{\prime }))\leq \rho_wd(x^{\prime
	},z^{\prime })\leq\rho_{\max}^k\bar{R}<2\rho_{\max}^k\bar{R},
\end{equation*}
from which follows $x\in B(\supp(f),2\rho_{\max}^k\bar{R})$. Thus 
\begin{equation}  \label{kw}
	\tbigcup_{w\in\Gamma_{f,k}}K_w\subset B\left(\supp(f),2\rho_{\max}^k\bar{R}%
	\right)\subset B(\supp(f),\delta/2),
\end{equation}
provided that $k$ is so large that $\rho_{\max}^k\bar{R}<\delta/4$.

For any $w\in \Gamma _{f,k}$, we know by using (\ref{kw}) that $g\circ
F_{w}\equiv c$ on $K$ and so 
\begin{equation*}
	\mathcal{E}(f\circ F_{w},g\circ F_{w})=\mathcal{E}(f\circ F_{w},c)=c\mathcal{%
		E}(f\circ F_{w},1)=0.
\end{equation*}%
On the other hand, for any $w\in I^{k}\setminus \Gamma _{f,k}$, we know that 
$f\circ F_{w}\equiv 0$ on $K$, and so 
\begin{equation*}
	\mathcal{E}(f\circ F_{w},g\circ F_{w})=\mathcal{E}(0,g\circ F_{w})=0.
\end{equation*}%
Therefore, it follows by (\ref{Fsimi}) that 
\begin{equation*}
	\mathcal{E}(f,g)=\sum_{w\in I^{k}}\frac{1}{s_{w}}\mathcal{E}(f\circ
	F_{w},g\circ F_{w})=\left\{ \sum_{w\in I^{k}\setminus \Gamma
		_{f,k}}+\sum_{w\in \Gamma _{f,k}}\right\} \frac{1}{s_{w}}\mathcal{E}(f\circ
	F_{w},g\circ F_{w})=0,
\end{equation*}%
thus showing the strong locality of $(\mathcal{E},\mathcal{F})$. The proof is completed.
\end{proof}

\begin{acknowledgement}
	The authors would like to thank Professor Jiaxin Hu for his continuous guidance and support throughout this research. We are also grateful to Professor Alexander Grigor'yan for his valuable discussions and support.
\end{acknowledgement}

\bibliographystyle{siam}

\end{document}